\documentclass[a4paper]{article}
\usepackage{amsthm,amsmath,amssymb,bbm}

\usepackage{geometry}
\newtheorem{assumption}{Assumption}
\newtheorem{theorem}{Theorem}
\newtheorem{example}{Example}
\newtheorem{lemma}{Lemma}
\newtheorem{remark}{Remark}
\newtheorem{prop}{Proposition}

\newtheorem{cor}{Corollary}
\newtheorem{definition}{Definition}

\def\lfhook#1{\setbox0=\hbox{#1}{\ooalign{\hidewidth
    \lower1.5ex\hbox{'}\hidewidth\crcr\unhbox0}}}

\newcommand{\R}{\mathbb{R}}
\newcommand{\N}{\mathbb{N}}
\newcommand{\E}{\mathbb{E}}
\renewcommand{\P}{\mathbb{P}}

\newcommand{\citationand}{\&}

\title{A mild It\^{o} formula
for SPDEs}

\author{Giuseppe 
Da Prato$^1$,
Arnulf Jentzen$^2$ 
and 
Michael R\"{o}ckner$^{3}$
\bigskip
\\
\small{$^1$Scuola Normale Superiore di Pisa, 
56126 Pisa, Italy,
e-mail: 
g.daprato@sns.it}
\smallskip
\\
\small{$^2$Program in Applied 
and Computational Mathematics, 
Princeton University,}
\\
\small{Princeton, 
NJ 08544-1000, 
USA, 
e-mail:
ajentzen@math.princeton.edu}
\smallskip
\\
\small{$^3$Faculty 
of Mathematics, 
Bielefeld University,
33501 Bielefeld, Germany,} 
\\
\small{e-mail: 
roeckner@math.uni-bielefeld.de;
Department
of Mathematics and} 
\\
\small{Statistics,
Purdue University,
West Lafayette, 
IN 47907-2067, USA,}
\\
\small{e-mail: 
roeckner@math.purdue.edu}}

\begin{document}

\maketitle

\begin{abstract}
This article introduces 
a certain class of stochastic 
processes, which we suggest
to call mild It\^{o} processes, 
and a new
- somehow mild - 
It\^{o} type formula 
for such processes.
Examples of mild
It\^{o} processes are
mild solutions 
of stochastic
partial differential equations
(SPDEs) and their numerical 
approximation
processes. 
\end{abstract}

\tableofcontents

%\section{Introduction}
%
%The abstract general theory
%is presented in 
%Section~\ref{sec:mildcalc}.
%It is then illustrated for
%stochastic partial differential
%equations in Section~\ref{sec:appl}.

\section{Introduction}

The following setting
is considered in this introductory
section.
Let 
$
  ( H, \left\| \cdot \right\|_H,
    \left< \cdot, \cdot \right>_H
  )
$,
$
  ( U, \left\| \cdot \right\|_U,
    \left< \cdot, \cdot \right>_U
  )
$
and
$
  ( V, \left\| \cdot \right\|_V,
    \left< \cdot, \cdot \right>_V
  )
$
be separable real
Hilbert spaces,
let
$ 
  \left( \Omega, \mathcal{F}, \P \right) 
$
be a probability space
with a normal filtration
$
  ( \mathcal{F}_t )_{ t \in [0,\infty) } 
$,
let
$
  ( W_t )_{ t \in [0,\infty) }
$
be a cylindrical
standard
$ ( \mathcal{F}_t )_{ t \in [0,\infty) } 
$-Wiener process on $ U $,
let
$ A \colon D(A) \subset H \to H $
be a generator of an analytic
semigroup and let
$
  \alpha \in 
  ( - 1, 0]
$,
$ 
  \beta \in 
  ( - \frac{ 1 }{ 2 }, 0]
$,
$
  \eta \in [0,\infty)
$
be real numbers
such that
$
  \eta - A 
$
is bijective and
positive.
Moreover,
to simplify the notation
define
$
  \| v \|_{ H_r }
:=
  \| ( \eta - A )^r v \|_H
$
for all
$ v \in H_r := D( ( \eta - A )^r ) $ 
and all $ r \in \R $
and
let
$
  F \colon H \to H_{ \alpha }
$
and
$
  B \colon H \to 
  HS( U, H_{ \beta } )
$
be globally Lipschitz
continuous functions
and let
$
  X \colon [0,\infty) 
  \times \Omega \to H
$
be an adapted stochastic
process with continuous sample
paths satisfying
$
  \sup_{ s \in [0,t] }
  \E[
    \| X_s \|_H^p
  ]
  < \infty
$
and
\begin{equation}
\label{eq:SPDE.intro}
  X_t
=
  e^{ A t } X_0
+
  \int_0^t
  e^{ A (t - s) }
  F( X_s ) \, ds
+
  \int_0^t
  e^{ A (t - s) }
  B( X_s ) \, dW_s
\end{equation}
$ \P $-a.s.\ for 
all $ t, p \in [0,\infty) $.
The stochastic process 
$
  X \colon [0,\infty) \times 
  \Omega \to H
$
is thus a {\it mild solution}
of the stochastic
partial differential equation (SPDE) \eqref{eq:SPDE.intro}
(SPDEs have been extremely 
intensively studied in the last decades;
see, e.g., 
the books 
\cite{r90,dz92,GreckschTudor1996,DaPratoZabczyk1996,cr99,DaPrato2004,SanzSole2005,c07,PeszatZabczyk2007,Kotelenez2008,Holdenetal2009}
and lecture notes \cite{Walsh1986,KallianpurXiong1995,KrylovRoecknerZabczyk1999,PrevotRoeckner2007,AlbeverioFlandoliSinai2008,Dalangetal2009,Hairer2009}
and the references therein).
A simple example of this framework
is the following setting:
If 
$ H = U = L^2( (0,1), \R ) $
is the Hilbert space of equivalence
classes of Lebesgue square integrable
functions, 
if $ A \colon D(A) \subset H \to H $
is the Laplacian with Dirichlet
boundary conditions on $ (0,1) $
and if 
$ 
  \big( F( v ) \big)(x)
  = f( x, v(x) )
$
and
$ 
  \big( B( v ) u \big)(x)
  = f( x, v(x) ) \cdot u(x)
$
for all $ x \in (0,1) $,
$ u, v \in H $ where
$
  f, b \colon (0,1) \times \R \to \R
$
are continuously differentiable
functions with globally bounded
derivatives, then the above
framework is fulfilled with 
$ \alpha = \eta = 0 $
and
$ 
  \beta \in 
  (- \frac{ 1 }{ 2 }, - \frac{ 1 }{ 4 } 
  )
$
and \eqref{eq:SPDE.intro}
reduces to
the SPDE
\begin{equation}
\label{eq:SPDE.intro.ex}
  d X_t(x)
=
  \left[
    \tfrac{ \partial^2 }{
      \partial x^2
    }
    X_t(x)
    + 
    f(x, X_t(x) )
  \right] 
  dt
  +
  b(x, X_t(x) ) \, dW_t(x)
\end{equation}
with
$
  X_t(0) = X_t(1) = 0
$
for 
$ x \in (0,1) $ and
$ t \in [0,\infty) $.
Further examples of the
above described framework 
and existence and uniqueness
results for \eqref{eq:SPDE.intro}
can, 
e.g., be found
in 
Da Prato \& Zabczyk~\cite{dz92,DaPratoZabczyk1996},
Brze\'{z}niak~\cite{b97b}
(see Theorem~4.3 in \cite{b97b}),
Van Neerven, Veraar \& 
Weis~\cite{vvw08}
(see Theorem~6.2 and
Section~10 in \cite{vvw08})
and in the references
therein.

%Moreover, observe that
%if 
%$ 
%  \psi \colon (0,1) \times 
%  \R \to \R 
%$
%is a twice continuously differentiable
%function with globally bounded
%derivatives and if
%$ V = \R $,
%then
%$ 
%  \varphi \colon H \to \R 
%$
%satisfying
%$
%  \varphi(v) =
%  \int_0^1 \psi( v(x) ) \, dx
%$
%for all $ v \in H $
%is a twice continuously
%Fr\'{e}chet differentiable 
%mapping.

Our aim is to derive an It\^{o}
type formula for the
solution process $ X $
of the SPDE~\eqref{eq:SPDE.intro}.
Let us briefly review some
related It\^{o} type formula results 
from the literature.
First, note that if
$ \alpha = \beta = 0 $ 
and if the mild solution process
$ X $
of the SPDE~\eqref{eq:SPDE.intro}
is also a $ D(A) $-valued strong
solution of the 
SPDE~\eqref{eq:SPDE.intro}, 
then
the standard It\^{o} formula
(see It\^{o}~\cite{Ito1951})
in infinite dimensions
can be applied to $ X $.
More precisely, in that case,
Theorem~2.4 in Brze\'{z}niak,
Van Neerven, Veraar 
\citationand\ 
Weis~\cite{bvvw08}
implies
\begin{equation}
\label{eq:standardito.intro}
\begin{split}
  \varphi( X_t )
& =
  \varphi( X_{ t_0 } )
+
  \int_{ t_0 }^t
  \varphi'( X_s ) 
  \left[
    A X_s + F(X_s)
  \right]
  ds
+
  \int_{ t_0 }^t
  \varphi'( X_s ) 
  B( X_s ) \,
  dW_s
\\ & \quad +
  \frac{ 1 }{ 2 }
  \sum_{ j \in \mathcal{J} }
  \int_{ t_0 }^t
  \varphi''( X_s ) 
  \big(
    B( X_s ) g_j , 
    B( X_s ) g_j
  \big)
  \,
  ds
\end{split}
\end{equation}
$ \P $-a.s.\ for all
$ t_0, t \in [0,\infty) $
with $ t_0 \leq t $
and all twice continuously
Fr\'{e}chet differentiable
functions
$ \varphi \in C^2( H, V ) $
where
$
  \mathcal{J}
$
is a set and
$
  ( g_j )_{ j \in \mathcal{J} } 
  \subset U
$
is an arbitrary orthonormal basis
of $ U $.
The case where
$ X $ is not $ D(A) $-valued
and thus not 
a strong solution 
of \eqref{eq:SPDE.intro}
is more subtle.
There are a few
results in the literature
in this direction.
First, 
in the case
$ \alpha \geq - \frac{ 1 }{ 2 } $
and
$ \beta = 0 $
(i.e., $ B $ maps
from $ H $ to $ HS( U, H ) $),
in the case where
$ 
  A \colon D(A) \subset H
  \to H 
$
is self-adjoined
and in the case of the special
test function
$
  \varphi(v) =
  \| v \|^2_H
$
for all $ v \in H $,
\eqref{eq:standardito.intro}
can be 
generalized and 
then reads as
\begin{equation}
\label{eq:itosquare.intro}
\begin{split}
  \left\| 
    X_t 
  \right\|^2_H
& =
  \left\| 
    X_{ t_0 }
  \right\|^2_H
+
  2
  \int_{ t_0 }^t
  \left< X_s, 
    A X_s 
  \right>_H
  ds
+
  2
  \int_{ t_0 }^t
  \left< X_s, 
    F(X_s)
  \right>_H
  ds
+
  2
  \int_{ t_0 }^t
  \left< X_s, 
    B( X_s ) \,
    dW_s
  \right>_H
\\ & \quad +
  \int_{ t_0 }^t
  \left\|
    B( X_s ) 
  \right\|^2_{ HS( U, H ) }
  ds
\end{split}
\end{equation}
$ \P $-a.s.\ for all
$ t_0, t \in [0,\infty) $
with $ t_0 \leq t $;
see 
Pardoux's
pioneering work \cite{p72,Pardoux1975,Pardoux1975Thesis}
and see, e.g., also
\cite{kr79,gk81,gk8182,p87,op89,rrw07,PrevotRoeckner2007}
for generalizations and reviews of
this It\^{o} formula for the squared
norm (the above mentioned results 
from the literature
consider slightly different frameworks
and, in particular, often allow 
$ A $ to be nonlinear too).
Note that in that case
$ X $ enjoys values in
$
  H_{ 1 / 2 }
  =
  D\big( 
    ( \eta - A )^{ 1 / 2 } 
  \big)
$
(see Theorem~4.2 in 
Kruse \citationand\ 
Larsson~\cite{KruseLarsson2011})
and therefore,
the integral
$
  \int_0^t
  \left< X_s, 
    A X_s 
  \right>_H
  ds
  :=
  \eta
  \int_0^t
  \left\| X_s \right\|_H^2 
  ds 
  -
  \int_0^t
  \| 
    (\eta - A)^{ \frac{ 1 }{ 2 } } X_s 
  \|_H^2
  \, ds 
$
in \eqref{eq:itosquare.intro}
is well defined.
Formula~\eqref{eq:itosquare.intro}
is a crucial ingredient in
the {\it variational approach}
for SPDEs
(see the monographs
\cite{Pardoux1975Thesis,kr79,r90,PrevotRoeckner2007}).
Formula~\eqref{eq:itosquare.intro}
is an It\^{o} formula for
possibly non-strong
solutions of SPDEs
in the case of the special
test function
$
  \left\| \cdot \right\|_H^2
$.
There are also a few
results in the literature
which establish It\^{o}
type formulas for possible
non-strong solutions
of SPDEs
for more general test functions
than the squared norm
$
  \left\| \cdot \right\|_H^2
$;
see \cite{p87,z06,gnt05,l07,lt08,l09b}.
In Theorem~5.1
in Pardoux~\cite{p87},
formula~\eqref{eq:itosquare.intro}
is generalized to a special
class of test functions which have
similar topological properties
as the function
$
  \left\| \cdot \right\|_H^2
$.
In Zambotti~\cite{z06}, the standard
It\^{o} formula is applied to 
regularized versions of the solution
process of the stochastic
heat equation with
additive noise and then the
limit of these regularized
It\^{o} formulas is made sense through 
a suitable renormalization term
that appears in the resulting formula.
In Gradinaru, Nourdin
\citationand\ Tindel~\cite{gnt05},
Malliavin calculus and a 
Skorokhod integral is used 
to prove an It\^{o} type formula
for the solution of the 
stochastic heat equation with 
additive noise
(see also 
Leon \citationand\ Tindel~\cite{lt08}
for a related It\^{o} formula 
result for the
stochastic heat equation with additive fractional
noise).
In Lanconelli~\cite{l07},
a Wick product is used to
formulate an It\^{o} type formula
for the solution process
of the stochastic heat equation
with additive noise
and the relation between the
formulas in \cite{z06,gnt05}
is analyzed
(see Section~3 in \cite{l07}
and see also 
Lanconelli~\cite{l09b}
for some consequences
of this Wick produckt It\^{o}
type formula for the stochastic
heat equation with additive noise).

In general it is 
not
clear how and whether 
\eqref{eq:standardito.intro}
can be generalized
to the case where
$ 
  X \colon [0,\infty) \times \Omega \to H 
$
is not a $ D(A) $-valued
strong solution of
\eqref{eq:SPDE.intro}.
This article suggests a different
approach for deriving an It\^{o}
formula for solutions of 
\eqref{eq:SPDE.intro}.
We
do not aim for a suitable
generalization of 
\eqref{eq:standardito.intro}
to the case of 
non-strong solutions but
instead we suggest
a somehow different It\^{o}
type formula for 
\eqref{eq:SPDE.intro}
which naturally holds
for \eqref{eq:SPDE.intro}
in its full generality
for all smooth
test functions.
More precisely, 
we establish
in Corollary~\ref{cor:ito}
in Subsection~\ref{sec:solutionSPDE}
below
the identity
\begin{equation}
\label{eq:mildito_intro}
\begin{split}
  \varphi( X_t )
 & =
  \varphi( e^{ A( t - t_0 ) } X_{ t_0 } )
  +
  \int_{ t_0 }^t
  \varphi'( e^{ A( t - s ) } X_s ) \,
  e^{ A( t - s ) }
  F( X_s ) \, ds
  +
  \int_{ t_0 }^t
  \varphi'( e^{ A( t - s ) } X_s ) \,
  e^{ A( t - s ) }
  B( X_s ) \, dW_s
\\ & \quad+
  \frac{1}{2}
  \sum_{ j \in \mathcal{J} }
  \int_{ t_0 }^t
  \varphi''( e^{ A( t - s ) } X_s )
  \left(
    e^{ A( t - s ) }
    B( X_s ) g_j,
    e^{ A( t - s ) }
    B( X_s ) g_j
  \right) ds
\end{split}
\end{equation}
$ \mathbb{P} $-a.s.\ for 
all $ t_0, t \in [0,\infty) $
with $ t_0 \leq t $
and all
$ 
  \varphi \in 
  \cup_{ 
    r < 
    \min( \alpha + 1, \beta + 1 / 2 ) 
  }
  C^2( H_r, V ) 
  \supset
  C^2( H, V )
$.
Corollary~\ref{cor:ito}
also ensures that all terms
in \eqref{eq:mildito_intro}
are well defined 
(see 
\eqref{eq:well1d}--\eqref{eq:well3d}
in Subsection~\ref{sec:solutionSPDE}).
In the case of
\eqref{eq:SPDE.intro.ex},
natural examples for the
test functions 
$ 
  \varphi \in 
  \cup_{ 
    r < 
    \min( \alpha + 1, \beta + 1 / 2 ) 
  }
  C^2( H_r, V ) 
$
in \eqref{eq:mildito_intro}
are Nemytskii operators
and nonlinear 
integral operators
such as 
$
  H_r \ni
  v \mapsto
  \int_0^1 \psi( x, v(x) ) \, dx
  \in \R
$
for any 
$ 
  0 < r < 
  \min( \alpha + 1 , \beta + 1 / 2 ) 
$
and any smooth function
$ \psi \colon (0,1) \times \R \to \R $
with globally bounded
derivatives.
In the special case 
$ 
  \varphi = id_H \colon
  H \ni v \mapsto v
  \in H = V
$,
%$ V = H $ and 
%$ \varphi(v) = v $ 
%for all $ v \in H $,
equation~\eqref{eq:mildito_intro}
reduces to 
the variant of constants
formula~\eqref{eq:SPDE.intro}
and in that sense,
\eqref{eq:mildito_intro}
is somehow a {\it mild It\^{o}  
formula}.
In the deterministic case
$ B \equiv 0 $,
equation~\eqref{eq:mildito_intro}
is somehow a 
{\it mild chain rule};
see Example~\ref{ex:2}
in Section~\ref{sec:mildito} below
for more details.
The identity~\eqref{eq:mildito_intro}
can be generalized to a much larger
class of stochastic processes
than solution processes of 
the SPDE~\eqref{eq:SPDE.intro}.
To be more precise,
in Definition~\ref{propdef}
in Subsection~\ref{sec:mildprocesses}
a class of stochastic processes
which exhibit a similar algebraic
structure as \eqref{eq:SPDE.intro}
is introduced and referred
as 
{\it mild It\^{o} processes}.
Examples of mild It\^{o} processes
are solution processes of
SPDEs such as \eqref{eq:SPDE.intro}
(see Subsection~\ref{sec:solutionSPDE})
as well as their numerical
approximation processes
(see Subsection~\ref{sec:numerics}).
The identity~\eqref{eq:mildito_intro}
is then a special case of 
equation~\eqref{eq:itoformel_start}
in Theorem~\ref{thm:ito}
below
in which a mild It\^{o} formula
for mild It\^{o} processes
is established.

Let us outline how
%In Subsection~\ref{sec:mildito}
%below two proofs of
\eqref{eq:mildito_intro}
and Theorem~\ref{thm:ito}
respectively are established.
A central idea in the proof
of \eqref{eq:mildito_intro}
is to consider a suitable
transformation of the
solution process 
$ 
  X \colon [0,\infty) \times \Omega
  \to H
$
of the 
SPDE~\eqref{eq:SPDE.intro}.
The transformed stochastic process
is then a standard It\^{o} process
to which the standard 
It\^{o} formula 
(see 
\eqref{eq:standardito.intro})
can be applied.
Relating then the transformed
stochastic process in an
appropriate way to the original
solution process 
$ 
  X \colon [0,\infty) \times
  \Omega \rightarrow H
$
of the SPDE~\eqref{eq:SPDE.intro}
finally results 
in the mild It\^{o} 
formula~\eqref{eq:mildito_intro}.
Two types of transformations
are well suited for this job.
One possibility is, roughly speaking,
to multiply the solution process
$ X $ of the 
SPDE~\eqref{eq:SPDE.intro}
by $ e^{ - A t } $, $ t \in [0,\infty) $,
where
$ e^{ - A t } $, $ t \in [0,\infty) $,
has to be understood in an 
appropriate large Hilbert space
(see Subsection~\ref{sec:mildito}
below for details).
In that sense the transformed stochastic
process becomes rougher
than the solution process $ X $
of the SPDE~\eqref{eq:SPDE.intro}.
This transformation 
has been suggested in
Teichmann~\cite{t09}
and
Filipovi{\'c}, Tappe 
\citationand\ 
Teichmann~\cite{ftt10}
(see also 
Hausenblas 
\citationand\  
Seidler~\cite{hs01,hs08}).
The other possible transformation
goes into the other direction
and, roughly speaking,
consists of
multiplying the solution
process $ X $ of the
SPDE~\eqref{eq:SPDE.intro}
by
$ e^{ A ( T - t ) } $,
$ t \in [0,T] $,
for some large value $ T \in (0,\infty) $.
%The transformed stochastic process
%then becomes smoother than
%the solution process $ X $
%of the SPDE~\eqref{eq:SPDE.intro}.
This transformation is based on an idea 
in 
Conus 
\citationand\ 
Dalang~\cite{cd08}
and
Conus~\cite{c08}
(see also 
Debussche 
\citationand\  
Printems~\cite{dp09},
Lindner \citationand\ 
Schilling~\cite{ls10}
and 
Kov{\'a}cs, Larsson
\citationand\  
Lindgren~\cite{kll11}).
The second transformation,
which makes the transformed process
smoother than the  
solution process $ X $ of the
SPDE~\eqref{eq:SPDE.intro},
turns out to be more flexible and allows
us to prove Theorem~\ref{thm:ito}
in its full generality.
For more details on the proofs
of \eqref{eq:mildito_intro} and 
Theorem~\ref{thm:ito}
respectively, the reader
is referred to
Subsection~\ref{sec:mildito}
below.

In the remainder of this
introductory section, a few 
consequences
of the mild Ito 
formula~\eqref{eq:mildito_intro}
and its generalization in 
Theorem~\ref{thm:ito}
are illustrated.
For this 
let
$
  X^x \colon [0,\infty)
  \times \Omega \to H
$,
$ x \in H $,
be a family of 
adapted stochastic processes
with continuous sample paths
satisfying
$
  X_t
=
  e^{ A t } x
+
  \int_0^t
  e^{ A (t - s) }
  F( X_s^x ) \, ds
+
  \int_0^t
  B( X_s^x ) \, dW_s
$
$ \P $-a.s.\ for 
all $ t \in [0,\infty) $
and all $ x \in H $
(see, e.g., 
Theorem~4.3 in 
Brze\'{z}niak~\cite{b97b}
or 
Theorem~6.2 in
Van Neerven, Veraar \& 
Weis~\cite{vvw08} 
for the up to indistinguishability
unqiue existence of such processes).
Then 
for every 
$ 
  r \in 
  \big(
    - \infty, 
    \min( \alpha + 1, \beta + 1 / 2 ) 
  \big)
$
and every
at most polynomially
growing continuous function
$
  \varphi \in C( H_r, V )
$
define 
the continuous
function
$ 
  u_{ \varphi } \colon
  [0,\infty) \times H_r
  \to V
$
through
$ 
  u_{ \varphi }(t,x) :=
  \E\big[
    \varphi( X^x_t )
  \big]
$
for all 
$ (t,x) \in [0,\infty) \times H_r $.
Under the assumption that
$ \alpha = \beta = 0 $
and that
$ F $ and $ B $ are three times continuously
Fr\'{e}chet differentiable with globally
bounded derivatives,
the functions
$
  u_{ \varphi } \colon [0,\infty) 
  \times H \to V
$,
$ 
  \varphi \colon H \to V
$
twice continuous Fr\'{e}chet differentiable
with globally bounded derivatives,
are strict solutions of 
the infinite dimensional
Kolmogorov 
partial differential
equation (PDE)
$
  \tfrac{ \partial }{
    \partial t
  }
  u_{ \varphi }(t,x)
  =
  (L u_{ \varphi })(t,x)
$
with
$
  u_{ \varphi }(0,x) = \varphi(x)
$
for 
$
  (t,x) \in (0,\infty)
  \times H_1
$
where 
$ 
  L \colon C^2( H, V ) 
  \to
  C( H_1, V ) 
$
is defined through
\begin{equation}
  (L \varphi)(x)
  :=
  \frac{ 1 }{ 2 }
  \text{Tr}\Big(
  \big( B(x) \big)^{ \! * }
  \varphi''( x ) \,
  B(x)
  \Big)
  +
  \varphi'( x )
  \big[ A x + F(x) \big]
\end{equation}
for all $ x \in H_1 = D(A) $
and all 
$ 
  \varphi 
  \in C^2( H, V )
$
(see
Theorem~7.5.1
in 
Da Prato \citationand\ 
Zabczyk~\cite{dz02}).
Infinite dimensional Kolmogorov
equation have
been intensively investigated
in the last two decades
(see, e.g., the monographs
\cite{MaRoeckner1992,c01,dz02b,DaPrato2004}
and articles
\cite{Zabczyk1999,Roeckner1999,RoecknerSobol2006,DaPrato2007}
and the references mentioned therein).
We prove here 
that
the functions
$ 
  u_{ \varphi } \colon [0,\infty)
  \times H \to V
$, 
$ \varphi $ sufficiently smooth, 
also solve 
another kind of 
Kolmogorov equation.
More precisely,
from 
\eqref{eq:mildito_intro}
we derive 
in 
Subsection~\ref{sec:kolmogorov}
below
(see 
\eqref{eq:mildP2})
the identity
\begin{equation}
\label{eq:mildKolmogorov.intro}
  u_{ \varphi }(t,x)
=
  \varphi( e^{ A t } x )
+
  \int_{ 0 }^t
  u_{ 
    L_{ t - s }( \varphi )
  }(s,x)
  \, ds
\end{equation}
for all 
$ (t,x) \in (0,\infty) \times H $
and all 
$ 
  \varphi \in
  \cup_{ 
    r < 
    \min( \alpha + 1, \beta + 1 / 2 )
  }
  C^2( H_r, V )
$
with at most polynomially growing
derivatives
where
$ 
  L_t \colon 
  \cup_{ 
    r < 
    \min( \alpha + 1, \beta + 1 / 2 )
  }
  C^2( H_r, V )
  \to 
  C( H, V)
$, 
$ t \in (0,\infty) $,
is a family of bounded linear 
operators defined through
\begin{equation}
\begin{split}
  \big( 
    L_{ t } \varphi
  \big)( x )
  :=
  \frac{ 1 }{ 2 }
  \text{Tr}\Big(
  \big( e^{ A t } B(x) \big)^{ \! * }
  \varphi''( e^{ A t } x ) \,
  e^{ A t } B(x)
  \Big)
  +
  \varphi'( e^{ A t } x )
  \, e^{ A t } F(x)
\end{split}
\end{equation}
for all
$ x \in H $,
$ 
  \varphi \in 
  \cup_{ 
    r < 
    \min( \alpha + 1, \beta + 1 / 2 )
  }
  C^2( H_r, V )
$, 
$ t \in (0,\infty) $.
Equation~\eqref{eq:mildKolmogorov.intro} 
is somehow a {\it mild Kolmogorov 
backward equation}.
From 
\eqref{eq:mildKolmogorov.intro}
we derive new regularity
properties 
of solutions of second-order
PDEs in Hilbert spaces.
More precisely,
using \eqref{eq:mildKolmogorov.intro}
we establish 
in Theorem~\ref{thm:continuity}
below the existence
of real numbers
$ 
  c_{ \delta, \rho, q, T } 
  \in [0,\infty) 
$,
$
  \delta, \rho, q, T \in [0,\infty)
$,
such that the 
{\it regularity estimate}
\begin{equation}
\label{eq:Schauder}
  \sup_{ 
    x \in 
    H_{ \delta } 
  } 
  \!
  \left(
  \,
  \frac{
  \left\|
    u_{ \varphi }(t, x)
  \right\|_V
  }{
    \left( 
      1 + 
      \| x 
      \|_{
        H_\delta 
      }
    \right)^{ (q + 2) }
  }
  \,
  \right)
\leq
  c_{ \delta, \rho, q, T }
  \cdot
  \| \varphi \|_{ t, q }^{ \delta, \rho }
\end{equation}
holds for all 
$ t \in (0,T] $,
$ 
  \varphi \in
  C^2( 
    H_{ \rho }
    , 
    V 
  )
$
with 
$
  \sup_{ x \in H_{ \rho } }
  \frac{ 
    \| \varphi''(x) \|_{
      L^{ (2) }( H_{ \rho }, V )
    }
  }{ 
    ( 1 + \| x \|_{ H_{ \rho } } )^q
  }
  < \infty
$,
$
  q, \delta, T \in [ 0, \infty )
$,
$
  \rho \in 
  [ 
    0, 
    \min( 
      \alpha + 1, 
      \beta + \frac{ 1 }{ 2 } 
    )
  )
$
where
\begin{equation}
\begin{split}
&
  \| \varphi \|_{ t, q }^{ \delta, \rho }
  :=
\\ & 
  \sup_{ x \in H_{ \delta } }
  \left[
  \frac{
    \left\| 
      \varphi( e^{ A t } x )
    \right\|_V
  }{
    \left(
      1 + \| x \|_{ H_{ \delta } }
    \right)^{ (q + 2) }
  }
  \right]
  +
  \int_0^t
  \left( t - s \right)^{ 
    \min( \delta - \rho, 0 )
  }
  \sup_{ 
    x \in H_{ \rho }
  }
  \left[
  \frac{
    \left\| 
      (K_t \varphi)'( x )
    \right\|_{ 
      L( H_{ \alpha }, V ) 
    }
  }{
    \left(
      1 + \| x \|_{ H_{ \rho } }
    \right)^{ ( q + 1 ) }
  }
  +
  \frac{
    \left\| 
      ( K_t \varphi )''( x )
    \right\|_{ 
      L^{(2)}( H_{ \beta }, V ) 
    }
  }{
    \left(
      1 + \| x \|_{ H_{ \rho } }
    \right)^{ q }
  }
  \right] 
  ds
\end{split}
\end{equation}
for all
$ 
  \varphi \in
  C^2( 
    H_{ \rho }
    , 
    V 
  )
$
with 
$
  \sup_{ x \in H_{ \rho } }
  \frac{ 
    \| \varphi''(x) \|_{
      L^{ (2) }( H_{ \rho }, V )
    }
  }{ 
    ( 1 + \| x \|_{ H_{ \rho } } )^q
  }
  < \infty
$,
$
  t, q, \delta \in [ 0, \infty )
$,
$
  \rho \in 
  [ 
    0, 
    \min( 
      \alpha + 1, 
      \beta + \frac{ 1 }{ 2 } 
    )
  )
$
and where
$ K_t \colon C( H_1, V ) \to C( H, V ) $,
$ t \in (0,\infty) $,
is defined through
$
  (K_t \varphi)(x) = \varphi( e^{ A t } x )
$
for all 
$ x \in H $,
$ t \in (0,\infty) $.
The constants 
$  
  c_{ \delta, \rho, q, T }
$,
$
  \delta, \rho, q, T \in [0,\infty)
$,
appearing
in \eqref{eq:Schauder}
are described explicitly
in Theorem~\ref{thm:continuity}
below.
Next a direct consequence
of the regularity 
estimate~\eqref{eq:Schauder}
is presented.
For this let 
$ 
  ( 
    C^2_{ Lip }( H, \R ) ,
    \left\| \cdot 
    \right\|_{ C^2_{ Lip }( H, \R ) }
  )
$
be the real Banach space of all twice
continuously differentiable 
globally Lipschitz continuous
real valued functions on $ H $
% from $ H $ to $ \R $
with globally Lipschitz 
continuous derivatives 
(see \eqref{eq:defCnLip} in
Subsection~\ref{sec:setting}
for details).
Moreover, 
for every $ t \in (0,\infty) $
let
$ 
  ( 
    \mathcal{G}_t( H, \R ) ,
    \left\| \cdot 
    \right\|_{ \mathcal{G}_t( H, \R ) }
  )
$
be the completion of
the normed real vector
space
$ 
  ( 
    C^2_{ Lip }( H, \R ) ,
    \left\| \cdot 
    \right\|_{ t, 0 }^{ 0, 0 }
  )
$.
Then consider the mapping
$
  \mathcal{I} \colon
  \{
    \mu \colon \mathcal{B}(H)
    \to [0,1]
    \text{ probability measure}
    \colon
    \int \| x \|_H \, \mu( dx ) < \infty
  \}
  \to
  (
    C^2_{ Lip }( H, \R )
  )'
$
given by
$
  \big(
    \mathcal{I}( \mu )  
  \big)( \varphi )
  =
  \int \varphi( x ) \, \mu( dx )
$
for all
$ 
  \varphi \in 
  C^2_{ Lip }( H, \R )
$
and all probability measures
$
  \mu \colon \mathcal{B}(H)
  \to [0,1]
$
with
$
  \int \| x \|_H \, \mu( dx ) < \infty
$.
Lemma~\ref{lem:embedding} below
proves that $ \mathcal{I} $
is injective, that is,
$ \mathcal{I} $ embeds the
probability measures with finite
first absolute moments into
linear forms on $ C^2_{ Lip }( H, \R ) $.
Next note that
$
  \mathcal{I}( \P_{ X_t } )
  \in 
  (
    C^2_{ Lip }( H, \R )
  )'
$
for all $ t \in [0,\infty) $
where
$
  \P_{ X_t }[ A ]
  =
  \P\big[ X_t \in A \big]
$
for all
$ A \in \mathcal{B}(H) $,
$ t \in [0,\infty) $
are the probability measures
of the solution process
$ X_t $, $ t \in [0,\infty) $,
of the SPDE~\eqref{eq:SPDE.intro}.
From \eqref{eq:Schauder} we then
infer 
for every $ t \in (0,\infty) $
that
$
  \mathcal{I}( \P_{ X_t } )
  \in 
  (
    C^2_{ Lip }( H, \R )
  )'
$
is not only in
$ 
  (
    C^2_{ Lip }( H, \R )
  )'
$
but in the
smaller space
$
  (
    \mathcal{G}_t( H, \R )
  )'
$
too
(the embedding
$
  (
    \mathcal{G}_t( H, \R )
  )'
  \subset
  (
    C^2_{ Lip }( H, \R )
  )'
$
continuously is proved
in Lemma~\ref{lem:norm} 
below).
We thus have established
more {\it regularity of the
probability measures}
$ \P_{ X_t } $,
$ t \in (0,\infty) $,
of the solution process
of the SPDE~\eqref{eq:SPDE.intro}.

Another application
of the regularity 
estimate~\eqref{eq:Schauder}
and the
mild Kolmogorov backward
equation~\eqref{eq:mildKolmogorov.intro}
is the analysis
of continuity properties of solutions of 
second-order PDEs in Hilbert spaces
(see, e.g., the books 
\cite{MaRoeckner1992,c01,dz02b,DaPrato2004}).
More precisely,
Corollary~\ref{cor:temporal}
in Section~\ref{sec:weakregularity}
below proves that
there exist real
number 
$ 
  c_{ r, \delta, \rho, T } \in [0,\infty) 
$,
$
  r, \delta, \rho, T \in [0,\infty)
$,
such that
\begin{equation}
\label{eq:Hoelder.intro}
  \left\|
    u_{ \varphi }(t_1, x_1) -
    u_{ \varphi }(t_2, x_2)
  \right\|_V
  \leq
  \left[
    \tfrac{
      c_{ r, \delta, \rho, T } 
      \,
      ( 
        1 
        + 
        \| x_1 
        \|_{
          H_{ \delta }
        }^3 
        + 
        \| x_2
        \|_{
          H_{ \delta }
        }^3 
      )
      \,
      \| \varphi \|_{
        C^2_{ Lip }( H_{ \rho }, V )
      }
    }{
      \left| 
        \min( t_1, t_2 ) 
      \right|^{
        \max( r + \rho - \delta, 0 )
      }    
    }
  \right]
  \Big[
    \left| t_1 - t_2 \right|^{
      r
    }
    +
    \left\| x_1 - x_2 \right\|_{ 
      H_{ \delta } 
    }
  \Big]
\end{equation}
for all
$ 
  t_1, t_2 \in (0,T] 
$,
$ 
  x_1, x_2 \in D( (-A)^{ \delta } ) 
$,
$ 
  \varphi \in 
  C^2_{ Lip }( H_{ \rho }, V)
$,
$ 
  r \in 
  [0, 1+ \alpha - \rho ) \cap
  [0, 1+ 2 \beta - 2 \rho ) 
$,
$ 
  \delta, T \in [0, \infty)
$,
$
  \rho \in 
  [ 
    0, 
    \min( 
      \alpha + 1, 
      \beta + \frac{ 1 }{ 2 } 
    )
  )
$.
Inequality~\eqref{eq:Hoelder.intro} 
thus proves 
H\"{o}lder continuity of 
the solutions 
$ 
  u_{ \varphi } \colon 
  [0,\infty) \times H \to V
$,
$ \varphi \in C^2_{ Lip }( H, V ) $,
of second-order Kolmogorov PDEs
in infinite dimensions.
In particular,
in the case of the
example 
SPDE~\eqref{eq:SPDE.intro.ex},
inequality~\eqref{eq:Hoelder.intro}
ensures that
\begin{equation}
\label{eq:Hoelder2.intro}
  \sup_{ 
    \substack{
      t_1, t_2 \in [0,T]
    \\
      t_1 < t_2
    }
  }
  \left(
  \frac{
    |t_1|^{ r } \,
    \big\|
    \mathbb{E}\big[ 
      \varphi( X_{ t_1 } )
    \big]
    -
    \mathbb{E}\big[ 
      \varphi( X_{ t_2 } )
    \big]
    \big\|_V
  }{
    | t_1 - t_2 |^r 
  } 
  \right)
  < \infty
\end{equation}
for all
$ 
  \varphi \in C^2_{ Lip }( H, V ) 
$,
$ T \in (0,\infty) $
and all
$ 
  r \in 
  [0, \frac{ 1 }{ 2 } ) 
$.
Results in the literature imply
that \eqref{eq:Hoelder2.intro}
holds for all
$ r \in [ 0, \frac{ 1 }{ 4 } ] $.
More formally, 
in the case of 
the SPDE~\eqref{eq:SPDE.intro.ex},
we get
from the 
global Lipschitz continuity of 
$ \varphi $
that
\begin{equation}
\label{eq:strong.intro}
\begin{split}
&
  \sup_{ 
    \substack{
      t_1, t_2 \in [0,T]
    \\
      t_1 < t_2
    }
  }
  \frac{
    |t_1|^{ r } \,
    \big\|
    \mathbb{E}\big[ 
      \varphi( X_{ t_1 } )
    \big]
    -
    \mathbb{E}\big[ 
      \varphi( X_{ t_2 } )
    \big]
    \big\|_V
  }{
    | t_1 - t_2 |^r 
  } 
\leq
  \left[
    \sup_{ 
      \substack{
        x, y \in H
      \\
        x \neq y
      }
    } 
    \tfrac{
      \| \varphi(x) - \varphi(y) \|_V
    }{
      \| x - y \|_H
    }
  \right]
  \left[
  \sup_{ 
    \substack{
      t_1, t_2 \in [0,T]
    \\
      t_1 < t_2
    }
  }
  \frac{
    |t_1|^{ r } \,
    \mathbb{E}\big[ 
      \|
        X_{ t_1 }
        -
        X_{ t_2 } 
      \|_{ H }
    \big]
  }{
    | t_1 - t_2 |^r 
  } 
  \right]
  < \infty
\end{split}
\end{equation}
for all
$ 
  \varphi \in C^2_{ Lip }( H, V ) 
$,
$ T \in (0,\infty) $
and all
$ 
  r \in 
  [0, \frac{ 1 }{ 4 }]
$
where the second factor
on the right hand side
of \eqref{eq:strong.intro}
is finite due to 
Theorem~6.3 in 
Van Neerven, Veraar \& 
Weiss~\cite{vvw08}
for the case $ r \in [0,\frac{1}{4}) $
and due to
Corollaries~A.16 and A.35 in
%Jentzen \citationand\ Kloeden 
\cite{jk12}
for the case $ r = \frac{ 1 }{ 4 } $
(see also 
Brze\'{z}niak~\cite{b97b},
Kruse \citationand\ 
Larsson~\cite{KruseLarsson2011},
Van Neerven, Veraar \& 
Weiss~\cite{vvw11}
for related results).
This shows that
regularity results in the literature
ensure 
that \eqref{eq:Hoelder2.intro}
holds for all
$ r \in [ 0, \frac{ 1 }{ 4 } ] $.
Up to our best knowledge,
this is the first result in
the literature which
establishes
that \eqref{eq:Hoelder2.intro}
also holds in the regime
$ r \in (\frac{1}{4}, \frac{1}{2}) $.

A further application
of the regularity 
estimate~\eqref{eq:Schauder}
and the
mild Kolmogorov backward
equation~\eqref{eq:mildKolmogorov.intro}
is the weak 
approximation of SPDEs.
Let us illstrate this in the case of
spectral Galerkin projections
for the example
SPDE~\eqref{eq:SPDE.intro.ex}.
More precisely,
in the case of the 
SPDE~\eqref{eq:SPDE.intro.ex},
Corollary~\ref{cor:galerkin}
in Section~\ref{sec:weakregularity}
implies 
that
there exist real
numbers 
$ C_{ r, T } \in [0,\infty) $,
$ r, T \in [0,\infty) $,
such that
\begin{equation}
\label{eq:weaknumeric_intro}
  \left\|
  \E\big[
    \varphi\big( X_T \big)
  \big]
  -
  \E\big[
    \varphi\big( P_N( X_T ) 
    \big) 
  \big]
  \right\|_V
\leq
  \frac{
    C_{ r, T } 
    \left\|
      \varphi
    \right\|_{ C^2_{ Lip }( H, V ) }
  }{
    N^r
  }
\end{equation}
for all $ N \in \mathbb{N} $,
$ T \in (0,\infty) $,
$
  \varphi \in 
  C^2_{ Lip }( H, V )
$
and all
$ 
  r \in [0, 1)
$
where
$ P_N \in L(H) $,
$ N \in \N $,
are spectral Galerkin
projections defined by
$
  ( P_N v)(x)
:=
  \sum_{ n = 1 }^N
  2 \sin( n \pi x )
  \int_0^1
  v(y) \, \sin( n \pi y )
  \, dy
$
for all 
$ x \in (0,1) $,
$ 
  v \in H = L^2( (0,1), \R )
$
and all $ N \in \N $.
Inequality~\eqref{eq:weaknumeric_intro}
and
Corollary~\ref{cor:galerkin}
respectively
are a straightforward consequence
of the regularity 
estimate~\eqref{eq:Schauder}
(see Section~\ref{sec:weakregularity}
for details).
In the case of the
stochastic heat equation with
additive noise
$ f(x,y) = 0 $
and $ b(x,y) = 1 $
for all $ x \in (0,1) $,
$ y \in \R $
in \eqref{eq:SPDE.intro},
inequality~\eqref{eq:weaknumeric_intro}
follows 
for all $ r \in [0,1) $
from the results in
\cite{dp09,ls10,kll11,KovacsLarssonLindgren2012}
at least in the case of bounded test
functions
(see also \cite{h03b,dd06,gkl09,h10c,d11,Brehier2012,Kruse2012} 
for further numerical
weak approximation results for SPDEs).
In addition, in the general
setting of the 
SPDE~\eqref{eq:SPDE.intro},
it is well know that
inequality~\eqref{eq:weaknumeric_intro}
holds for all $ r \in [0,\frac{1}{2}] $.
Indeed, 
in that case,
we get from
the global
Lipschitz continuity of
$ \varphi $ that
\begin{equation}  
\begin{split}
&
  \left\|
  \E\big[
    \varphi\big( X_T \big)
  \big]
  -
  \E\big[
    \varphi\big( P_N( X_T ) 
    \big) 
  \big]
  \right\|_V
\leq
  \left\|
    \varphi
  \right\|_{
    C^2_{ Lip }( H, V )
  }
  \E\big[
  \|
    ( I - P_N )
    X_T 
  \|_H
  \big]
\\ & \leq
  \left\|
    \varphi
  \right\|_{
    C^2_{ Lip }( H, V )
  }
  \E\big[
  \|
    X_T 
  \|_{ H_{ 1 / 4 } }
  \big]
  \|
    ( I - P_N )
    ( \eta - A )^{ - 1 / 4 }
  \|_{ L(H) }
\leq
  \frac{
    \left\|
      \varphi
    \right\|_{
      C^2_{ Lip }( H, V )
    }
    \E\big[
    \|
      X_T 
    \|_{ H_{ 1 / 4 } }
    \big]
  }{
    \sqrt{ N \pi }
  }
  < \infty
\end{split}
\end{equation}
for all 
$ N \in \mathbb{N} $,
$ T \in (0,\infty) $
and all
$
  \varphi \in 
  C^2_{ Lip }( H, V )
$
where finiteness of
$
    \E\big[
    \|
      X_T 
    \|_{ H_{ 1 / 4 } }
    \big]
$
for all $ T \in (0,\infty) $
follows, e.g., 
from 
Theorem~5.1 in \cite{jk12}
(see also 
Kruse \citationand\ 
Larsson~\cite{KruseLarsson2011}
and
Van Neerven, Veraar
\citationand\ 
Weiss~\cite{vvw11}
and the references therein
for similar results).
This shows that regularity results
from the literature ensure that
\eqref{eq:weaknumeric_intro}
holds for all
$ r \in [0,\frac{1}{2}] $.
The present article
proves that 
\eqref{eq:weaknumeric_intro}
also holds for all
$ r \in [0,1) $.
Observe that 
\eqref{eq:weaknumeric_intro}
estimates the weak
approximation error
of spatial spectral
Galerkin projections
only instead of more complicated
spatial approximations
(see also
Corollary~\ref{cor:galerkin}
in Section~\ref{sec:weakregularity}
below for a generalization
of \eqref{eq:weaknumeric_intro})
and also the time interval 
and the noise are not discretized 
in \eqref{eq:weaknumeric_intro}.
We believe that the mild
Kolmogorov backward 
equation~\eqref{eq:mildKolmogorov.intro} 
can also be used to solve
these more complicated
weak numerical approximation 
problems for SPDEs
and plan to develop these results
in a future publication.

Another application
of the mild It\^{o}
formula~\eqref{eq:mildito_intro}
and the mild
Kolmogorov backward 
equation~\eqref{eq:mildKolmogorov.intro}
respectively
are the derivation of strong
and weak stochastic Taylor
expansions of solutions
of SPDEs.
Details can be found 
in Subsection~\ref{sec:tay}
below.
These Taylor expansions can
then be used to derive 
higher order numerical schemes
for SPDEs.
In Subsection~\ref{sec:mil}
below this is illustrated
in the case of Milstein
scheme for SPDEs.

\subsubsection*{Acknowledgements}

We thank Josef
Teichmann for his instructive
presentations on the method of
the moving frame at the conference
``Rough paths in interaction''
at the Institut Henri Poincar\'{e}.
In addition, we thank Daniel
Conus for pointing out his instructive
thesis to us at the conference
``Recent Advances in the Numerical
Approximation of Stochastic Partial
Differential Equations'' at
the Illinois Institute of Technology. 
Moreover, we are grateful to
Raphael Kruse, Stig Larsson and
Jan van Neerven for 
their very helpful comments.
Finally, we gratefully acknowledge
Meihua~Yang for fruitful discussions
on the analysis of weak temporal regularity 
of solutions of SPDEs
(see Corollary~\ref{cor:temporal})
and its consequences in
numerical analysis.

This work has been 
partially supported 
by the Collaborative 
Research Centre $701$
``Spectral Structures 
and Topological Methods 
in Mathematics'',
by the
International Graduate 
School ``Stochastics and Real 
World Models'', 
by the research
project ``Numerical solutions
of stochastic differential equations
with non-globally Lipschitz continuous
coefficients'' (all funded by the German
Research Foundation)
and by the 
BiBoS Research Center.
The support of Issac Newton
Institute for Mathematical
Sciences in Cambridge is
also gratefully acknowledged
where part of this was done
during the special semester
on ``Stochastic Partial Differential
Equations''.

\section{Mild stochastic
calculus}
\label{sec:mildcalc}
Throughout this section 
assume that the
following setting is fulfilled.
Let $ \mathbbm{I} \subset [0,\infty) $
be a 
closed and convex subset
of $ [0,\infty) $
with nonempty interior,
let 
$ \left( \Omega, \mathcal{F}, \mathbb{P} \right) $ 
be a probability space 
with a normal filtration
$ ( \mathcal{F}_t )_{ t \in \mathbbm{I} } $
and let
$ 
  \big( 
    \check{H}, 
    \left< \cdot , \cdot \right>_{ \check{H} }, 
    \left\| \cdot \right\|_{ \check{H} }
  \big) 
$,
$ 
  \big( 
    \tilde{H}, 
    \left< \cdot , \cdot \right>_{ \tilde{H} }, 
    \left\| \cdot \right\|_{ \tilde{H} }
  \big) 
$,
$ 
  \big( 
    \hat{H}, 
    \left< \cdot , \cdot \right>_{ \hat{H} }, 
    \left\| \cdot \right\|_{ \hat{H} }
  \big) 
$
and
$ 
  \left( 
    U, 
    \left< \cdot , \cdot \right>_U, 
    \left\| \cdot \right\|_U 
  \right) 
$ 
be separable
$\mathbb{R}$-Hilbert spaces
with 
$ 
  \check{H}
  \subset
  \tilde{H} 
  \subset 
  \hat{H} 
$
continuously and densely.
In addition, let 
$ 
  Q \colon U \rightarrow U 
$ be 
a bounded nonnegative symmetric linear
operator and 
let 
$ \left( W_t \right)_{ t \in \mathbbm{I} }
$
be a cylindrical $ Q $-Wiener 
process with respect 
to 
$ 
  ( \mathcal{F}_t )_{ t \in \mathbbm{I} } 
$.
Moreover, by
$ 
  \left( 
    U_0, 
    \left< \cdot , \cdot \right>_{ U_0 }, 
    \left\| \cdot \right\|_{ U_0 }
  \right) 
$ 
the $ \mathbb{R} $-Hilbert space 
with
$ U_0 = Q^{ 1/2 }( U ) $
and
$ 
  \| u \|_{ U_0 } 
  = \| Q^{ - 1/2 }( u ) \|_U 
$ 
for all $ u \in U_0 $
is denoted
(see, for example, 
Proposition~2.5.2 in 
Pr\'{e}v\^{o}t 
\citationand\
R\"{o}ckner~\cite{PrevotRoeckner2007}).
Here and below 
$
  S^{ - 1 } \colon
  \text{im}(S) \subset U
  \rightarrow U
$
denotes the pseudo inverse
of a bounded linear
operator 
$ 
  S \colon U \rightarrow U \in L(U) 
$
(see, e.g., Appendix~C in \cite{PrevotRoeckner2007}).
In addition, let
$ 
  i_v \colon 
  L( \hat{H}, \check{H} )
  \rightarrow \check{H}
  \in 
  L\big( 
    L( \hat{H}, \check{H} ), \check{H} 
  \big)
$,
$ v \in \hat{H} $,
be a family of bounded
linear operators defined
through
$ i_v A = A v $
for all
$ A \in L( \hat{H}, \check{H} ) $ and all
$ v \in \hat{H} $.
Then by 
$
  \mathcal{S}( \hat{H}, \check{H} ) 
  := 
  \sigma( 
    \,
    \cup_{ v \in \hat{H} } \,
    i_v^{-1}( \mathcal{B}( \check{H} ) )
    \,
  ) 
=
\sigma( 
  \{
    i_v^{-1}( \mathcal{A} )
    \subset L( \hat{H}, \check{H} )
    \colon
    v \in \hat{H},
    \mathcal{A} \in 
    \mathcal{B}( \check{H} ) 
  \}
) 
$
the strong sigma algebra on
$ L( \hat{H}, \check{H} ) $ 
is denoted
(see, for instance,
Section~1.2 in 
Da Prato
\citationand\ Zabczyk~\cite{dz92}).
Finally, let
$ \angle \subset \mathbbm{I}^2 $
be a set defined
through
$
  \angle :=
  \left\{ 
    (t_1, t_2) \in \mathbbm{I}^2 \colon
    t_1 < t_2
  \right\}
$
and let
$ \tau \in \mathbbm{I} $
be defined throught
$ \tau := \inf( \mathbbm{I} )
$.

\subsection{Mild stochastic processes}
\label{sec:mildprocesses}

\begin{definition}[Mild
It\^{o} process]
\label{propdef}
Let 
$ 
  S \colon 
  \angle
  \rightarrow 
  L( \hat{H}, \check{H} )
$
be a
$ 
\mathcal{B}( \angle ) 
$/$
\mathcal{S}( \hat{H}, \check{H} ) 
$-measurable
mapping satisfying
$
  S_{ t_2, t_3 }
  S_{ t_1, t_2 }
=
  S_{ t_1, t_3 }
$
for all $ t_1, t_2, t_3
\in \mathbbm{I} $
with $ t_1 < t_2 < t_3 $.
Additionally,
let 
$ 
  Y \colon \mathbbm{I}
  \times \Omega
  \rightarrow \hat{H} 
$
and
$ 
  Z \colon \mathbbm{I}
  \times \Omega
  \rightarrow HS( U_0, \hat{H} )
$
be two predictable
stochastic processes
with
$
    \int_{ \tau }^{ t }
    \left\| S_{ s, t } Y_s
    \right\|_{ \check{H} }
    ds < \infty
$
$ \mathbb{P} $-a.s.\ and
$
    \int_{ \tau }^{ t }
    \left\| S_{ s, t } Z_s
    \right\|_{ HS( U_0, \check{H} ) }^2
    ds < \infty
$
$ \mathbb{P} $-a.s.\ for 
all $ t \in \mathbbm{I} $.
Then a predictable stochastic
process
$ 
  X \colon \mathbbm{I} \times \Omega
  \rightarrow \tilde{H}
$
satisfying
\begin{equation}
\label{eq:mildito}
  X_t 
= 
  S_{ \tau, t } \,
  X_{ \tau }
  +
  \int_{ \tau }^t
    S_{ s, t } \,
    Y_s
  \, ds
  +
  \int_{ \tau }^t
    S_{ s, t } \,
    Z_s
  \, dW_s
\end{equation}
$ \mathbb{P} $-a.s.\ 
for all 
$ 
  t \in \mathbbm{I} 
  \cap ( \tau, \infty )
$
is called 
a \underline{mild It\^{o} process}
(with \underline{semigroup} $S$,
\underline{mild drift} $ Y $
and \underline{mild diffusion} $ Z $).
\end{definition}
Note that if
$ 
  ( 
    \check{H}, 
    \left< \cdot , \cdot \right>_{ \check{H} }, 
    \left\| \cdot \right\|_{ \check{H} }
  ) 
  =
  ( 
    \hat{H}, 
    \left< \cdot , \cdot \right>_{ \hat{H} }, 
    \left\| \cdot \right\|_{ \hat{H} }
  ) 
$
and if the semigroup
$  
  S \colon \angle \rightarrow 
  L( \hat{H} ) 
$
satisfies
$ S_{ t_1, t_2 } = I $
for all $ (t_1, t_2) \in \angle $,
then the mild It\^{o} 
process~\eqref{eq:mildito}
is nothing else but a standard
It\^{o} process.
(Throughout this article the terminology 
``standard It\^{o} process'' instead
of simply ``It\^{o} process''
is used in order to distinguish
more clearly from the above 
notion of a mild It\^{o} process.)
Every standard It\^{o} process
is thus a mild It\^{o} process too.
However, a mild It\^{o} process
is, in general, 
not a standard It\^{o} process
(see Section~\ref{sec:appl}
for some examples).
The concept of a mild It\^{o} process
in Definition~\ref{propdef}
thus generalizes the concept
of a standard It\^{o} process.
In concrete examples of mild
It\^{o} processes it will be crucial
that the semigroup
$ 
  S \colon
  \angle \rightarrow
  L( \hat{H}, \check{H} )
$
in Definition~\ref{propdef}
depends explicitly on 
both variables
$ t_1 $ and $ t_2 $
with $ (t_1, t_2) \in \angle $
instead of on
$ t_2 - t_1 $ only
(see Subsection~\ref{sec:numerics}
for details).
The assumptions
$
    \int_{ \tau }^{ t }
    \left\| S_{ s, t } Y_s
    \right\|_{ \check{H} }
    ds < \infty
$
$ \mathbb{P} $-a.s.\ and
$
    \int_{ \tau }^{ t }
    \left\| S_{ s, t } Z_s
    \right\|_{ HS( U_0, \check{H}) }^2
    ds < \infty
$
$ \mathbb{P} $-a.s.\ for 
all $ t \in \mathbbm{I} $
in Definition~\ref{propdef}
ensure that both
the deterministic and the
stochastic integral 
in \eqref{eq:mildito}
are well defined.
In the next step an immediate
consequence of
Definition~\ref{propdef}
is presented.

\begin{prop}
\label{propsimple}
Let
$ 
  X \colon \mathbbm{I} \times \Omega
  \rightarrow \tilde{H}
$
be a mild It\^{o} process
with semigroup 
$ 
  S \colon \angle
  \rightarrow L( \hat{H}, \check{H} )
$,
mild drift 
$ 
  Y \colon \mathbbm{I} \times \Omega
  \rightarrow \hat{H}
$ 
and mild 
diffusion 
$ 
  Z \colon \mathbbm{I} \times \Omega
  \rightarrow HS( U_0, \hat{H} )
$.
Then
\begin{equation}
\label{eq:itovolterra}
  X_{ t_2 }
= 
  S_{ t_1, t_2 } \,
  X_{ t_1 }
  +
  \int_{t_1}^{t_2}
    S_{ s, t_2 } \,
    Y_s
  \, ds
  +
  \int_{t_1}^{t_2}
    S_{ s, t_2 } \,
    Z_s
  \, dW_s
\end{equation}
$ \mathbb{P} $-a.s.\ 
for all 
$ t_1, t_2 \in \mathbbm{I} $
with $ t_1 < t_2 $.
\end{prop}

Proposition~\ref{propsimple}
follows directly from
Theorem~\ref{thm:ito}
below.
Obviously, 
equation~\eqref{eq:itovolterra}
in Proposition~\ref{propsimple}
generalizes
equation~\eqref{eq:mildito}
in the definition
of a mild It\^{o} process.
Let us complete this subsection
on mild It\^{o} processes
with the notion of a certain subclass of
mild It\^{o} processes.

\begin{definition}[Mild
martingale]
A
mild It\^{o} process
$ 
  X \colon \mathbbm{I} \times \Omega
  \rightarrow \tilde{H}
$
with semigroup 
$ 
  S \colon \angle
  \rightarrow L( \hat{H}, \check{H} )
$,
mild drift 
$ 
  Y \colon \mathbbm{I} \times \Omega
  \rightarrow \hat{H}
$ 
and mild 
diffusion 
$ 
  Z \colon \mathbbm{I} \times \Omega
  \rightarrow HS( U_0, \hat{H} )
$
satisfying
$
\mathbb{E}\big[ \| X_t \|_{ \tilde{H} }
\big] < \infty 
$
for all $ t \in \mathbbm{I} $
is called a
\underline{mild martingale}
(with respect to the filtration
$ 
  \mathcal{F}_t,
  t \in \mathbbm{I},
$
and with respect to
the semigroup $ S $) if
\begin{equation}
\label{eq:mildmartingale}
  \mathbb{E}\!\left[
    X_{ t_2 } 
  \big| 
    \mathcal{F}_{ t_1 }
  \right]
  =
  S_{ t_1, t_2 } \,
  X_{ t_1 }
\end{equation}
$ \mathbb{P} $-a.s.\ for 
all $ t_1, t_2 \in \mathbbm{I} $
with $ t_1 < t_2 $.
\end{definition}

\begin{prop}[Stochastic convolutions]
\label{propmartingale}
Let
$ 
  X \colon \mathbbm{I} \times \Omega
  \rightarrow \tilde{H}
$
be a mild It\^{o} process
with semigroup 
$ 
  S \colon 
  \angle
  \rightarrow L( \hat{H}, \check{H} )
$,
mild drift 
$ 
  Y \colon \mathbbm{I} \times \Omega
  \rightarrow \hat{H}
$ 
and mild 
diffusion 
$ 
  Z \colon \mathbbm{I} \times \Omega
  \rightarrow HS( U_0, \hat{H} )
$
satisfying
$
\mathbb{E}\big[ 
  \| X_t \|_{ \tilde{H} }^2
\big] < \infty 
$
for all $ t \in \mathbbm{I} $.
If $ \mathbb{P}\big[ Y_t = 0 \big] = 1 $
for all $ t \in \mathbbm{I} $, then
$ X $ is a mild martingale
with respect to the
filtration $ \mathcal{F}_t $,
$ t \in \mathbbm{I} $, and
with respect to
the semigroup $ S $.
\end{prop}

\begin{proof}[Proof
of Propostion~\ref{propmartingale}]
Propostion~\ref{propsimple}
yields
\begin{equation}
\label{eq:convolution}
  X_{ t_2 }
  =
  S_{ t_1, t_2 } \, X_{ t_1 }
  +
  \int_{ t_1 }^{ t_2 }
  S_{ s, t_2 }
  \, Z_s \, dW_s
\end{equation}
$ \mathbb{P} $-a.s.\ for 
all $ t_1, t_2 \in \mathbbm{I} $
with $ t_1 < t_2 $.
Equation~\eqref{eq:convolution}
and the assumption
$
  \mathbb{E}\big[
    \| X_t \|_{ \tilde{H} }^2
  \big]
  < \infty 
$
for all $ t \in \mathbbm{I} $
imply
equation~\eqref{eq:mildmartingale}.
The proof of Proposition~\ref{propmartingale}
is thus completed.
\end{proof}

\subsection{Mild It\^{o} formula}
\label{sec:mildito}

Let $ \mathcal{J} $ be a set and let
$ g_j \in U_0 $, 
$ j \in \mathcal{J} $,
be an arbitrary orthonormal basis
of the $ \mathbb{R} $-Hilbert
space
$ \left( 
  U_0, 
  \left< \cdot, \cdot 
  \right>_{ U_0 }, 
  \left\| \cdot
  \right\|_{ U_0 } 
\right) $.
For an
$ \mathbb{R} $-vector space
$ ( V, 
\left\| \cdot \right\|_V ) 
$ 
and a function 
$ 
  \varphi \in 
  C^{ 1, 2 }( 
    \mathbbm{I} \times \check{H},
    V 
  )
$
we denote by
$
  \partial_1 \varphi
  \in 
  C( \mathbbm{I} \times \check{H}, V )
$,
$
  \partial_2 \varphi
  \in 
  C( \mathbbm{I} \times \check{H}, 
    L(\check{H}, V) 
  )
$
and 
$
  \partial_2^2 \varphi
  \in 
  C( \mathbbm{I} \times \check{H}, 
    L^{(2)}( \check{H}, V) 
  )
$
the partial 
Fr\'{e}chet 
derivatives of $ \varphi $
given by
$
  ( \partial_1 \varphi)(t,x)
  =
  \big(
    \tfrac{ \partial \varphi }{ \partial t }
  \big)( t, x )
$,
$
  ( \partial_2 \varphi)(t,x)
  =
  \big(
    \tfrac{ \partial \varphi }{ \partial x }
  \big)( t, x )
$
and 
$
  ( \partial_2^2 \varphi)(t,x)
  =
  \big(
    \tfrac{ \partial^2 \varphi }{ \partial x^2 }
  \big)( t, x )
$
for all
$ t \in \mathbbm{I} $
and all 
$ x \in \check{H} $.
\begin{theorem}[Mild
It{\^o} formula]
\label{thm:ito}
Let 
$ 
  X \colon \mathbbm{I} \times \Omega
  \rightarrow \tilde{H} 
$ 
be a mild It{\^o} process
with semigroup 
$ 
  S \colon \angle
  \rightarrow L( \hat{H}, \check{H} )
$,
mild drift 
$ 
  Y \colon \mathbbm{I} \times
  \Omega \rightarrow \hat{H} 
$
and mild 
diffusion
$ 
  Z \colon \mathbbm{I} \times
  \Omega \rightarrow 
  HS(U_0,\hat{H}) 
$
and let $ ( V, 
\left< \cdot, \cdot \right>_V, 
\left\| \cdot \right\|_V ) $ 
be a separable 
$ \mathbb{R} $-Hilbert
space.
Then
\begin{equation}
\label{eq:well1}
  \mathbb{P}\!\left[
    \int_{ t_0 }^t
    \left\|
      ( \partial_2 \varphi)( s, S_{ s, t } X_s )
      S_{ s, t } Y_s
    \right\|_V
    +
    \left\|
      ( \partial_2 \varphi)( s, S_{ s, t } X_s )
      S_{ s, t } Z_s
    \right\|_{ HS(U_0, V ) }^2 
    ds
    < \infty 
  \right] = 1,
\end{equation}
\begin{equation}
\label{eq:well2}
  \mathbb{P}\!\left[
    \int_{ t_0 }^t
    \left\|
      ( \partial_1 \varphi)( s, X_s )
    \right\|_V
    +
    \|
      (\partial^2_2 \varphi)( s, S_{ s, t } X_s )
    \|_{ L^{(2)}( \check{H}, V ) } \,
    \|
      S_{ s, t } Z_s
    \|_{ HS(U_0, \check{H} ) }^2 \,
    ds
    < \infty 
  \right] = 1
\end{equation}
% $ \mathbb{P} $-a.s.~and
and
\begin{equation}
\label{eq:itoformel_start}
\begin{split}
&
  \varphi( t, X_t )
 =
  \varphi( t_0,
    S_{ t_0, t } 
    X_{ t_0 } 
  )
  +
  \int_{ t_0 }^t
  (\partial_1
  \varphi)( s, 
    S_{ s, t } 
    X_s 
  ) \, ds
  +
  \int_{ t_0 }^t
  (\partial_2
  \varphi)( s, 
    S_{ s, t } 
    X_s 
  ) \,
  S_{ s, t } \, 
  Y_s \, ds
\\&\quad+
  \int_{ t_0 }^t
  (\partial_2
  \varphi)( s,
    S_{ s, t } 
    X_s 
  ) \,
  S_{ s, t } \, 
  Z_s \, dW_s
  +
  \frac{1}{2}
  \sum_{ j \in \mathcal{J} }
  \int_{ t_0 }^t
  ( \partial_2^2 \varphi )( s,
    S_{ s, t } 
    X_s 
  )
  \left(
    S_{ s, t } 
    Z_s g_j,
    S_{ s, t } 
    Z_s g_j
  \right) ds
\end{split}
\end{equation}
$ \mathbb{P} $-a.s.\ for 
all $ t_0, t \in \mathbbm{I} $
with $ t_0 < t $
and all
$ 
  \varphi \in C^{1,2}( 
    \mathbb{I} \times \check{H}, V 
  )
$. 
\end{theorem}
Note that \eqref{eq:well1}
and \eqref{eq:well2}
ensure that the 
possibly infinite sum
and all integrals in 
\eqref{eq:itoformel_start} 
are well defined.
Indeed, 
equation~\eqref{eq:well2}
implies
\begin{equation}
\begin{split}
\lefteqn{
  \sum_{ j \in \mathcal{J} }
  \int_{ t_0 }^t
  \left\|
    (\partial^2_2 \varphi)( s,
      S_{ s, t } 
      X_s 
    )
    \!\left(
      S_{ s, t } 
      Z_s g_j,
      S_{ s, t } 
      Z_s g_j
    \right) 
  \right\|_{ V }
  ds
}
\\ & \leq
  \int_{ t_0 }^t
  \left\|
    (\partial^2_2 \varphi)( s, 
      S_{ s, t } 
      X_s 
    )
  \right\|_{ L^{(2)}(\check{H}, V) }
  \left(
  \sum\nolimits_{ j \in \mathcal{J} }
  \left\|
    S_{ s, t } 
    Z_s g_j
  \right\|_{ \check{H} }^2
  \right)
  ds
\\ & =
  \int_{ t_0 }^t
  \left\|
    (\partial^2_2 \varphi)( s, 
      S_{ s, t } 
      X_s 
    )
  \right\|_{ L^{(2)}( \check{H}, V ) }
  \left\|
    S_{ s, t } 
    Z_s 
  \right\|_{ HS( U_0, \check{H} ) }^2
  ds
  < \infty
\end{split}
\end{equation}
$ \mathbb{P} $-a.s.\ for 
all $ t_0, t \in \mathbbm{I} $
with $ t_0 < t $
and all
$
  \varphi
  \in C^{ 1, 2 }( 
    \mathbbm{I} \times 
    \check{H}, V
  )
$.
Moreover, note that
the mild It\^{o} 
formula~\eqref{eq:itoformel_start}
is independent of the
particular choice
of the orthonormal
basis $ g_j \in U_0 $,
$ j \in \mathcal{J} $,
of $ U_0 $.

In the next step a certain
flow property of the mild
It\^{o} formula~\eqref{eq:itoformel_start}
is observed. To be more precise,
the mild It\^{o} formula~\eqref{eq:itoformel_start}
on the time
interval $ [\hat{t}, t] $
applied to the test function
$ \varphi( s, v ) $, 
$ s \in [\hat{t}, t] $,
$ v \in \check{H} $,
reads as
\begin{equation}
\label{eq:sgprop1}
\begin{split}
&
  \varphi( t, X_t )
 =
  \varphi( 
    \hat{t}, S_{ \hat{t}, t } 
    X_{ \hat{t} } 
  )
  +
  \int_{ \hat{t} }^t
  (\partial_1 \varphi)( s,
    S_{ s, t } 
    X_s 
  ) \,
  ds
  +
  \int_{ \hat{t} }^t
  (\partial_2 \varphi)( s,
    S_{ s, t } 
    X_s 
  ) \,
  S_{ s, t } \, 
  Y_s \, ds
\\&\quad+
  \int_{ \hat{t} }^t
  (\partial_2 \varphi)( s,
    S_{ s, t } 
    X_s 
  ) \,
  S_{ s, t } \, 
  Z_s \, dW_s
  +
  \frac{1}{2}
  \sum_{ j \in \mathcal{J} }
  \int_{ \hat{t} }^t
  (\partial_2^2 \varphi)( s,
    S_{ s, t } 
    X_s 
  )
  \left(
    S_{ s, t } 
    Z_s g_j,
    S_{ s, t } 
    Z_s g_j
  \right) ds
\end{split}
\end{equation}
$ \mathbb{P} $-a.s.\ for 
all $ \hat{t}, t \in \mathbbm{I} $
with $ \hat{t} < t $
and all 
$ 
  \varphi \in 
  C^{1,2}( 
    \mathbbm{I} \times 
    \check{H}, 
    V 
  ) 
$.
Moreover, observe that
the mild It\^{o}
formula~\eqref{eq:itoformel_start}
on the time interval
$ [t_0,\hat{t}] $
applied to the test function
$ \varphi( s, S_{ \hat{t}, t } v ) $,
$ s \in [t_0,\hat{t}] $,
$ v \in \check{H} $,
reads as
\begin{equation}
\label{eq:sgprop2}
\begin{split}
&
  \varphi( \hat{t}, S_{ \hat{t}, t } 
  X_{ \hat{t} } )
 =
  \varphi( 
    t_0, 
    S_{ t_0, t } 
    X_{ t_0 } 
  )
  +
  \int_{ t_0 }^{ \hat{t} }
  (\partial_1 \varphi)( s,
    S_{ s, t } 
    X_s 
  ) \,
  ds
  +
  \int_{ t_0 }^{ \hat{t} }
  (\partial_2 \varphi)( s,
    S_{ s, t } 
    X_s 
  ) \,
  S_{ s, t } \, 
  Y_s \, ds
\\&\quad+
  \int_{ t_0 }^{ \hat{t} }
  (\partial_2 \varphi)( s,
    S_{ s, t } 
    X_s 
  ) \,
  S_{ s, t } \, 
  Z_s \, dW_s
  +
  \frac{ 1 }{ 2 }
  \sum_{ j \in \mathcal{J} }
  \int_{ t_0 }^{ \hat{t} }
  (\partial_2^2 \varphi)( s,
    S_{ s, t } 
    X_s 
  )
  \left(
    S_{ s, t } 
    Z_s g_j,
    S_{ s, t } 
    Z_s g_j
  \right) ds
\end{split}
\end{equation}
$ \mathbb{P} $-a.s.\ for 
all $ t_0, \hat{t}, t \in \mathbbm{I} $
with $ t_0 < \hat{t} < t $
and all 
$ 
  \varphi \in 
  C^{1,2}( 
    \mathbbm{I} \times 
    \check{H}, 
    V 
  ) 
$.
Putting \eqref{eq:sgprop2}
into \eqref{eq:sgprop1} 
then results in the mild It\^{o}
formula~\eqref{eq:itoformel_start}
on the time interval $ [t_0,t] $
for $ t_0, t \in \mathbbm{I} $
with $ t_0 < t $.
If the test function 
$
  ( \varphi(t, x) )_{
    t \in \mathbbm{I},
    x \in \check{H}
  } 
  \in
  C^{1,2}(
    \mathbbm{I} \times \check{H}, V
  )
$
in the mild It\^{o} 
formula~\eqref{eq:itoformel_start}
does not depend on 
$ t \in \mathbbm{I} $, then the
mild It\^{o} formula in Theorem~\ref{thm:ito}
reads as follows.
\begin{cor}
\label{cor:itoauto}
Let 
$ 
  X \colon \mathbbm{I} \times \Omega
  \rightarrow \tilde{H} 
$ 
be a mild It{\^o} process
with semigroup 
$ 
  S \colon \angle
  \rightarrow 
  L( \hat{H}, \check{H} )
$,
mild drift 
$ 
  Y \colon \mathbbm{I} \times
  \Omega \rightarrow \hat{H} 
$
and mild 
diffusion
$ 
  Z \colon \mathbbm{I} \times
  \Omega \rightarrow 
  HS(U_0,\hat{H}) 
$
and let $ ( V, 
\left< \cdot, \cdot \right>_V, 
\left\| \cdot \right\|_V ) $ 
be a separable 
$ \mathbb{R} $-Hilbert
space.
Then
\begin{equation}
\label{eq:well1NON}
  \mathbb{P}\!\left[
    \int_{ t_0 }^t
    \left\|
      \varphi'( S_{ s, t } X_s )
      S_{ s, t } Y_s
    \right\|_V
    +
    \left\|
      \varphi'( S_{ s, t } X_s )
      S_{ s, t } Z_s
    \right\|_{ HS(U_0, V ) }^2 
    ds
    < \infty 
  \right] = 1,
\end{equation}
\begin{equation}
\label{eq:well2NON}
  \mathbb{P}\!\left[
    \int_{ t_0 }^t
    \|
      \varphi''( S_{ s, t } X_s )
    \|_{ L^{(2)}( \check{H}, V ) } \,
    \|
      S_{ s, t } Z_s
    \|_{ HS(U_0, \check{H} ) }^2 \,
    ds
    < \infty 
  \right] = 1
\end{equation}
% $ \mathbb{P} $-a.s.~and
and
\begin{align}
\label{eq:itoformel_startNON}
  \varphi( X_t )
&=
  \varphi( 
    S_{ t_0, t } 
    X_{ t_0 } 
  )
  +
  \int_{ t_0 }^t
  \varphi'(  
    S_{ s, t } 
    X_s 
  ) \,
  S_{ s, t } \, 
  Y_s \, ds
  +
  \int_{ t_0 }^t
  \varphi'( 
    S_{ s, t } 
    X_s 
  ) \,
  S_{ s, t } \, 
  Z_s \, dW_s
\nonumber
\\&\quad+
  \frac{1}{2}
  \sum_{ j \in \mathcal{J} }
  \int_{ t_0 }^t
  \varphi''( 
    S_{ s, t } 
    X_s 
  )
  \left(
    S_{ s, t } 
    Z_s g_j,
    S_{ s, t } 
    Z_s g_j
  \right) ds
\end{align}
$ \mathbb{P} $-a.s.\ for 
all 
$ t_0, t \in \mathbbm{I} $
with $ t_0 < t $
and all
$ 
  \varphi \in C^{2}( 
    \check{H}, V 
  )
$. 
\end{cor}

Let us illustrate
Theorem~\ref{thm:ito}
and Corollary~\ref{cor:itoauto}
by two simple examples.
The first one is a mild 
version of the stochastic 
integration by parts 
formula (see, e.g., Corollary~2.6
in \cite{bvvw08}).

\begin{example}[Mild stochastic
integration
by parts]
\label{ex:1}
Let
$ 
  ( 
    \check{\mathcal{H}}, 
    \left< \cdot , \cdot 
    \right>_{ \check{\mathcal{H}} }, 
    \left\| \cdot \right\|_{      
      \check{\mathcal{H}} 
    }
  ) 
$,
$ 
  ( 
    \tilde{\mathcal{H}}, 
    \left< \cdot , \cdot 
    \right>_{ \tilde{\mathcal{H}} }, 
    \left\| \cdot 
    \right\|_{ \tilde{\mathcal{H}} }
  ) 
$,
$ 
  ( 
    \hat{\mathcal{H}}, 
    \left< \cdot , \cdot 
    \right>_{ \hat{\mathcal{H}} }, 
    \left\| \cdot \right\|_{ \hat{\mathcal{H}} }
  ) 
$,
$ 
  ( 
    \mathcal{U}, 
    \left< \cdot , \cdot \right>_\mathcal{U}, 
    \left\| \cdot \right\|_{ \mathcal{U} } 
  ) 
$
and
$ 
  ( V, 
  \left< \cdot, \cdot \right>_V, 
  \left\| \cdot \right\|_V ) 
$ 
be separable
$\mathbb{R}$-Hilbert spaces
with 
$ 
  \check{\mathcal{H}} 
  \subset
  \tilde{\mathcal{H}} 
  \subset 
  \hat{\mathcal{H}} 
$
continuously and densely, 
let 
$ 
  X \colon \mathbbm{I} \times \Omega
  \rightarrow \tilde{H} 
$ 
be a mild It{\^o} process
with semigroup 
$ 
  S \colon \angle
  \rightarrow L( \hat{H}, \check{H} )
$,
mild drift 
$ 
  Y \colon \mathbbm{I} \times
  \Omega \rightarrow \hat{H} 
$
and mild 
diffusion
$ 
  Z \colon \mathbbm{I} \times
  \Omega \rightarrow 
  HS(U_0,\hat{H}) 
$
and 
let 
$ 
  \mathcal{X} \colon 
  \mathbbm{I} \times \Omega
  \rightarrow \tilde{\mathcal{H}} 
$ 
be a mild It{\^o} process
with semigroup 
$ 
  \mathcal{S} \colon \angle
  \rightarrow 
  L( 
    \hat{\mathcal{H}},
    \check{\mathcal{H}}
  )
$,
mild drift 
$ 
  \mathcal{Y} 
  \colon \mathbbm{I} \times
  \Omega \rightarrow \hat{\mathcal{H}} 
$
and mild 
diffusion
$ 
  \mathcal{Z} \colon \mathbbm{I} \times
  \Omega \rightarrow 
  HS(U_0,\hat{\mathcal{H}}) 
$.
Corollary~\ref{cor:itoauto} then shows
\begin{align}
&
  \varphi\big( 
    X_t, \mathcal{X}_t 
  \big)
  =
  \varphi\big( 
    S_{ t_0, t } X_t, 
    \mathcal{S}_{ t_0, t } \mathcal{X}_t 
  \big)
  +
  \int_{ t_0 }^t
  \varphi\big(
    S_{ s, t } Y_s, 
    \mathcal{S}_{ s, t } \mathcal{X}_s
  \big) 
  \, ds
  +
  \int_{ t_0 }^t
  \varphi\big(
    S_{ s, t } X_s, 
    \mathcal{S}_{ s, t } \mathcal{Y}_s
  \big)
  \,
  ds
\\ & +
  \int_{ t_0 }^t
  \varphi\big(
    S_{ s, t } Z_s( \cdot ), 
    \mathcal{S}_{ s, t } \mathcal{X}_s
  \big) 
  \, dW_s
  +
  \int_{ t_0 }^t
  \varphi\big(
    S_{ s, t } X_s, 
    \mathcal{S}_{ s, t } 
    \mathcal{Z}_s( \cdot )
  \big) 
  \,
  dW_s
% \\ & \quad 
  +  
  \sum_{ j \in \mathcal{J} }
  \int_{ t_0 }^t
  \varphi\big( 
    S_{ s, t } Z_s g_j,
    \mathcal{S}_{ s, t } \mathcal{Z}_s 
    g_j
  \big)
  \,
  ds
\nonumber
\end{align}
$ \mathbb{P} $-a.s.\ for
all $ t_0, t \in \mathbbm{I} $
with $ t_0 < t $
and all bounded bilinear mappings
$ 
  \varphi \colon
  \check{H} \times \check{\mathcal{H}}
  \rightarrow
  V
$.
\end{example}

\begin{example}[Mild chain rule]
\label{ex:2}
Let 
$ 
  S \colon 
  \angle
  \rightarrow 
  L( \hat{H}, \check{H} )
$
be an
$ 
\mathcal{B}( \angle ) 
$/$
\mathcal{S}( \hat{H}, \check{H} ) 
$-measurable
mapping satisfying
$
  S_{ t_2, t_3 }
  S_{ t_1, t_2 }
=
  S_{ t_1, t_3 }
$
for all $ t_1, t_2, t_3
\in \mathbbm{I} $
with $ t_1 < t_2 < t_3 $
and 
let 
$ 
  x \colon \mathbbm{I}
  \rightarrow \tilde{H} 
$
and
$ 
  y \colon \mathbbm{I}
  \rightarrow \hat{H} 
$
be two Borel measurable
functions
with
$
    \int_{ \tau }^{ t }
    \left\| S_{ s, t } y_s
    \right\|_{ \tilde{H} }
    ds < \infty
$
and
$
  x_t 
= 
  S_{ \tau, t } \,
  x_{ \tau }
  +
  \int_{ \tau }^t
    S_{ s, t } \,
    y_s
  \, ds
$
for all $ t \in \mathbbm{I} $.
Corollary~\ref{cor:itoauto} 
then shows
\begin{equation}
\label{eq:mildchain}
  \varphi( x_t )
=
  \varphi( S_{ t_0, t} x_{ t_0 } )
  +
  \int_{ t_0 }^t
  \varphi'( S_{ s, t } x_s ) \,
  S_{ s, t } \,
  y_s \, ds
\end{equation}
for all $ t_0, t \in \mathbbm{I} $
with $ t_0 < t $
and all 
$ 
  \varphi \in 
  C^2( \check{H}, V ) 
$.
Equation~\eqref{eq:mildchain}
is somehow a mild
chain rule for
the mild process
$
  x \colon \mathbbm{I} \rightarrow
  \tilde{H}
$.
\end{example}

Let us now concentrate on proofs
of the mild It\^{o} 
formula~\eqref{eq:itoformel_start}.
A central difficulty in order 
to establish an It{\^o} formula 
for the stochastic process 
$ 
  X \colon \mathbbm{I} 
  \times \Omega \rightarrow 
  \tilde{H}
$ 
is that this stochastic process 
is, in general, 
not a standard It\^{o} process
to which 
the standard
It{\^o} formula 
(see, e.g.,
Theorem~4.17 in Section~4.5 in 
Da Prato \citationand\ 
Zabczyk~\cite{dz92}) could be 
applied. 
(Here and below the terminology
``standard It\^{o} formula'' instead
of simply ``It\^{o} formula'' is used
in order to distinguish more clearly
from the above mild It\^{o} formula.)
The stochastic process 
$ 
  X \colon \mathbbm{I} \times 
  \Omega \rightarrow \tilde{H} 
$ 
is, in general, not a standard
It\^{o} process
since it satisfies the 
It{\^o}-Volterra type 
equation~\eqref{eq:mildito}
in which the integrand processes 
$ S_{ s, t } \, Y_s $, 
$ s \in [\tau, t] $, 
and $ S_{ s, t } \, Z_s $, 
$ s \in [\tau, t] $, 
depend explicitly 
on $ t \in \mathbbm{I} $ too
(this was the reason for introducing
the notion of a 
mild It\^{o} process;
see Definition~\ref{propdef}).
Below we present two proofs 
%for 
%Theorem~\ref{thm:ito}
which overcome
this difficulty and which establish
the mild It\^{o} formula~\eqref{eq:itoformel_start}.
Both proofs consider
appropriate transformations
of the mild It\^{o} process 
$ X \colon \mathbbm{I} \times \Omega
\rightarrow \tilde{H} $.
The transformed stochastic processes
are then standard It\^{o} processes
to which the standard 
It\^{o} formula can be applied.
Relating then the transformed
stochastic processes in a
suitable way to the original
mild It\^{o} process 
$ 
  X \colon \mathbbm{I} \times
  \Omega \rightarrow \tilde{H}
$
finally results 
in the mild It\^{o} 
formula~\eqref{eq:itoformel_start}.
The main difference of the two
proofs is the type of transformation
applied to the mild It\^{o} 
process
$
  X \colon \mathbbm{I} \times \Omega
  \rightarrow \tilde{H} 
$.

The first proof makes use of
a transformation in 
Teichmann~\cite{t09}
and
Filipovi{\'c}, Tappe 
\citationand\ 
Teichmann~\cite{ftt10}
(see equations (1.3)
and (1.4) in Teichmann~\cite{t09}
and Section~8 in
Filipovi{\'c}, Tappe 
\citationand\ 
Teichmann~\cite{ftt10}
and see also 
Hausenblas 
\citationand\  
Seidler~\cite{hs01,hs08}).
The first proof does not show
Theorem~\ref{thm:ito} in the
general case but in the case
in which the semigroup of the
mild It\^{o} process is 
pseudo-contractive
(see below for the precise
description of
the used assumptions).
Under this additional assumption,
the semigroup
$ 
  ( S_{ t_1, t_2 } )_{
    (t_1, t_2) \in \angle
  }
$
on the Hilbert space $ \hat{H} $
can be dilated to a group 
$ 
  ( \mathcal{U}_t )_{ t \in \mathbb{R} }
$
on
a larger Hilbert space
(see, e.g.,
Sz\"{o}kefalvi-Nagy~\cite{na53,na54} 
and Theorem~I.81 in 
Sz\"{o}kefalvi-Nagy 
\citationand\ 
Foia{\lfhook{s}}~\cite{naf70}
for the so-called dilations of the unitary theorem).
On this larger Hilbert space,
the mild It\^{o} process~\eqref{eq:mildito} can
be transformed into a
standard It\^{o} process
by -- roughly speaking -- multiplying with
$ \mathcal{U}_{ - t } $
for $ t \in \mathbbm{I} $.
Next the standard It\^{o} formula
can be applied to the transformed
standard It\^{o} process.
Relating this transformed 
standard It\^{o}
process then in a suitable way
to the original mild It\^{o} process
finally results
in the mild It\^{o} 
formula~\eqref{eq:itoformel_start}.

The second proof 
establishes
Theorem~\ref{thm:ito} in the general
case.
It makes use of an idea
in Conus 
\citationand\ 
Dalang~\cite{cd08}
and
Conus~\cite{c08}
(see Section~6 in 
Conus \citationand\ 
Dalang~\cite{cd08}
and 
equations~(1.7)
and (7.6) in 
Conus~\cite{c08}
and see also 
Section~3 in 
Debussche 
\citationand\  
Printems~\cite{dp09},
Theorem~4 in 
Lindner \citationand\ 
Schilling~\cite{ls10}
and Theorem~3.1 in
Kov{\'a}cs, Larsson
\citationand\  
Lindgren~\cite{kll11})
and exploits a more elementary
transformation.
Roughly speaking, the
mild It\^{o} process
$ X \colon \mathbbm{I} \times \Omega
\rightarrow \tilde{H} $
is transformed in the second
proof by multiplying
with $ S_{ t, T } $ for 
$ t \in [\tau, T) $
with a fixed $ T \in \mathbbm{I} $
(compare that 
the transformation in first proof
is based on multiplying with the group
at the negative time value $ - t $).
Since
$ T - t > 0 $, the transformed
process of the 
$ \tilde{H} $-valued process
$ X \colon \mathbbm{I} \times \Omega
\rightarrow \tilde{H} $
enjoys values in 
$ \tilde{H} $ too
(this is in contrast to the first
proof where the transformed process
of $ X \colon \mathbbm{I} \times
\Omega \rightarrow \tilde{H} $ 
takes values in a larger
Hilbert space in which $ \tilde{H} $
is continuously embedded).
Nonetheless, as in the first proof,
the transformed 
stochastic process is 
a standard It\^{o} process
to which the standard It\^{o} formula
can be applied.
Relating the transformed
standard It\^{o} process
in a suitable way
to the original mild It\^{o}
process then again results 
in the mild It\^{o} 
formula~\eqref{eq:itoformel_start}.

Both proofs thus essentially consist
of three steps: a 
{\it transformation},
an {\it application of 
the standard
It\^{o} formula} 
and the use of a
suitable {\it relation} 
of the transformed 
standard It\^{o}
process and the original 
mild It\^{o} process.
The second proof also uses 
the following simple result.

\begin{lemma}
\label{lem:easy}
Let $ Y, Z \colon \mathbbm{I}
\times \Omega \rightarrow [0,\infty) $
be two product measurable
stochastic processes with
$
  \mathbb{P}\big[
    Y_t = Z_t
  \big] = 1
$
for all $ t \in \mathbbm{I} $
and with
$
  \mathbb{P}\big[
    \int_{ \mathbbm{I} } Y_s \, ds < \infty
  \big] = 1
$.
Then
$
  \mathbb{P}\big[
    \int_{ \mathbbm{I} } Z_s \, ds < \infty
  \big] = 1
$.
\end{lemma}

The proof of Lemma~\ref{lem:easy}
is clear and therefore omitted.
Instead the first proof
of Theorem~\ref{thm:ito}
in the special case of
a pseudo-contractive semigroup
is now presented.

\begin{proof}
{\it Proof
of Theorem~\ref{thm:ito}
in the case 
where
the partial Fr\'{e}chet 
derivatives
$ \partial_1 \varphi $,
$ \partial_2 \varphi $
and
$ \partial_2^2 \varphi $
of $ \varphi $
are globally bounded,
where
$ Y \colon \mathbbm{I} \times \Omega
\rightarrow \hat{H} 
$ 
and 
$ Z \colon \mathbbm{I} \times \Omega
\rightarrow HS( U_0, \hat{H} ) 
$ 
have
continuous sample paths,
where
$ 
  ( 
    \check{H}, 
    \left< \cdot , \cdot \right>_{ \check{H} }, 
    \left\| \cdot \right\|_{ \check{H} }
  ) 
  =
  ( 
    \tilde{H}, 
    \left< \cdot , \cdot \right>_{ \tilde{H} }, 
    \left\| \cdot \right\|_{ \tilde{H} }
  ) 
  =
  ( 
    \hat{H}, 
    \left< \cdot , \cdot \right>_{ \hat{H} }, 
    \left\| \cdot \right\|_{ \hat{H} }
  ) 
$,
where 
$ U_t \in L( \tilde{H} ) $, 
$ t \in [0, \infty) $, 
is a strongly continuous pseudo-contractive semigroup 
on $ \tilde{H} $
and where
$ 
  S_{ t_1, t_2 } = U_{ ( t_2 - t_1 ) } 
  \in L( \tilde{H} ) 
$
for all $ (t_1, t_2) \in \angle $.}
First, observe that,
under these additional
assumptions,
\eqref{eq:well1}
and \eqref{eq:well2}
are obviously fulfilled.
Moreover, due to Proposition~8.7 
in \cite{ftt10}, 
there exists
a separable $ \mathbb{R} $-Hilbert 
space $ \left( \mathcal{H}, \left< \cdot, \cdot \right>_{ \mathcal{H} }, \left\| \cdot \right\|_{ \mathcal{H} } \right) $
with $ \tilde{H} \subset \mathcal{H} $ and $ \left\| v \right\|_{ \tilde{H} } = \left\| v \right\|_{ \mathcal{H} } $ for
all $ v \in \tilde{H} $ and 
a strongly continuous group 
$ \mathcal{U}_t \in L( \mathcal{H} ) $, 
$ t \in \mathbb{R} $, such that
\begin{equation}
\label{eq:put}
  U_t( v )
  =
  P\big( 
    \mathcal{U}_t( v ) 
  \big)
\end{equation}
for all 
$ v \in \tilde{H} \subset \mathcal{H} $ 
and all $ t \in [0,\infty) $ 
where
$ 
  P \colon 
  \mathcal{H} \rightarrow \tilde{H} 
$ 
is the orthogonal projection from $ \mathcal{H} $ to $ \tilde{H} $.
In this first proof
the mild It\^{o} process
$
  X \colon \mathbbm{I} \times \Omega
  \rightarrow \tilde{H}
$
is now transformed into
a standard It\^{o} process
by - roughly speaking - multiplying
with $ \mathcal{U}_{ - t } $
for $ t \in \mathbbm{I} $.
In a more concrete setting
this transformation
has been proposed 
in Teichmann~\cite{t09} and 
Filipovi{\'c}, Tappe 
\citationand\ 
Teichmann~\cite{ftt10};
see equations~(1.3) and (1.4) in
Teichmann~\cite{t09} and
Section~8 in Filipovi{\'c}, 
Tappe 
\citationand\ 
Teichmann~\cite{ftt10}
and see also 
Hausenblas \citationand\ 
Seidler~\cite{hs01,hs08}.
Let us now go into details.
Let
$ 
  \bar{X} \colon \mathbbm{I} \times 
  \Omega \rightarrow \mathcal{H} 
$ 
be the up to
indistinguishability 
unique adapted stochastic process with continuous
sample paths satisfying
\begin{equation}
\label{eq:barX}
  \bar{X}_t
  =
  X_{ \tau }
  +
  \int_{ \tau }^t
  \mathcal{U}_{-s} \, Y_s \, ds
  +
  \int_{ \tau }^t
  \mathcal{U}_{-s} \, Z_s \, dW_s
\end{equation}
$ \mathbb{P} 
$-a.s.\ for 
all $ t \in \mathbbm{I} $ 
({\it Transformation};
see also equation~(1.4)
in \cite{t09}
and equation~(8.6)
in \cite{ftt10}).
Next observe that
the identity
$ 
  X_t =
  \mathcal{P}\big(
    \mathcal{U}_t(
      \bar{X}_t
    )
  \big)
$
$ \mathbb{P} $-a.s.\ (see 
also 
Theorem~8.8 in 
\cite{ftt10})
and
the standard
It{\^o} formula 
in infinite dimensions (see 
Theorem~2.4 in Brze\'{z}niak,
Van Neerven, Veraar 
\citationand\ 
Weis~\cite{bvvw08}) 
applied to the test function
$ \varphi\big( s,
  P\big( \mathcal{U}_t( v ) \big) 
\big) $,
$ s \in [t_0,t] $,
$ v \in \mathcal{H} $,
give
\begin{equation}
\label{eq:important}
\begin{split}
&
  \varphi( t,
    X_t  
  )
=
  \varphi\big( t,
    P\big( \mathcal{U}_t( \bar{X}_t ) \big) 
  \big)
=
  \varphi\big( t_0,
    P\big( \mathcal{U}_t( \bar{X}_{ t_0 } ) 
    \big) 
  \big)
  +
  \int_{t_0}^t
  (\partial_1 \varphi)\big( s, 
    P\big( \mathcal{U}_t( \bar{X}_{ s } ) 
    \big) 
  \big)
  \,
  ds
\\ & \quad
  +
  \int_{t_0}^t
  (\partial_2 \varphi)\big( s,
    P\big( \mathcal{U}_t( \bar{X}_{ s } ) 
    \big) 
  \big)
  \,
  P \; 
  \mathcal{U}_{(t-s)} \,
  Y_s \, ds
  +
  \int_{t_0}^t
  (\partial_2 \varphi)\big( s,
    P\big( \mathcal{U}_t( \bar{X}_{ s } ) 
    \big) 
  \big)
  \,
  P \; 
  \mathcal{U}_{(t-s)} \,
  Z_s \, dW_s
\\ & \quad
  +
  \frac{1}{2}
  \sum_{ j \in \mathcal{J} }
  \int_{t_0}^t
  (\partial_2^2 \varphi)\big( s,
    P\big( \mathcal{U}_t( \bar{X}_{ s } ) 
    \big) 
  \big) \!
  \left( 
    P \, \mathcal{U}_{(t-s)} Z_s g_j, 
    P \, \mathcal{U}_{(t-s)} Z_s g_j 
  \right)  ds
\end{split}
\end{equation}
$ \mathbb{P} $-a.s.\ for 
all $ t_0, t \in \mathbbm{I} $
with $ t_0 \leq t $
and all
$
  \varphi \in
  C^{ 1, 2 }(
    \mathbbm{I} \times \tilde{H}, V
  )
$
({\it Application of the
standard It\^{o} formula}).
Next note that
equation~\eqref{eq:put} gives
\begin{equation}
\label{eq:backtransform0}
\begin{split}
  P\big( \mathcal{U}_t( \bar{X}_{ s } ) 
  \big) 
& =
  P\big( \mathcal{U}_t( X_{ \tau } ) 
  \big) 
  +
  \int_{ \tau }^s
  P \;
  \mathcal{U}_{(t-u)} \, Y_u \, du
  +
  \int_{ \tau }^s
  P \;
  \mathcal{U}_{(t-u)} \, Z_u \, 
  dW_u
\\ & =
  S_{\tau, t} \, X_{ \tau }  
  +
  \int_{ \tau }^s
  S_{u,t} \, Y_u \, du
  +
  \int_{ \tau }^s
  S_{ u, t } \, Z_u \, 
  dW_u
\\ & =
  S_{ s, t }
  \left(
    S_{ \tau, s } \, X_{ \tau }
  +
    \int_{ \tau }^s
    S_{ u, s } \,
    Y_u \, du
  +
    \int_{ \tau }^s
    S_{ u, s } \,
    Z_u \, dW_u
  \right)
  =
  S_{ s, t } \, X_s
\end{split}
\end{equation}
$ \mathbb{P} $-a.s.\ for
all $ s, t \in \mathbbm{I} $
with $ s \leq t $
({\it Relation} of the transformed
standard It\^{o} process 
$ 
  \bar{X} 
  \colon \mathbbm{I} \times \Omega
  \rightarrow \mathcal{H}
$
and the original mild It\^{o}
process 
$ 
  X \colon \mathbbm{I} \times
  \Omega \rightarrow \tilde{H}
$). 
Using
\eqref{eq:put}
and
\eqref{eq:backtransform0}
in
\eqref{eq:important}
finally shows 
\eqref{eq:itoformel_start}.
The proof 
is thus completed. 
\end{proof}

In the next step 
the proof
of Theorem~\ref{thm:ito}
in the general case is 
given.
Above an outline of this second proof
is given.

\begin{proof}[Proof
of Theorem~\ref{thm:ito}]
In this second proof
the time variable $ t \in \mathbbm{I} $
within the integrand processes
in \eqref{eq:mildito} is fixed
and then, the standard It\^{o} formula
is applied to the resulting 
standard It\^{o} process.
In a more concrete setting
this trick has been
proposed in
Conus \citationand\ 
Dalang~\cite{cd08}
and 
Conus~\cite{c08};
see Section~6 in 
Conus \citationand\ Dalang~\cite{cd08}
and
equations~(1.7) and (7.6) in 
Conus~\cite{c08} and see
also 
Section~5 in 
Lindner \citationand\ 
Schilling~\cite{ls10} and
Section~3 in
Kovacs, Larsson \citationand\ 
Lindgren~\cite{kll11}.
Another related result can 
be found in Section~3
in Debussche \citationand\  
Printemps~\cite{dp09}.
Let us now go into details.
Let 
$ 
  \bar{X}^{t} \colon 
  [ \tau, t ] \times \Omega 
  \rightarrow 
  \check{H} 
$,
$
  t \in 
  \mathbbm{I} \cap
  ( \tau, \infty )
$,
be a family of adapted
stochastic processes 
with continuous sample paths
given by
\begin{equation}
\label{eq:semiX}
  \bar{X}_u^t
  =
  S_{ \tau, t } \, X_{ \tau }
  +
  \int_{ \tau }^u
  S_{ s, t } \, Y_s \, ds
  +
  \int_{ \tau }^u
  S_{ s, t } \, Z_s \, dW_s
\end{equation}
$ \mathbb{P} $-a.s.~for 
all $ u \in [\tau, t] $
and all
$ 
  t \in 
  \mathbbm{I} \cap
  ( \tau, \infty )
$
({\it Transformation};
see also 
Section~6 in \cite{cd08},
Section~7 in \cite{c08},
Section~5 in \cite{ls10} and
Section~3 in \cite{kll11}).
Note that 
the assumptions
$
  \mathbb{P}\big[
    \int_{ \tau }^t
    \left\| S_{ s, t } Y_s \right\|_{ 
      \check{H} 
    } 
    ds < \infty
  \big] = 1
$
and
$
  \mathbb{P}\big[
    \int_{ \tau }^t
    \left\| S_{ s, t } Z_s \right\|_{ 
      HS( U_0, \check{H} )
    }^2 
    ds < \infty
  \big] = 1
$
for all 
$ 
  t \in 
  \mathbbm{I} 
$
(see Definition~\ref{propdef})
ensure that
$ 
  \bar{X}^t \colon
  [\tau,t] \times \Omega
  \rightarrow \check{H} 
$,
$ 
  t \in \mathbbm{I} \cap
  ( \tau, \infty )
$,
in \eqref{eq:semiX}
are indeed
well defined 
adapted
stochastic processes
with continuous sample paths.
In the next step 
the 
continuity
of the partial derivatives
of
$ 
  \varphi \colon 
  \mathbbm{I} \times
  \check{H}
  \rightarrow V
$,
the continuity
of the sample paths
of 
$
  \bar{X}^t
  \colon [\tau,t]
  \times \Omega
  \rightarrow \check{H}
$
and again
the assumptions
$
  \mathbb{P}\big[
    \int_{ \tau }^t
    \left\| S_{ s, t } Y_s \right\|_{ 
      \check{H} 
    } 
    ds < \infty
  \big] = 1
$
and
$
  \mathbb{P}\big[
    \int_{ \tau }^t
    \left\| S_{ s, t } Z_s \right\|_{ 
      HS( U_0, \check{H} )
    }^2 
    ds < \infty
  \big] = 1
$
in Definition~\ref{propdef}
imply
\begin{equation}
\label{eq:well1c}
  \mathbb{P}\!\left[
    \int_{ t_0 }^t
    \left\|
      (\partial_2 \varphi)( s, \bar{X}^t_s )
      S_{ s, t } Y_s
    \right\|_V
    + 
    \left\|
      (\partial_2 \varphi)( s, \bar{X}^t_s )
      S_{ s, t } Z_s
    \right\|_{ HS(U_0, V ) }^2 
    ds
    < \infty 
  \right] = 1
\end{equation}
and
\begin{equation}
\label{eq:well2c}
  \mathbb{P}\!\left[
    \int_{ t_0 }^t
    \left\|
      (\partial_1 \varphi)( s, \bar{X}^t_s )
    \right\|_V
    +
    \left\|
      (\partial_2^2 \varphi)( s, \bar{X}^t_s )
    \right\|_{ L^{(2)}( \check{H}, V ) }
    \left\|
      S_{ s, t } Z_s
    \right\|_{ HS(U_0, \check{H} ) }^2 
    ds
    < \infty 
  \right] = 1 
\end{equation}
for all
$ 
  t_0 \in [ \tau, t ] 
$,
$ 
  t \in 
  \mathbbm{I} \cap 
  ( \tau, \infty ) 
$
and all
$
  \varphi \in C^{1,2}(
    \mathbbm{I} \times \check{H}, V
  )
$.
Moreover, 
the identity
$
  X_t =
  \bar{X}_t^t
$
$ \mathbb{P} $-a.s.\ and
the standard
It{\^o} formula 
(see
Theorem~2.4 in Brze\'{z}niak,
Van Neerven, Veraar 
\citationand\ 
Weis~\cite{bvvw08})
give
\begin{equation}
\label{eq:itoformel2inProof}
\begin{split}
&
  \varphi( t, X_{ t } )
=
  \varphi( t, \bar{X}_t^{ t } )
=
  \varphi( t_0, \bar{X}_{ t_0 }^{ t } )
  +
  \int_{t_0}^t
  (\partial_1 \varphi)( s, \bar{X}_{ s }^{ t } )     
  \, ds
  +
  \int_{t_0}^t
  (\partial_2 \varphi)( s, \bar{X}_{ s }^{ t } ) \,
  S_{ s, t } \, Y_s \, ds
\\&\quad+
  \int_{t_0}^t
  (\partial_2 
  \varphi)( s, \bar{X}_{ s }^{ t } ) \,
  S_{ s, t } \, Z_s \, dW_s
  +
  \frac{1}{2}
  \sum_{ j \in \mathcal{J} }
  \int_{t_0}^t
  (\partial_2^2 \varphi)( s,
    \bar{X}_{ s }^{ t }  
  )
  \left(
    S_{ s, t } \, Z_s \, g_j,
    S_{ s, t } \, Z_s \, g_j
  \right) ds
\end{split}
\end{equation}
$ \mathbb{P} $-a.s.\ for all
$ 
  t_0, t \in \mathbbm{I}
$
with $ t_0 < t $
and all
$
  \varphi \in C^{1,2}(
    \mathbbm{I} \times \check{H}, V
  )
$
({\it Application 
of the 
standard It\^{o} formula};
see also Section~6
in \cite{cd08},
equations~(1.7) and (7.6) 
in \cite{c08},
Theorem~4 in \cite{ls10}
and Theorem~3.1 in
\cite{kll11}).
Equation~\eqref{eq:itoformel2inProof} 
is an expansion formula 
for the stochastic processes
$ \varphi( t, X_t ) $,
$ 
  t \in \mathbbm{I} \cap
  ( \tau, \infty )
$,
for
$
  \varphi \in C^{1,2}(
    \mathbbm{I} \times \check{H}, V
  )
$.
Nevertheless, this formula 
seems to be of limited use 
since the integrands 
in~\eqref{eq:itoformel2inProof} contain
the transformed stochastic processes 
$ \bar{X}_s^t $,
$ s \in [t_0,t] $,
$ t_0, t \in \mathbbm{I} $,
$ t_0 < t $,
instead of 
the mild It\^{o} process
$ X_s $, 
the mild drift 
$ Y_s $
and the mild diffusion 
$ Z_s $
for $ s \in [t_0,t] $,
$ t_0, t \in \mathbbm{I} $,
$ t_0 < t $,
only. 
However, a key 
observation 
here is to 
exploit the elementary
identity
\begin{equation}
\label{eq:fact}
\begin{split}
  \bar{X}^{ t }_s
&=
    S_{ \tau, t } \, X_{ \tau }
  +
    \int_{ \tau }^s
    S_{ u, t } \,
    Y_u \, du
  +
    \int_{ \tau }^s
    S_{ u, t } \,
    Z_u \, dW_u
\\&=
  S_{ s, t }
  \left(
    S_{ \tau, s } \, X_{ \tau }
  +
    \int_{ \tau }^s
    S_{ u, s } \,
    Y_u \, du
  +
    \int_{ \tau }^s
    S_{ u, s } \,
    Z_u \, dW_u
  \right)
  =
  S_{ s, t } \,
  X_s
\end{split}
\end{equation}
$ \mathbb{P} $-a.s.\ for 
all $ s, t \in \mathbbm{I} $
with $ s < t $ 
in 
equation~\eqref{eq:itoformel2inProof}
({\it Relation} of the transformed
standard It\^{o} processes
$ 
  \bar{X}^t
  \colon [\tau,t] \times \Omega
  \rightarrow \mathcal{H}
$,
$ 
  t \in \mathbbm{I} 
  \cap ( \tau, \infty )
$,
and the original mild It\^{o}
process 
$ 
  X \colon \mathbbm{I} \times
  \Omega \rightarrow \tilde{H}
$). 
This enables us to obtain a
closed formula for the stochastic
processes 
$ \varphi( t, X_t ) $,
$ 
  t \in \mathbbm{I} \cap
  ( \tau, \infty )
$,
for
$
  \varphi \in C^{1,2}(
    \mathbbm{I} \times \check{H}, V
  )
$.
More precisely,
\eqref{eq:fact},
\eqref{eq:well1c},
\eqref{eq:well2c} 
and Lemma~\ref{lem:easy}
imply
\eqref{eq:well1} and \eqref{eq:well2}.
Putting~\eqref{eq:fact} into~\eqref{eq:itoformel2inProof} 
then gives
\eqref{eq:itoformel_start}.
The proof of Theorem~\ref{thm:ito}
is thus completed.
\end{proof}

Let us close this section on mild
stochastic calculus with a remark
on possible generalizations.
\begin{remark}
\label{rem:moregeneral}
Note that here
mild It\^{o} processes,
mild drifts and mild diffusions
with values in 
separable Hilbert spaces
are considered. Instead one could
develop the above notions and
the above mild It\^{o} formula for 
stochastic processes with values 
in an appropriate class
of possibly non-separable Banach spaces too.
Indeed, the standard It\^{o} formula
also holds for stochastic processes
with values in UMD Banach 
spaces (see Theorem~2.4 in 
Brze\'{z}niak, Van Neerven, Veraar
\citationand\  
Weis~\cite{bvvw08}).
Details on the stochastic 
integration
in UMD Banach spaces 
can be found in
Van Neerven, 
Veraar \citationand\ 
Weis~\cite{vvw07,vvw08}
and in the references therein. 
Another possible 
generalization
is to consider more general integrators
than the cylindrical Wiener process
$ ( W_t )_{ t \in \mathbbm{I} } $.
This leads to the concept
of a {\it mild semimartingale} instead
of a mild It\^{o} process
in Definition~\ref{propdef}.
In particular, the fourth 
integral in the
mild It\^{o} formula~\eqref{eq:itoformel_start}
then needs to be replaced by
an integral involving the
quadratic
variation of the
driving noise process.
\end{remark}

\section{Stochastic
partial differential 
equations (SPDEs)}
\label{sec:appl}

\subsection{Setting and
assumptions}
\label{sec:setting}

Throughout this section
suppose that the
following setting and the following assumptions are fulfilled.
Let $ T \in (0,\infty) $
be a real number,
let 
$ \left( \Omega, \mathcal{F}, \mathbb{P} \right) $ 
be a probability space 
with a normal filtration
$ ( \mathcal{F}_t )_{ t \in [0,\infty) } $
and let
$ 
  \left( 
    H, 
    \left< \cdot , \cdot \right>_H, 
    \left\| \cdot \right\|_H
  \right) 
$,
$ 
  \left( 
    U, 
    \left< \cdot , \cdot \right>_U, 
    \left\| \cdot \right\|_U 
  \right) 
$ 
and
$
  \left( 
    V, 
    \left< \cdot , \cdot \right>_V, 
    \left\| \cdot \right\|_V 
  \right) 
$
be separable
$ \mathbb{R} $-Hilbert spaces.
In addition, let 
$ Q \colon U \rightarrow U $ 
be 
a bounded nonnegative symmetric linear
operator and 
let $ ( W_t )_{ t \in [0,\infty) } $
be a cylindrical $ Q $-Wiener 
process with respect 
to $ ( \mathcal{F}_t )_{ t \in [0,\infty) } $.
Moreover, by
$ 
  \left( 
    U_0, 
    \left< \cdot , \cdot \right>_{ U_0 }, 
    \left\| \cdot \right\|_{ U_0 }
  \right) 
$ 
the separable
$\mathbb{R}$-Hilbert space 
with
$ U_0 = Q^{ 1/2 }( U ) $
and
$ 
  \| u \|_{ U_0 } 
  = \| Q^{ - 1/2 }( u ) \|_U 
$ 
for all $ u \in U_0 $
is denoted.
\begin{assumption}[Linear operator A]\label{semigroup}
Let $ A \colon D(A) \subset H
\rightarrow H $ be a generator
of a strongly continuous analytic
semigroup 
$ e^{ A t } \in L(H) $,
$ t \in [0,\infty) $.
%,
%on $ H $.
\end{assumption}
Let $ \eta \in [0,\infty) $ be a nonnegative real number such that
$
  \sigma(A)
  \subset
  \{ 
    \lambda \in \mathbb{C}
    \colon
    \text{Re}( \lambda ) < \eta
  \}
$ where 
$ \sigma(A) \subset \mathbb{C} 
$ denotes 
as usual the spectrum
of the linear
operator 
$ A \colon D(A) \subset
H \rightarrow H $. 
Such a real number exists
since $ A $ is assumed to be a generator
of a strongly continuous semigroup
(see Assumption~\ref{semigroup}).
By 
$ 
  H_r := 
  D\!\left( 
    \left( \eta - A \right)^r 
  \right) 
$
equipped with 
the norm
$ 
  \left\| v \right\|_{ H_r }
  := 
  \left\| 
    \left( \eta - A \right)^r \! v 
  \right\|_H 
$
for all $ v \in H_r $ 
and all $ r \in \mathbb{R} $, 
the $ \mathbb{R} $-Hilbert spaces
of domains of fractional powers 
of the linear operator 
$ \eta - A \colon D(A) \subset H
\rightarrow H $ are denoted
(see, e.g., Subsection~11.4.2
in Renardy \citationand\ 
Roggers~\cite{rr93}).
\begin{assumption}[Drift term $F$]\label{drift}
Let $ \alpha, \gamma \in \mathbb{R} $
be real numbers with 
$ \gamma - \alpha < 1 $
and let
$ F \colon H_{ \gamma } 
\rightarrow H_{ \alpha } $ 
be globally
Lipschitz continuous.
\end{assumption}
\begin{assumption}[Diffusion term $B$]\label{diffusion}
Let $ \beta \in \mathbb{R} $ be a
real number with 
$ \gamma - \beta < \frac{1}{2} $
and 
let $ B \colon H_{ \gamma } 
\rightarrow HS(U_0,H_{ \beta }) $ 
be globally Lipschitz continuous.
\end{assumption}
\begin{assumption}[Initial value $\xi$]\label{initial}
Let 
$ p \in [2,\infty) $
be a real number and
let 
$ 
  \xi \colon \Omega 
  \rightarrow H_{ \gamma } 
$ 
be an
$ \mathcal{F}_{0} 
$/$ \mathcal{B}\left( 
H_{ \gamma } \right) 
$-measurable 
mapping
with 
$ 
\mathbb{E}\big[ 
  \| \xi \|^p_{ H_{ \gamma } } 
\big] < \infty 
$.
\end{assumption}
%
%
%
%Below 
%the following well known
%estimate on the analytic
%semigroup $ e^{ A t } \in L(H) $, 
%$ t \in [0,\infty) $.
%is frequently 
%used (see, e.g., Lemma 11.36 in
%Renardy and Roggers~\cite{rr93}).
%To be more precise,
%Assumption~\ref{semigroup}
%implies the existence
%of real numbers
%$ c_r \in [1,\infty) $, $ r \in [0,1] $,
%such that
%\begin{equation}
%  \big\|
%    ( \eta - A )^r
%    e^{ A t }
%  \big\|_{ L(H) }
%  \leq
%  c_r \, t^{ - r }
%\qquad
%  \text{and}
%\qquad
%  \big\|
%    ( \eta - A )^{ - r }
%    ( e^{ A t } - I )
%  \big\|_{ L(H) }
%  \leq
%  c_r \,
%  t^r
%\end{equation}
%for all $ t \in (0,T] $
%and all $ r \in [0,1] $.
Furthermore, 
similar as
in Section~\ref{sec:mildcalc},
let
$ \angle \subset [0,T]^2 $
be defined
through
$
  \angle :=
  \left\{ 
    (t_1, t_2) \in [0,T]^2 \colon
    t_1 < t_2
  \right\}
$.
In addition to the above assumptions,
the following notations will be used
in the remainder of this
article.
For two 
$ \mathbb{R} $-Banach spaces
$ ( V_1, \left\| \cdot \right\|_{ V_1 } ) $
and
$ 
  ( V_2, \left\| \cdot \right\|_{ V_2 } ) 
$
and real numbers
$ 
  n \in \{ 0, 1, 2, \ldots \} 
$
and 
$ q \in [0,\infty) $
define
$
  \left\| v 
  \right\|_{ L^{ (0) }( V_1, V_2 ) }
  :=
  \left\| v 
  \right\|_{ V_2 }
$
for every
$ v \in V_1 $,
define
\begin{equation}
  \| \varphi \|_{
    G^n_q( V_1, V_2 )
  }
  :=
  \sum_{ i = 0 }^{ n - 1 }
  \| \varphi^{(i)}(0) \|_{
    L^{(i)}( V_1, V_2 )
  }
  +
  \sup_{ v \in V_1 }
  \left(
  \frac{
    \| \varphi^{(n)}(v) \|_{ 
      L^{(n)}( V_1, V_2 )
    }
  }{
    \left(
      1 + \| v \|_{ V_1 }
    \right)^q
  }
  \right)
  \in [0,\infty] ,
\end{equation}
\begin{equation}
  \| \varphi \|_{
    \text{Lip}^{ n + 1 }_q( V_1, V_2 )
  }
  :=
  \sum_{ i = 0 }^{ n }
  \| \varphi^{(i)}(0) \|_{
    L^{(i)}( V_1, V_2 )
  }
  +
  \sup_{ 
    \substack{
      v, w \in V_1 
    \\
      v \neq w
    }
  }
  \left(
  \frac{
    \| 
      \varphi^{(n)}(v) -
      \varphi^{(n)}(w)
    \|_{ 
      L^{(n)}( V_1, V_2 )
    }
  }{
    \left(
      1 + 
      \max( 
        \| v \|_{ V_1 }, 
        \| w \|_{ V_1 }
      )
    \right)^q
    \left\|
      v - w
    \right\|_{ V_1 }
  }
  \right)
  \in [0,\infty] 
\end{equation}
and
\begin{equation}
\label{eq:defCnLip}
  \| \varphi \|_{
    C^n_{ Lip }( V_1, V_2 )
  }
  :=
  \| \varphi(0) \|_{ V_2 }
  +
  \sum_{ i = 1 }^{ n }
  \| \varphi^{(i)} \|_{ 
    L^{ \infty }( V_1, L^{(i)}(V_1,V_2) ) 
  }
  +
  \sup_{ 
    \substack{
      v, w \in V_1 
    \\
      v \neq w
    }
  }
  \left(
  \frac{
    \| 
      \varphi^{ (n) }(v) -
      \varphi^{ (n) }(w)
    \|_{ 
      L^{(n)}( V_1, V_2 )
    }
  }{
    \left\|
      v - w
    \right\|_{ V_1 }
  }
  \right)
  \in [0,\infty] 
\end{equation}
for every 
$   
  \varphi \in C^n( V_1, V_2 )
$,
define
$
  G^n_q( V_1, V_2 )
  :=
  \{ 
    \varphi \in C^n( V_1, V_2 ) \colon
    \| \varphi \|_{ 
      G^n_{q}( V_1, V_2 ) 
    }
    < \infty 
  \}
$,
define
$
  \text{Lip}^{ n + 1 }_q( V_1, V_2 )
$
$
  :=
  \{ 
    \varphi \in C^n( V_1, V_2 ) \colon
    \| \varphi \|_{ 
      \text{Lip}^{ n + 1 }_{ q }( V_1, V_2 ) 
    }
    < \infty 
  \}
$
and define
$
  C_{ Lip }^n( V_1, V_2 )
  :=
  \{ 
    \varphi \in C^n( V_1, V_2 ) \colon
    \| \varphi \|_{ C^n_{ Lip }( V_1, V_2 ) }
    < \infty 
  \}
$
and note that
$
  \left\| \varphi \right\|_{ 
    G^{ m }_q( V_1, V_2 )
  }
  =
  \left\| \varphi \right\|_{
    \text{Lip}^{ m }_q( V_1, V_2 )
  }
$
for every 
$   
  \varphi \in C^{ m }( V_1, V_2 )
$
and every
$ m \in \N $.
Let us collect a few simple properties
of the defined objects.
More precisely, observe that
\begin{align}
\label{eq:Gnq1}
  \| \varphi \|_{
    G^n_q( V_1, V_2 )
  }
  & =
  \sum\nolimits_{ i = 0 }^{ n - 1 }
  \| \varphi^{(i)}(0) \|_{
    L^{(i)}( V_1, V_2 )
  }
  +
  \| \varphi^{(n)} \|_{
    G^0_{ q }( V_1, V_2 )
  }
\\
\label{eq:Gnq2}
  \|
    \varphi
  \|_{ 
    G^{ 0 }_{ q + n }( V_1, V_2 )
  }
& \leq
  \|
    \varphi
  \|_{ 
    G^{ n - k }_{ q + k }( V_1, V_2 )
  }
  \leq
  \|
    \varphi
  \|_{ 
    G^n_{ q }( V_1, V_2 )
  } ,
\\
\label{eq:growthestimateGnq}
  \| 
    \varphi^{ ( k ) }(v)
  \|_{ 
    L^{ ( k ) }( V_1, V_2 )   
  }
& \leq
  \| 
    \varphi
  \|_{ 
    G^n_q( V_1, V_2 )
  }
  \left( 1 + \| v \|_{ V_1 } \right)^{ 
    ( q + n - k ) 
  } 
\\
\label{eq:lipestimateGnq}
  \|
    \varphi^{ (k) }(v)
    -
    \varphi^{ (k) }(w)
  \|_{ 
    L^{ (k) }( V_1, V_2 )
  }
& \leq
  \| 
    \varphi 
  \|_{
    \mathrm{Lip}^{ n + 1 }_{ q }( V_1, V_2 )
  }
  \,
  \big(
    1 + \max( \| v \|_{ V_1 },  \| w \|_{ V_1 } 
    )
  \big)^{
    ( q + n - k )
  }
  \,
  \| v - w \|_{ V_1 } 
\end{align}
for all 
$ v, w \in V_1 $,
$ \varphi \in C^n( V_1, V_2 ) $,
$ 
  k \in \{ 0, 1, \dots, n \}
$,
$ n \in \{ 0, 1, \dots \} $,
$ q \in [0,\infty) $
and all
$ \mathbb{R} $-Banach spaces
$ ( V_1, \left\| \cdot \right\|_{ V_1 } ) $
and
$ 
  ( V_2, \left\| \cdot \right\|_{ V_2 } ) 
$.
Moreover, 
note 
for all $ n \in \{ 0, 1, 2, \dots \} $,
$ q \in [0,\infty) $
and all
$ \mathbb{R} $-Banach spaces
$ ( V_1, \left\| \cdot \right\|_{ V_1 } ) $
and
$ 
  ( V_2, \left\| \cdot \right\|_{ V_2 } ) 
$
that the pairs
$
  ( 
    G^n_q(V_1, V_2) ,
    \left\| \cdot 
    \right\|_{ G^n_q( V_1, V_2 ) }
  )
$,
$
  ( 
    \text{Lip}^{n+1}_q(V_1, V_2) ,
    \left\| \cdot 
    \right\|_{ 
      \text{Lip}^{n+1}_q( V_1, V_2 ) 
    }
  )
$
and
$
  ( 
    C^n_{ Lip }(V_1, V_2) ,
$
$
    \left\| \cdot 
    \right\|_{ 
      C^{ n }_{ Lip }( V_1, V_2 ) 
    }
  ) 
$
are 
$ \mathbb{R} $-Banach
spaces
with
$
  G^{ n + 1 }_q(V_1, V_2) 
  \subset
  \text{Lip}^{n+1}_q(V_1, V_2) 
  \subset
  G^{ n }_{ q + 1 }(V_1, V_2) 
$
continuously.
More function spaces
of similar type can be found
in D\"{o}rsek \citationand\  
Teichmann~\cite{dt10}.

\subsection{Solution processes
of SPDEs}
\label{sec:solutionSPDE}

The following proposition
shows that
the setting in 
Section~\ref{sec:setting}
ensures that the
SPDE~\eqref{eq:SPDE} below
admits an up to modifications
unique mild solution process.
It is similar to special cases of 
Theorem~4.3 
in Brze\'{z}niak~\cite{b97b}
and
Theorem~6.2 in
Van Neerven, Veraar 
\citationand\ 
Weis~\cite{vvw08}. 
Its proof is clear and therefore
omitted.

\begin{prop}
\label{prop:prop}
Assume that the setting
in Section~\ref{sec:setting}
is fulfilled.
Then there exists an
up to modifications unique
predictable stochastic
process 
$
  X \colon [0,T] \times 
  \Omega
  \rightarrow H_{ \gamma }
$
which fulfills
$
  \sup_{ t \in [0,T] }
  \mathbb{E}\big[
    \| X_t \|_{ H_{ \gamma } }^p
  \big]
  < \infty
$
and
\begin{equation}
\label{eq:SPDE}
  X_t = e^{ A t } \,
  \xi
  +
  \int_{0}^t 
    e^{ A (t - s) }
    F( X_s ) \,
  ds
  +
  \int_{0}^t 
    e^{ A (t - s) }
    B( X_s ) \,
  dW_s
\end{equation}
$\mathbb{P}$-a.s.\ for 
all $ t \in [0,T] $.
In addition, we have
$
  X
  \in
  \cap_{ r \in (-\infty, \gamma] }
  \,
  C^{ \min( \gamma - r, 1/2 ) }\!\left(
    [0,T] ,
    L^p( \Omega; H_{ r } )
  \right)
$.
\end{prop}

Proposition~\ref{prop:prop},
in particular, ensures that
the mild solution process
$ X \colon [0,T] \times \Omega
\rightarrow H_{ \gamma } $
of the SPDE~\eqref{eq:SPDE}
is a mild It\^{o} process
with semigroup 
$ 
  e^{ A (t_2 - t_1) } 
  \in 
  L( 
    H_{ \min( \alpha, \beta. \gamma ) } ,
    H_{ \gamma }
  )
$,
$ (t_1, t_2) \in \angle $,
with mild drift 
\begin{equation}
\label{eq:milddrift}
  F( X_t ) , \;
  t \in [0,T] ,
\end{equation}
and with mild diffusion
\begin{equation}
\label{eq:milddiffusion}
  B( X_t ) , \; t \in [0,T] .
\end{equation}
This fact now enables us to apply
the mild It\^{o} formula~\eqref{eq:itoformel_start}
to the solution process $ X $
of the 
SPDE~\eqref{eq:SPDE}.
To this end let $ \mathcal{J} $ be a set 
and let
$ g_j \in U_0 $, 
$ j \in \mathcal{J} $,
be an arbitrary orthonormal basis
of the $ \mathbb{R} $-Hilbert
space
$ 
  ( 
  U_0, 
  \left< \cdot, \cdot 
  \right>_{ U_0 }, 
  \left\| \cdot
  \right\|_{ U_0 } 
  ) 
$.
A direct consequence
of Theorem~\ref{thm:ito}
and Corollary~\ref{cor:itoauto}
is the next corollary.

\begin{cor}[A new - 
somehow mild - 
It{\^o} formula for solutions
of SPDEs]
\label{cor:ito}
Assume that
the setting in Section~\ref{sec:setting}
is fulfilled.
Then
\begin{equation}
\label{eq:well1d}
  \mathbb{P}\!\left[
    \int_{ t_0 }^t
    \big\|
      \varphi'( e^{ A(t-s) } X_s ) \,
      e^{ A(t-s) }	 F( X_s )
    \big\|_V \,
    ds
    < \infty 
  \right] = 1,
\end{equation}
\begin{equation}
\label{eq:well2d}
  \mathbb{P}\!\left[
    \int_{ t_0 }^t
    \big\|
      \varphi'( e^{ A(t-s) } X_s ) \,
      e^{ A(t-s) } B( X_s )
    \big\|_{ HS(U_0, V ) }^2 
    \,
    ds
    < \infty 
  \right] = 1,
\end{equation}
\begin{equation}
\label{eq:well3d}
  \mathbb{P}\!\left[
    \int_{ t_0 }^t
    \big\|
      \varphi''( e^{ A( t- s) } X_s )
    \big\|_{ L^{(2)}( H_{ r }, V ) }
    \,
    \big\|
      e^{ A (t - s) } B( X_s )
    \big\|_{ HS(U_0, H_{ r }) }^2
    \, 
    ds
    < \infty 
  \right] = 1
\end{equation}
and
\begin{align}
\label{eq:itoformel}
  \varphi( X_t )
 &=
  \varphi( e^{ A( t - t_0 ) } X_{ t_0 } )
  +
  \int_{ t_0 }^t
  \varphi'( e^{ A( t - s ) } X_s ) \,
  e^{ A( t - s ) }
  F( X_s ) \, ds
\nonumber
\\&\quad+
  \int_{ t_0 }^t
  \varphi'( e^{ A( t - s ) } X_s ) \,
  e^{ A( t - s ) }
  B( X_s ) \, dW_s
\\&\quad+
\nonumber
  \frac{1}{2}
  \sum_{ j \in \mathcal{J} }
  \int_{ t_0 }^t
  \varphi''( e^{ A( t - s ) } X_s )
  \left(
    e^{ A( t - s ) }
    B( X_s ) g_j,
    e^{ A( t - s ) }
    B( X_s ) g_j
  \right) ds
\end{align}
$ \mathbb{P} $-a.s.\ for 
all $ t_0, t \in [0,T] $
with $ t_0 < t $,
all
$ 
  \varphi \in
  C^2( H_{ r }, V ) 
$
and all
$
  r \in
  ( - \infty,
    \min( \alpha + 1, \beta + \frac{ 1 }{ 2 } )
  )
$.
\end{cor}

First, observe that the
possibly infinite sum
and all integrals in 
\eqref{eq:itoformel} are
well defined due to
\eqref{eq:well1d}--\eqref{eq:well3d}.
Next define mappings 
$
  K_0
  \colon
  \cup_{ r \in \mathbb{R} } 
  C( H_{ r }, V )
  \rightarrow 
  \cup_{ r \in \mathbb{R} }
  C( H_{ r }, V )
$
and
$
  K_t 
  \colon
  \cup_{ r \in \mathbb{R} } 
  C( H_{ r }, V )
  \rightarrow 
  C( H_{ \min(\alpha,\beta,\gamma) }, V )
$,
$ t \in (0,\infty) $,
through
$ K_0( \varphi ) := \varphi $
and
\begin{equation}
\label{eq:defK}
  \big( K_t \varphi \big)(x)
:=
  \varphi( e^{ A t } x )
\end{equation}
for all
$ x \in H_{ \min( \alpha, \beta, \gamma) } $,
$ 
  \varphi \in 
  \cup_{ r \in \mathbb{R} }
  C( H_{ r }, V )
$
and all
$ t \in (0,\infty) $.
Note that
$ 
  K_{ t_1 } \circ K_{ t_2 } =
  K_{ t_1 + t_2 }
$
for all $ t_1, t_2 \in [0,\infty) $.
In addition, define 
linear operators
$ 
  L^{(0)} \colon
  C^2( 
    H_{ \min( \alpha, \beta, \gamma ) },
    V
  )
  \rightarrow
  C( 
    H_{ \gamma },
    V
  )
$ 
and
$ 
  L^{(1)} \colon
  C^1( 
    H_{ \min( \beta, \gamma ) },
    V
  )
  \rightarrow
  C( 
    H_{ \gamma },
    V
  )
$ 
through
\begin{equation}
\label{eq:defL0}
\begin{split}
  ( L^{ (0) } \varphi )(x)
& :=
  \varphi'(x) F(x)
+
  \frac{1}{2}
  \sum_{ j \in \mathcal{J} }
  \varphi''(x) \big(
    B(x) g_j, B(x) g_j
  \big)
\\ & =
  \varphi'(x) F(x)
+
  \frac{1}{2}
  \text{Tr}\Big(
    \big( B(x) \big)^{ * }
    \varphi''(x) \,
    B(x)
  \Big)
\end{split}
\end{equation}
for all 
$
  x \in H_{ \gamma }
$,
$ 
  \varphi \in 
  C^{2}( 
    H_{ \min( \alpha, \beta, \gamma ) }, 
    V
  ) 
$
and through
$
  ( L^{ (1) } \varphi
  )(x)
:=
  \varphi'(x) B(x)
$
for all 
$
  x \in H_{ \gamma }
$,
$ 
  \varphi \in 
  C^{1}( 
    H_{ \min( \beta, \gamma ) }, 
    V
  ) 
$.
Furthermore,
define mappings
$
  L_t^{(0)} \colon 
  \cup_{ r \in \mathbb{R} }
  C^{ 2 }( H_{ r }, V )
  \rightarrow
  C( H_{ \gamma }, V )
$,
$ t \in (0,\infty) $,
and
$
  L_t^{(1)} \colon 
  \cup_{ r \in \mathbb{R} }
  C^{ 1 }( H_{ r }, V )
  \rightarrow
  C( H_{ \gamma }, HS( U_0, V) )
$,
$ t \in (0,\infty) $,
through
$
  L_{ t }^{(0)}( \varphi )
  :=
  L^{ (0) }( K_t( \varphi ) )
$
for all
$ 
  \varphi \in
  \cup_{ r \in \mathbb{R} }
  C^2( H_r , V )
$ 
and through
$
  L_{ t }^{(1)}( \varphi )
  :=
  L^{ (1) }( K_t( \varphi ) )
$
for all
$ 
  \varphi \in
  \cup_{ r \in \mathbb{R} }
  C^1( H_r , V )
$.
Note that these definitions imply
\begin{equation}
\label{eq:defL0t}
\begin{split}
  \big( 
    L_{ t }^{(0)} \varphi
  \big)( x )
& =
  \varphi'( e^{ A t } x )
  \, e^{ A t } F(x)
+
  \frac{ 1 }{ 2 }
  \sum_{ j \in \mathcal{J} }
  \varphi''( e^{ A t } x )\big(
    e^{ A t } B(x) g_j,
    e^{ A t } B(x) g_j
  \big)
\\ & =
  \varphi'( e^{ A t } x )
  \, e^{ A t } F(x)
+
  \frac{ 1 }{ 2 }
  \text{Tr}\Big(
  \big( e^{ A t } B(x) \big)^{ \! * }
  \varphi''( e^{ A t } x ) \,
  e^{ A t } B(x)
  \Big)
\end{split}
\end{equation}
for all
$ x \in H_{ \gamma } $,
$ 
  \varphi \in 
  \cup_{ r \in \mathbb{R} }
  C^2( H_{ r }, V ) 
$, 
$ t \in (0,\infty) $
and 
\begin{equation}
\label{eq:defL1}
  \big( 
    L_{ t }^{(1)} \varphi
  \big)( x )
=
  \varphi'( e^{ A t } x )
  \, e^{ A t } B(x)
\end{equation}
for all
$ x \in H_{ \gamma } $,
$ 
  \varphi \in 
  \cup_{ r \in \mathbb{R} }
  C^1( H_{ r }, V ) 
$,
$ t \in (0,\infty) $.
The mild It\^{o} 
formula~\eqref{eq:itoformel} can
thus be written as
\begin{equation}
\label{eq:operatorIto}
\begin{split}
  \varphi( X_t )
& =
  \varphi( e^{ A (t - t_0) } X_{ t_0 } )
+
  \int_{ t_0 }^t
  \!
  \big( 
    L_{ ( t - s ) }^{(0)} \varphi
  \big) ( X_s )
  \, ds
+
  \int_{ t_0 }^t
  \!
  \big( 
    L_{ ( t - s ) }^{(1)} \varphi
  \big) ( X_s )
  \, dW_s
\\ & =
  \big(
    K_{ ( t - t_0 ) } \varphi
  \big)( X_{ t_0 } )
+
  \int_{ t_0 }^t
  \!
  \big( 
    L_{ ( t - s ) }^{(0)} \varphi
  \big) ( X_s )
  \, ds
+
  \int_{ t_0 }^t
  \!
  \big( 
    L_{ ( t - s ) }^{(1)} \varphi
  \big) ( X_s )
  \, dW_s
\\ & =
  \big(
    K_{ ( t - t_0 ) } \varphi
  \big)( X_{ t_0 } )
+
  \int_{ t_0 }^t
  \!
  \big( 
    L^{(0)} K_{(t-s)} \varphi
  \big) ( X_s )
  \, ds
+
  \int_{ t_0 }^t
  \!
  \big( 
    L^{(1)} K_{(t-s)} \varphi
  \big) ( X_s )
  \, dW_s
\end{split}
\end{equation}
$ \mathbb{P} $-a.s.\ for all 
$ t_0, t \in [0,T] $ with $ t_0 < t $
and all
$ 
  \varphi \in 
  \cup_{ 
    r < 
    \min( \alpha + 1, \beta + \frac{ 1 }{ 2 } )
  }
  C^2( H_r, V ) 
$.
Moreover, taking 
expectations on both
side of \eqref{eq:operatorIto}
gives
\begin{equation}
\label{eq:operatorIto2}
\begin{split}
  \mathbb{E}\Big[
    \varphi( X_t )
  \Big]
& =
  \mathbb{E}\Big[
    \big(
      K_{ (t - t_0) }
      \varphi 
    \big)( X_{ t_0 } )
  \Big]
+
  \int_{ t_0 }^t
  \mathbb{E}\Big[
  \big( 
    L_{ ( t - s ) }^{(0)} \varphi
  \big) ( X_s )
  \Big]
  \, ds
\\ & =
  \mathbb{E}\Big[
  \big(
    K_{ ( t - t_0 ) } \varphi
  \big)( X_{ t_0 } ) \,
  \Big]
+
  \int_{ t_0 }^t
  \mathbb{E}\Big[
  \big( 
    L^{(0)} K_{(t-s)} \varphi
  \big) ( X_s ) \,
  \Big]
  \, ds
\end{split}
\end{equation}
for all $ t_0, t \in [0,T] $
with $ t_0 < t $
and all
$ 
  \varphi \in
  \cup_{ 
    r < 
    \min( \alpha + 1, \beta + \frac{ 1 }{ 2 } )
  } 
  G^2_0( H_r, V )
$.
Based on \eqref{eq:operatorIto2}
a mild Kolmogorov
backward equation is derived
in Subsection~\ref{sec:kolmogorov} below.
Other kinds of It{\^o} type formulas
for solutions of SPDEs can 
be found 
in \cite{c07,dz92,gnt05,gk81,gk8182,kr79,l07,l09b,lt08,op89,p72,p87,PrevotRoeckner2007,rrw07,z06}.
In the next step
Corollary~\ref{cor:ito}
is illustrated by two simple 
examples.

%\subsubsection{Examples
%for the mild It\^{o} formula}
%
%In this subsection various examples
%of Corollary~\ref{cor:ito} are presented.

\begin{example}[Identity]
Assume that the setting
in Section~\ref{sec:setting}
is fulfilled, 
let
$ V = H_{ \gamma } $, 
let 
$ \left\| v \right\|_V
= \left\| v \right\|_{ H_\gamma } $ 
for all 
$ v \in H_{ \gamma } $
and
let $ \varphi \colon H_{\gamma} 
\rightarrow H_{\gamma} $
be the identity on 
$ H_{\gamma} $, 
i.e., $ \varphi(v) = v $
for all $ v \in H_{\gamma} $.
The mild It\^{o} formula~\eqref{eq:itoformel}
in Corollary~\ref{cor:ito} 
then reduces to
\begin{equation}
  X_t
  =
  e^{ A( t - t_0 ) } X_{ t_0 }
  +
  \int_{ t_0 }^t
    e^{ A( t - s ) } 
    F( X_s ) \,
  ds
  +
  \int_{ t_0 }^t
    e^{ A( t - s ) } B( X_s ) \, 
  dW_s
\end{equation}
$ \mathbb{P} $-a.s.\ for
all $ t_0, t \in [0,T] $
with $ t_0 \leq t $. 
This is nothing else but
the mild
formulation of 
the SPDE~\eqref{eq:SPDE}.
In this sense, the 
formula~\eqref{eq:itoformel}
is somehow a mild 
It{\^o} formula for SPDEs.
\end{example}
\begin{example}[Squared norm]
Assume that the setting
in Section~\ref{sec:setting}
is fulfilled,
let $ V = \mathbb{R} $,
$ 
  \left\| v \right\|_V = 
  \left| v \right| 
$ 
for all 
$ v \in V = \mathbb{R} $,
assume 
$ 
  \min( \alpha + 1, \beta + \frac{ 1 }{ 2 } ) 
  > 0 
$ 
and
let 
$ 
  \varphi \colon H 
  \rightarrow V 
$ 
be given 
by 
$ 
  \varphi( v ) = \left\| v \right\|_H^2 
$ 
for all 
$ v \in H 
$. 
The mild It\^{o} formula~\eqref{eq:itoformel}
in Corollary~\ref{cor:ito} 
then reduces to
\begin{equation}
\label{eq:squarenorm}
\begin{split}
  \left\| X_t \right\|^2_H
 &=
  \big\|
    e^{ A( t - t_0 ) } X_{ t_0 }
  \big\|_H^2
  +
  2 \int_{ t_0 }^t
  \left<
    e^{ A( t - s ) } X_s,
    e^{ A( t - s ) } F( X_s )
  \right>_{ \! H } ds
\\&\quad+
  2 \int_{ t_0 }^t
  \left<
    e^{ A( t - s ) } X_s,
    e^{ A( t - s ) } B( X_s ) \, dW_s
  \right>_{ \! H }
  +
  \int_{ t_0 }^t
  \big\|
    e^{ A( t - s ) } B( X_s )
  \big\|_{ HS( U_0, H ) }^2
  \, ds
\end{split}
\end{equation}
$ \mathbb{P} $-a.s.\ for
all $ t_0, t \in [0,T] $
with $ t_0 < t $
(see also Example~\ref{ex:1}
above).
We refer to
\cite{gk81,gk8182,kr79,op89,p72,PrevotRoeckner2007,rrw07}
for other It{\^o} type 
formulas
with the particular test 
function
$ 
  \varphi(v) = 
  \left\| v \right\|_H^2 
$,
$ v \in H $.
If 
$ X_0 = 0 $
and
$ F(v) = 0 $ 
for all $ v \in H $
in addition to 
the above assumptions,
then \eqref{eq:squarenorm}
simplifies to
\begin{equation}
\label{eq:squarenorm2}
  \left\| X_t \right\|^2_H
 =
  2 \int_{ 0 }^t
  \left<
    e^{ A( t - s ) } X_s,
    e^{ A( t - s ) } B( X_s ) \, dW_s
  \right>_{ \! H }
  +
  \int_{ 0 }^t
  \big\|
    e^{ A( t - s ) } B( X_s )
  \big\|_{ HS( U_0, H ) }^2 
  \, ds
\end{equation}
$ \mathbb{P} $-a.s.\ for
all $ t \in [0,T] $ and this,
in particular, gives
\begin{equation}
\label{eq:squarenorm3}
  \mathbb{E}\!\left[
    \left\| X_t \right\|^2_H
  \right]
 =
  \int_{ 0 }^t
  \mathbb{E}\!\left[
  \big\|
    e^{ A( t - s ) } B( X_s )
  \big\|_{ HS( U_0, H ) }^2 
  \right] 
  ds
\end{equation}
for all $ t \in [0,T] $.
Clearly, 
equation~\eqref{eq:squarenorm3}
is nothing else but
a special case of It\^{o}'s 
isometry.
\end{example}

\subsubsection{SPDEs
with time dependent
coefficients}

In addition to the setting
in Section~\ref{sec:setting}
assume in this subsection
that
$ 
  \tilde{F} \colon
  [0,\infty) \times 
  H_{ \gamma }
  \rightarrow
  H_{ \alpha }
$ 
and 
$ 
  \tilde{B} \colon
  [0,\infty) \times H_{ \gamma }
  \rightarrow
  HS( U_0, H_{ \beta } ) 
$
are two globally Lipschitz 
continuous mappings.
Then there exists an up to
modifications unique predictable
stochastic process
$
  \tilde{X} \colon [0,\infty) \times
  \Omega \rightarrow H_{ \gamma } 
  \in
  C( [0,\infty), 
    L^p( \Omega; H_{ \gamma } ) 
  )
$
which fulfills
\begin{equation}
  \tilde{X}_t
=
  e^{ A t } \xi
+
  \int_0^t
  e^{ A (t - s) } 
  \tilde{F}( s, \tilde{X}_s ) \, ds
+
  \int_0^t
  e^{ A (t - s) } 
  \tilde{B}( s, \tilde{X}_s ) \, dW_s
\end{equation}
$ \mathbb{P} $-a.s.\ for
all $ t \in [0,\infty) $.
Next define mappings
$
  L^{ (0) }_{ s, t }
  \colon
  \cup_{ r \in \mathbb{R} }
  C^2( H_r, V )
  \rightarrow
  C( H_{ \gamma }, V )
$,
$
  s, t \in [0,\infty)
$,
$
  s < t 
$,
and
$
  L^{ (1) }_{ s, t }
  \colon
  \cup_{ r \in \mathbb{R} }
  C^1( H_r, V )
  \rightarrow
  C( H_{ \gamma }, V )
$,
$
  s, t \in [0,\infty)
$,
$
  s < t 
$,
through
\begin{equation}
  \big( 
    L^{(0)}_{ s, t } \varphi 
  \big)( x ) 
:=
  \varphi'( e^{ A (t - s) } x ) 
  \,
  e^{ A (t - s) }
  \tilde{F}(s, x)
+
  \frac{ 1 }{ 2 }
  \sum_{ 
    j \in \mathcal{J} 
  }
  \varphi''( e^{ A (t - s) } x )
  \big(
    e^{ A (t - s) } \tilde{B}( s, x ) g_j,
    e^{ A (t - s) } \tilde{B}( s, x ) g_j
  \big)
\end{equation}
for all
$ x \in H_{ \gamma } 
$,
$
  \varphi \in 
  \cup_{ r \in \mathbb{R} }
  C^2( H_{ r }, V )
$
and all
$
  s, t \in [0,\infty)
$
with $ s < t $
and through
\begin{equation}
  \big( 
    L^{ (1) }_{ s, t } \varphi 
  \big)( x )
:=
  \varphi'( e^{ A (t - s) } x ) \,
  e^{ A (t - s) } 
  \tilde{B}(s,x)
\end{equation}
for all
$ x \in H_{ \gamma } $,
$
  \varphi \in 
  \cup_{ r \in \mathbb{R} }
  C^1( H_{ r }, V )
$
and all
$
  s, t \in [0,\infty)
$
with $ s < t $.
Corollary~\ref{cor:itoauto}
then implies
\begin{equation}
\label{eq:timedependent}
  \varphi( \tilde{X}_t )
=
  \big(
    K_{ ( t - t_0 ) } \varphi
  \big)( \tilde{X}_{ t_0 } )
+
  \int_{ t_0 }^t
  \!
  \big( 
    L^{ (0) }_{ s, t } \varphi
  \big) ( \tilde{X}_s )
  \, ds
+
  \int_{ t_0 }^t
  \!
  \big( 
    L^{(1)}_{ s, t } \varphi
  \big) ( \tilde{X}_s )
  \, dW_s
\end{equation}
for all 
$ t_0, t \in [0,\infty) $
with $ t_0 < t $
and all
$ 
  \varphi \in 
  \cup_{ 
    r < 
    \min( \alpha + 1, \beta + \frac{ 1 }{ 2 } )
  }
  C^2( H_r, V )
$.
The mild It\^{o} 
formula~\eqref{eq:timedependent}
is nothing else but the 
counterpart of
\eqref{eq:itoformel}
for SPDEs with
time dependent coefficients.

\subsubsection{Mild Kolmogorov 
backward equation for 
SPDEs}
\label{sec:kolmogorov}

Based on \eqref{eq:operatorIto2}
a mild Kolmogorov
backward equation is derived 
in this subsection.
Proposition~\ref{prop:prop}
implies the existence of
predictable stochastic processes
$
  X^x \colon [0,\infty) \times \Omega
  \rightarrow H_{ \gamma }
  \in
  \cap_{ q \in [1,\infty) }
  C( 
    [0,\infty), 
    L^q( \Omega; H_{ \gamma } )
  )
$,
$ x \in H_{ \gamma } 
$,
such that
\begin{equation}
\label{eq:SPDE00}
  X_t^x = e^{ A t } x 
  +
  \int_0^t e^{ A (t - s) } 
  F( X_s^x ) \, ds
  +
  \int_0^t e^{ A (t - s) } 
  B( X_s^x ) \, dW_s
\end{equation}
$ \mathbb{P} $-a.s.\ for all
$ t \in [0,\infty) $ and all
$ x \in H_{ \gamma } $.
Proposition~\ref{prop:prop}
also implies that
$
  \P\big[
    X^x_t \in
    H_r
  \big] = 1
$
for all $ t \in (0,\infty) $
and all
$ 
  r \in 
  ( - \infty, 
    \min( \alpha + 1 , 
    \beta + \frac{ 1 }{ 2 } ) 
  ) 
$.
Then define mappings
$ 
  P_t \colon
  \cup_{ q \in [0,\infty) }
  \cup_{ 
    r \in ( - \infty,
      \min( \alpha + 1, \beta + \frac{ 1 }{ 2 } 
      )
    ) 
  }
  G^0_q( H_{ r }, V )
  \rightarrow
  \cup_{ q \in [0,\infty) }
  \cup_{ 
    r \in ( - \infty,
      \min( \alpha + 1, \beta + \frac{ 1 }{ 2 } 
      )
    ) 
  }
  G^0_q( H_{ r }, V )
$,
$ t \in [0,\infty) $, 
through
$
  P_0( \varphi ) := \varphi
$
for all
$
  \varphi \in
  \cup_{ q \in [0,\infty) }
  \cup_{ 
    r \in ( - \infty,
      \min( \alpha + 1, \beta + \frac{ 1 }{ 2 } 
      )
    ) 
  }
$
and through
$ 
  P_t( \varphi ) \in 
  \cup_{ q \in [0,\infty) }
  G^0_q( H_{ \gamma }, V )
$
and
\begin{equation}
\label{eq:defPt}
  ( 
    P_t \varphi 
  )( x )
:=
  \mathbb{E}\big[
    \varphi( X^{ x }_t )
  \big]
\end{equation}
for all 
$ x \in H_{ \gamma } $,
$ 
  \varphi \in
  \cup_{ q \in [0,\infty) }
  \cup_{ 
    r \in 
    ( - \infty,
      \min( \alpha + 1, \beta + \frac{ 1 }{ 2 } 
      )
    ) 
  }
$
and all
$ t \in (0,\infty) $.
Note that
$
  P_t\big(
    G^0_q( H_r, V )
  \big)
  \subset
  G^0_q( H_{ \gamma }, V )
$
for all
$ t \in (0,\infty) $,
$ 
  r \in 
  ( - \infty, 
    \min( \alpha + 1,
    \beta + \frac{ 1 }{ 2 }
    )
  )
$
and all
$
  q \in [0,\infty)
$.
The next lemma collect a few
simple properties of
the linear operators 
$ P_t $, $ t \in [0,\infty) $.

\begin{lemma}[Properties
of $ P_t $, $ t \in [0,\infty) $]
\label{lem:Pt}
Assume that the setting in
Section~\ref{sec:setting}
is fulfilled. Then
$ 
  P_t \in   
  L( G^0_q( H_{ \gamma }, V ) )
$
and
$
  \sup_{ s \in [0,t] }
  \| P_s \|_{
    L( G^0_q( H_{ \gamma }, V ) )
  }
  < \infty
$
for all $ t, q \in [0,\infty) $
and it holds for every
$ q \in [0,\infty) $
that the function 
$
  [0,\infty) \ni t
  \mapsto
  P_t \in
  L( 
    Lip^1_{ q }( H_{ \gamma }, V), 
    G^0_{ q + 1 }( H_{ \gamma }, V)
  )
$
is locally
$ \frac{ 1 }{ 2 } $-H\"{o}lder
continuous.
\end{lemma}

The proof of 
Lemma~\ref{lem:Pt}
is a straightforward
consequence of
inequality~\eqref{eq:lipestimateGnq}
and is therefore omitted.
Next the
mild It\^{o} formula
in \eqref{eq:operatorIto}
implies
\begin{equation}
\label{eq:mildP2}
\begin{split}
  \big( 
    P_t \varphi 
  \big)( x )
& =
  \mathbb{E}\Big[
    \big(
      K_{ (t - t_0) }
      \varphi 
    \big)( X_{ t_0 }^x )
  \Big]
+
  \int_{ t_0 }^t
  \mathbb{E}\Big[
  \big( 
    L_{ ( t - s ) }^{(0)} \varphi
  \big) ( X_s^x )
  \Big]
  \, ds
\\ & =
  \big(
    P_{ t_0 } K_{ ( t - t_0 ) } 
    \varphi 
  \big)(x)
+
  \int_{ t_0 }^t
  \big( 
    P_s
    L_{ ( t - s ) }^{(0)} \varphi
  \big) ( x )
  \, ds
\\ & =
  \big(
    P_{ t_0 } K_{ ( t - t_0 ) } 
    \varphi 
  \big)(x)
+
  \int_{ t_0 }^t
  \big( 
    P_s \,
    L^{(0)} K_{ ( t - s ) } \varphi
  \big) ( x )
  \, ds
\end{split}
\end{equation}
for all $ t_0, t \in [0,\infty) $
with $ t_0 < t $,
$ x \in H_{ \gamma } $
and all 
$ 
  \varphi \in
  \cup_{ q \in [0,\infty) }
  \cup_{ 
    r \in 
    ( - \infty,
      \min( \alpha + 1, \beta + \frac{ 1 }{ 2 } 
      )
    ) 
  }
$.
In the next subsection 
we will use \eqref{eq:mildP2} to
study regularity properties
of solutions of SPDEs.
In this regularity analysis
we also use the following 
two lemmas.

\begin{lemma}[Estimates
for $ K_t $, $ t \in (0,\infty) $]
\label{lem:Ktbound}
Assume that the setting
in Section~\ref{sec:setting}
is fulfilled. Then 
the function
$
  (0,\infty) \ni t
  \mapsto
  K_t \in
    L\!\left(
      Lip^{ n + 1 }_q( H_{ r_1 }, V ) 
      ,
      G^n_{ q + 1 }( 
        H_{ r_2 
        }, 
        V 
      )
    \right)
$
is continuous
for every $ r_1, r_2 \in \R $,
$ q \in [0,\infty) $,
$ n \in \{ 0, 1, \dots \} $
and it holds that
$
  K_t \in
    L(
      Lip^n_q( H_{ r_1 }, V ) 
      ,
      Lip^n_{ q }( 
        H_{ r_2 
        }, 
        V 
      )
    )
$
and that
\begin{align}
\label{eq:normuse1}
  \| 
    K_t( \varphi_0 )
  \|_{ 
    G^0_q( H_{ r_1 }, V )
  }
& \leq
  \max\!\big( 
    1, 
    \| 
      e^{ A t } 
    \|_{ 
      L( H_{ r_2 }, H_{ r_1 } ) 
    }^q 
  \big)
  \,
  \|
    \varphi_0
  \|_{
    G^0_q( H_{ r_2 }, V )
  },
\\
\label{eq:normuse2}
  \| 
    ( K_t \varphi_1 )'
  \|_{ 
    G^0_{q}( H_{ \gamma }, 
      L( H_{ \alpha }, V) 
    )
  }
& \leq
  \max\!\big( 
    1, \| e^{ A t } \|_{ L(H) }^q 
  \big)
  \,
  \|
    e^{ A t }
  \|_{ L(H_{\alpha},H_{\gamma}) }
  \,
  \|
    \varphi_1'
  \|_{
    G^0_{q}( 
      H_{ \gamma }, L( H_{\gamma}, V) 
    )
  } ,
\\
\label{eq:normuse3}
  \| 
    ( K_t \varphi_2 )''
  \|_{ 
    G^0_{q}( H_{ \gamma }, 
      L^{(2)}( H_{ \beta }, V) 
    )
  }
& \leq
  \max\!\big( 1, \| e^{ A t } \|_{ L(H) }^q 
  \big)
  \,
  \|
    e^{ A t }
  \|_{ L(H_{\beta},H_{\gamma}) }^2
  \,
  \|
    \varphi_2''
  \|_{
    G^0_{q}( 
      H_{ \gamma }, 
      L^{(2)}( H_{\gamma}, V) 
    )
  } 
\end{align}
for all 
$
  \varphi_0 \in
  C( H_{ r_2 }, V )
$,
$
  \varphi_1 \in
  C^1( H_{ \gamma }, V )
$,
$
  \varphi_2 \in
  C^2( H_{ \gamma }, V )
$,
$ q \in [0,\infty) $,
$ n \in \N $,
$ r_1, r_2 \in \R $
and all
$
  t \in (0,\infty)
$.
\end{lemma}

\begin{lemma}[Estimates
for $ L^{(0)} $]
\label{kolm:bound0}
Assume that the setting
in Section~\ref{sec:setting}
is fulfilled. Then
it holds that
$
  L^{ (0) }
  \in 
  L( 
    G^2_{ q }( 
      H_{ \min(\alpha,\beta, \gamma) 
      }, 
      V 
    ),
    G^0_{ q + 2 }( H_{ \gamma }, V )
  )
$,
that
$
  L^{ (0) }
  \in 
  L( 
    Lip^3_{ q }( 
      H_{ \min(\alpha,\beta, \gamma) 
      }, 
      V 
    ),
    \text{Lip}^1_{ q + 2 
    }( H_{ \gamma }, V )
  )
$
and that
\begin{equation}
\label{eq:estimateL0}
\begin{split}
&
  \| 
    L^{(0)}( \varphi )
  \|_{
    G^0_{ q + 2 }( H_{ r }, V )
  }
\\ & \leq
  \| F \|_{ 
    G_1^0( H_{ r }, H_{ \alpha } )
  }
  \,
  \|
    \varphi'
  \|_{ 
    G_{ q + 1 }^0( H_{ r },
      L( H_{ \alpha }, V )  
    )
  }
  +
  \tfrac{ 1 }{ 2 } \,
  \| B \|_{ 
    G_1^0( H_{ r }, 
    HS( U_0, H_{ \beta } ) )
  }^2
  \,
  \|
    \varphi''
  \|_{ 
    G_{ q }^0( H_{ r },
      L^{(2)}( H_{ \beta }, V )  
    )
  } ,
\\ & \leq
  \max\!\big(
    \| F \|_{ 
      G_1^0( H_{ r }, H_{ \alpha } )
    },
    \| B \|_{ 
      G_1^0( H_{ r }, 
      HS( U_0, H_{ \beta } ) )
    }^2
  \big)
  \big(
  \|
    \varphi'
  \|_{ 
    G_{ q + 1 }^0( H_{ r },
      L( H_{ \alpha }, V )  
    )
  }
  +
  \|
    \varphi''
  \|_{ 
    G_{ q }^0( H_{ r },
      L^{(2)}( H_{ \beta }, V )  
    )
  } 
  \big)
\end{split}
\end{equation}
for all 
$ r \in [ \gamma, \infty ) 
$,
$ 
  \varphi \in C^{2}( 
    H_{ \min( \alpha, \beta, \gamma ) }, 
  V) 
$
and all $ q \in [0,\infty) $.
\end{lemma}

The proofs of 
Lemma~\ref{lem:Ktbound}
and Lemma~\ref{kolm:bound0}
are straightforward 
and therefore omitted.
The proof of Lemma~\ref{lem:Ktbound}
makes use of
inequality~\eqref{eq:lipestimateGnq}
above.
The next corollary follows
from 
Lemmas~\ref{lem:Pt}--\ref{kolm:bound0}.

\begin{cor}
\label{cor:continuity}
Assume that the setting in
Section~\ref{sec:setting} is fulfilled.
Then 
the function
$
  (t_0, t) \ni
  s \mapsto
  P_s \, L^{(0)} K_{ t - s } 
  \in
  L( 
    Lip^3_q( H_{ \gamma }, V ), 
    G^0_{ q + 3 }( H_{ \gamma }, V )
  )
$
is continuous and satisfies
$
  \int_{ t_0 }^t
  \| 
    P_s \, L^{ (0) } K_{ t - s }
  \|_{
    L( 
      Lip^3_q( H_{ \gamma }, V ), 
      G^0_{ q + 3 }( H_{ \gamma }, V )
    )
  } 
  \, ds
  < \infty
$
for every $ t_0, t \in [0,\infty) $
with $ t_0 < t $
and every 
$ q \in [0,\infty) $.
\end{cor}

\begin{proof}[Proof
of Corollary~\ref{cor:continuity}]
The triangle inequality implies that
\begin{equation}
\label{eq:continuity}
\begin{split}
&
  \big\|
    P_{ s } L^{(0)} K_{ t - s } 
    -
    P_{ s_0 } L^{(0)} K_{ t - s_0 } 
  \big\|_{ 
    L( 
      Lip^3_q( H_{ \gamma }, V ), 
      G^0_{ q + 3 }( H_{ \gamma }, V )
    )
  }
\\ & \leq
  \big\|
    P_{ s } L^{(0)} K_{ t - s } 
    -
    P_{ s } L^{(0)} K_{ t - s_0 } 
  \big\|_{ 
    L( 
      Lip^3_q( H_{ \gamma }, V ), 
      G^0_{ q + 3 }( H_{ \gamma }, V )
    )
  }
\\ & 
  +
  \big\|
    P_{ s } L^{(0)} K_{ t - s_0 } 
    -
    P_{ s_0 } L^{(0)} K_{ t - s_0 } 
  \big\|_{ 
    L( 
      Lip^3_q( H_{ \gamma }, V ), 
      G^0_{ q + 3 }( H_{ \gamma }, V )
    )
  }
\\ & \leq
  \| P_{ s } \|_{       
    L(
      G^0_{ q + 3 }( H_{ \gamma }, V )
    )
 }
  \,
  \| L^{ (0) } \|_{ 
    L( 
      G^2_{ q + 1 }( 
        H_{ 
          \min( \alpha, \beta, \gamma)
        }, 
        V 
      ),
      G^0_{ q + 3 }( H_{ \gamma }, V )
    )
  }
  \\ & \quad \cdot
  \|
    K_{ t - s } 
    -
    K_{ t - s_0 } 
  \|_{ 
    L( 
      Lip^3_q( H_{ \gamma }, V ), 
      G^2_{ q + 1 }(
        H_{ 
          \min( \alpha, \beta, \gamma)
        }, V
      )
    )
  }
\\ & 
  +
  \big\|
    P_{ s } 
    -
    P_{ s_0 }
  \big\|_{ 
    L( 
      Lip^1_{ q + 2 }( H_{ \gamma }, V ), 
      G^0_{ q + 3 }( H_{ \gamma }, V )
    )
  }
  \,
  \| L^{ (0) } \|_{ 
    L( 
      Lip^3_q( 
        H_{ \min( \alpha, \beta, \gamma ) 
        }, V 
      ) ,
      Lip^1_{ q + 2 }( H_{ \gamma }, V )
    )
  }
  \\ & \quad \cdot
  \|
    K_{ t - s_0 }
  \|_{
    L( 
      Lip^3_q( H_{ \gamma }, V ) ,
      Lip^3_q( 
        H_{ \min( \alpha, \beta, \gamma ) 
        }, V 
      ) 
    )
  }
\\ & \leq
  \underbrace{
  \left[
  \| L^{ (0) } \|_{ 
    L( 
      G^2_{ q + 1 }( 
        H_{ 
          \min( \alpha, \beta, \gamma)
        }, 
        V 
      ),
      G^0_{ q + 3 }( H_{ \gamma }, V )
    )
  }
  \cdot
  \sup_{ u \in [t_0,t] }
  \| P_{ u } \|_{       
    L(
      G^0_{ q + 3 }( H_{ \gamma }, V )
    )
  }
  \right]
  }_{ < \infty 
    \text{ due to 
    Lemmas~\ref{kolm:bound0}
    and \ref{lem:Pt}}
  }
\\ & \quad \cdot
  \|
    K_{ t - s } 
    -
    K_{ t - s_0 } 
  \|_{ 
    L( 
      Lip^3_q( H_{ \gamma }, V ), 
      G^2_{ q + 1 }(
        H_{ 
          \min( \alpha, \beta, \gamma)
        }, V
      )
    )
  }
\\ & 
  +
  \underbrace{
  \left[
  \| L^{ (0) } \|_{ 
    L( 
      Lip^3_q( 
        H_{ \min( \alpha, \beta, \gamma ) 
        }, V 
      ) ,
      Lip^1_{ q + 2 }( H_{ \gamma }, V )
    )
  }
  \cdot
  \|
    K_{ t - s_0 }
  \|_{
    L( 
      Lip^3_q( H_{ \gamma }, V ) ,
      Lip^3_q( 
        H_{ \min( \alpha, \beta, \gamma ) 
        }, V 
      ) 
    )
  }
  \right]
  }_{ < \infty 
    \text{ due to 
    Lemmas~\ref{kolm:bound0}
    and \ref{lem:Ktbound}}
  }
\\ & \quad \cdot
  \|
    P_{ s } 
    -
    P_{ s_0 }
  \|_{ 
    L( 
      Lip^1_{ q + 2 }( H_{ \gamma }, V ), 
      G^0_{ q + 3 }( H_{ \gamma }, V )
    )
  }
\end{split}
\end{equation}
for all 
$ s_0, s \in (t_0, t) $,
$ q \in [0,\infty) $
and
$ t_0, t \in [0,\infty) $
with $ t_0 \leq t $.
Combining \eqref{eq:continuity}
with Lemma~\ref{lem:Pt} 
and Lemma~\ref{lem:Ktbound}
shows that
the function
$
  (t_0, t) \ni
  s \mapsto
  P_s \, L^{(0)} K_{ t - s } 
  \in
  L( 
    Lip^3_q( H_{ \gamma }, V ), 
    G^0_{ q + 3 }( H_{ \gamma }, V )
  )
$
is continuous 
for every $ t_0, t \in [0,\infty) $
with $ t_0 < t $
and every 
$ q \in [0,\infty) $.
Combining this 
with 
Lemmas~\ref{lem:Pt}--\ref{kolm:bound0}
completes the proof
of Corollary~\ref{cor:continuity}.
\end{proof}

In the following we reformulate
\eqref{eq:mildP2} in a suitable
abstract way
by using Corollary~\ref{cor:continuity}.
More precisely, 
combining Corollary~\ref{cor:continuity}
with equation~\eqref{eq:mildP2}
shows that
\begin{equation}
\label{eq:mildP}
\begin{split}
  P_t( \varphi )
& =
  P_{ t_0 }\!\left(
    K_{ ( t - t_0 ) }( \varphi ) 
  \right)
+
  \int_{ t_0 }^{ t }
  P_{ s }\big(  
    L^{ (0) }_{ ( t - s ) }( \varphi ) 
  \big)
  \, 
  ds
\\ & =
  P_{ t_0 }\!\left(
    K_{ ( t - t_0 ) }( \varphi ) 
  \right)
+
  \int_{ t_0 }^{ t }
  P_{ s }\big(  
  L^{ (0) }\big( K_{ ( t - s ) }( \varphi ) \big)
  \big)
  \, 
  ds
\end{split}
\end{equation}
in
$
  G^0_{ q + 3 }( H_{ \gamma }, V ) 
$
for all $ t_0, t \in [0,\infty) $
with $ t_0 \leq t $,
$ 
  \varphi \in 
  Lip^3_{ q }( H_{ \gamma }, V )
$
and all
$ q \in [0,\infty) $
where the integrals in \eqref{eq:mildP}
are Bochner integrals
in $ \mathbb{R} $-Banach
space 
$ 
  G^0_{ q + 3 }( H_{ \gamma }, V ) 
$.
According to
Corollary~\ref{cor:continuity},
these Bochner integrals
are indeed well defined.
We would like to add 
to the mild Kolmogorov 
backward equation~\eqref{eq:mildP}
that 
the mild Kolmogorov operators
$ L^{(0)}_t $, $ t \in (0,\infty) $,
appearing in \eqref{eq:mildP}
do, in general, not commute with
the semigroup operators
$ P_t $, $ t \in [0,\infty) $,
i.e, we do, in general, not have
that
$
  ( P_t L^{ (0) }_s )( \varphi )
  =
  ( L^{ (0) }_s P_t )( \varphi )
$
for all $ s, t \in (0,\infty) $
and all 
$ 
  \varphi \in G^2_0( H_{ \gamma }, V ) 
$.
This is in contrast
to the standard Kolmogorov backward
equation where the semigroup and
the Kolmogorov operator do commute
(see, e.g., Section~8.1
in Oeksendal~\cite{Oeksendal2000}).
In the
next step 
let
$
  \mathcal{K}_t 
  \colon
      L( 
        G^2_{ 1 }( 
          H_{ \min( \alpha, \beta, \gamma ) 
          }, 
        V ),
        G^0_{ 3 }( H_{ \gamma }, V )
      ) 
  \rightarrow
  L( 
    Lip^3_0( 
      H_{ \gamma },
      V 
    ),
    G^0_3( H_{ \gamma }, V )
  )
$,
$
  t \in (0,\infty) 
$,
let
$ 
  \mathcal{K}_0
  \colon
  L( 
    G^0_3( H_{ \gamma }, V )
  )
  \rightarrow
  L( 
    Lip^3_{ 0 }( H_{ \gamma }, V ),
    G^0_{ 3 }( H_{ \gamma }, V )
  )
$
and let
$
  \mathcal{L} 
  \colon
  L( 
    G^0_{ 3 }( H_{ \gamma }, V )
  ) 
  \rightarrow
      L( 
        G^2_{ 1 }( 
          H_{ \min( \alpha, \beta, \gamma ) 
          }, 
        V ),
        G^0_{ 3 }( H_{ \gamma }, V )
      ) 
$
be bounded linear
operators defined through
\begin{equation}
\label{eq:defLK2}
  ( \mathcal{K}_t \Phi )( \varphi )
:=
  \Phi\!\left( K_t( \varphi ) 
  \right)
\end{equation}
for all 
$ t \in (0,\infty) $,
$ 
  \varphi \in 
  Lip^3_0( 
    H_{ \gamma
    }, V 
  ) 
$
and all
$ 
  \Phi \in 
      L( 
        G^2_{ 1 }( 
          H_{ \min( \alpha, \beta, \gamma ) 
          }, 
        V ),
        G^0_{ 3 }( H_{ \gamma }, V )
      ) 
$,
through
$
  ( \mathcal{K}_0 \Phi )( \varphi )
  := \Phi( \varphi )
$
for all 
$
  \varphi \in
  Lip^3_{ 0 }( H_{ \gamma }, V )
$
and all
$
  \Phi \in
  L( 
    G^0_{ 3 }( H_{ \gamma }, V )
  ) 
$
and through
\begin{equation}
\label{eq:defLK}
    (
      \mathcal{L}
      \Phi
    )( \varphi )
:=
  \Phi\big( L^{ (0) }( \varphi ) \big)
\end{equation}
for all 
$
  \varphi \in
        G^2_{ 1 }( 
          H_{ \min( \alpha, \beta, \gamma ) 
          }
$
and all
$ 
  \Phi \in 
  L( 
    G^0_{ 3 }( H_{ \gamma }, V )
  ) 
$.
Lemmas~\ref{lem:Ktbound}
and \ref{kolm:bound0} ensure
that 
$ \mathcal{K}_t $, $ t \in [0,\infty) $,
and
$ \mathcal{L} $
are indeed well defined
bounded linear operators.
The next corollary follows
from
Lemmas~\ref{lem:Pt}--\ref{kolm:bound0}.

\begin{cor}
\label{cor:welldefinedMKB}
Assume that the setting
in Section~\ref{sec:setting}
is fulfilled. 
Then
the function
$
  (t_0, t) \ni s
  \mapsto
  \mathcal{K}_{ t - s }
  (
  \mathcal{L}(
  P_s )
  )
  \in
  L( 
    Lip^3_{ 0 }( H_{ \gamma }, V ),
    G^0_{ 3 }( H_{ \gamma }, V )
  )
$
is continuous and satisfies
$
  \int_{ t_0 }^t
  \|
  \mathcal{K}_{ t - s }
  (
  \mathcal{L}(
  P_s )
  )
  \|_{
    L( 
      Lip^3_{ 0 }( H_{ \gamma }, V ),
      G^0_{ 3 }( H_{ \gamma }, V )
    )
  }
  \, ds
  < \infty
$
for every $ t_0, t \in [0,\infty) $
with $ t_0 < t $.
\end{cor}

\begin{proof}[Proof
of Corollary~\ref{cor:welldefinedMKB}]
Note that the triangle
inequality implies that
\begin{equation}
\label{eq:continuity2}
\begin{split}
&
  \big\|
    \mathcal{K}_{ t - s }
    (
      \mathcal{L}(
      P_{ s } )
    )
    -
    \mathcal{K}_{ t - s_0 }
    (
      \mathcal{L}(
      P_{ s_0 } )
    )
  \big\|_{
    L( 
      Lip^3_{ 0 }( H_{ \gamma }, V ),
      G^0_{ 3 }( H_{ \gamma }, V )
    )
  }
\\ & \leq
  \big\|
    \mathcal{K}_{ t - s }
    (
      \mathcal{L}(
      P_{ s } )
    )
    -
    \mathcal{K}_{ t - s_0 }
    (
      \mathcal{L}(
      P_{ s } )
    )
  \big\|_{
    L( 
      Lip^3_{ 0 }( H_{ \gamma }, V ),
      G^0_{ 3 }( H_{ \gamma }, V )
    )
  }
\\ & +
  \big\|
    \mathcal{K}_{ t - s_0 }
    (
      \mathcal{L}(
      P_{ s } )
    )
    -
    \mathcal{K}_{ t - s_0 }
    (
      \mathcal{L}(
      P_{ s_0 } )
    )
  \big\|_{
    L( 
      Lip^3_{ 0 }( H_{ \gamma }, V ),
      G^0_{ 3 }( H_{ \gamma }, V )
    )
  }
\\ & \leq
  \|
    \mathcal{K}_{ t - s }
    -
    \mathcal{K}_{ t - s_0 }
  \|_{
    L\left(
      L( 
        G^2_{ 1 }( 
          H_{ \min( \alpha, \beta, \gamma ) 
          }, 
        V ),
        G^0_{ 3 }( H_{ \gamma }, V )
      ) ,
      L( 
        Lip^3_{ 0 }( H_{ \gamma }, V ),
        G^0_{ 3 }( H_{ \gamma }, V )
      ) 
    \right)
  } 
\\ & \cdot
  \underbrace{
  \|
    \mathcal{L}
  \|_{ 
    L\left( 
      L( 
        G^0_{ 3 }( 
          H_{ \gamma 
          }, 
        V )
      ),
      L( 
        G^2_{ 1 }( 
          H_{ \min( \alpha, \beta, \gamma ) 
          }, 
        V ),
        G^0_{ 3 }( H_{ \gamma }, V )
      ) 
    \right) 
  }
  }_{ 
    < \infty 
    \text{ due to     
    Lemma~\ref{kolm:bound0}}
  }
  \cdot
  \underbrace{
  \left[
  \sup_{ u \in [t_0,t] }
  \|
    P_{ u } 
  \|_{
    L(       
      G^0_{ 3 }( H_{ \gamma }, V )
    )
  }
  \right]
  }_{ 
    < \infty
    \text{ due to
    Lemma~\ref{lem:Pt}}
  }
\\ & +
  \underbrace{
  \left[
  \sup_{
    \substack{  
        \Phi \in 
        L( 
          G^0_{ 3 }( H_{ \gamma }, V )
        )
      \\
        \| \Phi \|_{
          L( 
            Lip^1_{ 2 }( H_{ \gamma }, V ), 
            G^0_{ 3 }( H_{ \gamma }, V )
          )
        } \leq 1 
    }
  }
  \|
    \mathcal{K}_{ t - s_0 }(
      \mathcal{L}( \Phi )
    )
  \|_{
    L( 
      Lip^3_{ 0 }( 
        H_{ \gamma 
        }, 
      V ),
      G^0_{ 3 }( H_{ \gamma }, V )
    ) 
  }
  \right]
  }_{ < \infty 
    \text{ due to     
    Lemmas~\ref{lem:Ktbound}
    and \ref{kolm:bound0}}
  }
  \cdot
  \,
  \|
    P_{ s } 
    -
    P_{ s_0 }
  \|_{ 
    L( 
      Lip^1_{ 2 }( H_{ \gamma }, V ), 
      G^0_{ 3 }( H_{ \gamma }, V )
    )
  }
\end{split}
\end{equation}
for all 
$ s_0, s \in (t_0, t) $
and all
$ t_0, t \in [0,\infty) $
with $ t_0 \leq t $.
Combining \eqref{eq:continuity2}
with Lemma~\ref{lem:Ktbound}
and Lemma~\ref{lem:Pt}
completes the proof
of Corollary~\ref{cor:welldefinedMKB}.
\end{proof}

We now use
Corollary~\ref{cor:welldefinedMKB}
to reformulate 
equation~\eqref{eq:mildP}.
More precisely,
combining 
Corollary~\ref{cor:welldefinedMKB},
equation~\eqref{eq:mildP},
definition~\eqref{eq:defLK} 
and 
definition~\eqref{eq:defLK2}
shows that
\begin{equation}
\label{eq:kolm1}
  P_t
=
  \mathcal{K}_{ (t - t_0) }(
  P_{ t_0 } )
+
  \int_{ t_0 }^{ t }
  \mathcal{K}_{ (t-s) }
  \big(
  \mathcal{L}(
  P_s )
  \big)
  \, ds
\end{equation}
in 
$ 
  L( 
    Lip^3_{ 0 }( H_{ \gamma }, V ),
    G^0_{ 3 }( H_{ \gamma }, V )
  )
$
for all $ t_0, t \in [0,\infty) $
with $ t_0 \leq t $
and this, in particular, implies that
\begin{equation}
\label{eq:kolm2}
  P_t
=
  \mathcal{K}_{ t }(
  P_{ 0 } )
+
  \int_{ 0 }^{ t }
  \mathcal{K}_{ (t - s) }\big(
  \mathcal{L}(
  P_s )
  \big)
  \, ds
\end{equation}
in
$ 
  L( 
    Lip^3_{ 0 }( H_{ \gamma }, V ),
    G^0_{ 3 }( H_{ \gamma }, V )
  )
$
for all $ t \in [0,\infty) $
where the integrals in 
\eqref{eq:kolm1} and
\eqref{eq:kolm2}
are understood to be 
Bochner integrals in
$ 
  L( 
    G^3_{ 0 }( H_{ \gamma }, V ),
    G^0_{ 3 }( H_{ \gamma }, V )
  )
$.
According to
Corollary~\ref{cor:welldefinedMKB},
these Bochner integrals are
indeed well defined.
Equation~\eqref{eq:kolm2}
and equation~\eqref{eq:kolm1}
are somehow 
{\it mild Kolmogorov
backward equations}
for the $ P_t $, $ t \in [0,\infty) $,
(see \eqref{eq:defPt})
associated to the
SPDE~\eqref{eq:SPDE00}.

\subsubsection{Weak regularity 
for solutions of SPDEs}
\label{sec:weakregularity}

Another consequence
of the mild It\^{o} 
formula~\eqref{eq:itoformel}
is to study weak regularity of
solutions of SPDEs. 
To be more
precise, in this subsection 
regularity 
of the 
probability
measures
$
  \mathbb{P}_{ X_t }
$,
$ 
  t \in (0,T] 
$,
of the solution process
$
  X_t
$,
$ 
  t \in [0,T]
$,
of the SPDE~\eqref{eq:SPDE} 
are studied by using the mild
Kolmogorov backward 
equation~\eqref{eq:mildP2} above.
Below (see the illustrations
below Lemma~\ref{lem:embedding})
we also describe in more detail
what we understand by regularity
of a probability measure.
While strong regularity
of solutions
of SPDEs have been 
intensively analyzed
in the literature (see, e.g.,
Da Prato \citationand\ 
Zabczyk~\cite{dz92, DaPratoZabczyk1996},
Brze\'{z}niak~\cite{b97b}, 
Brze\'{z}niak,
Van Neerven, Veraar 
\citationand\ Weis~\cite{bvvw08},
Van Neerven, Veraar 
\citationand\ 
Weis~\cite{vvw07,vvw08,vvw11},
Jentzen \citationand\ 
R\"{o}ckner~\cite{jr12},
Kruse \citationand\ 
Larsson~\cite{KruseLarsson2011}
and the references therein), 
weak regularity
for solutions of SPDEs
seem to be much less investigated.

Let us now go into details.
An important ingredient in our
analysis on weak regularity
of solutions of SPDEs are
the following mappings.
Let
$
  \left\| \cdot \right\|_{ 
    t, q 
  }^{
    \delta, \rho
  }
  \colon
  G^2_{ q }( 
    H_{ \rho }, 
    V 
  )
  \rightarrow
  [0,\infty)
$,
$ t \in (0,\infty) $,
$ q \in [0,\infty) $,
$ \delta \in (\rho - 1,\infty) $,
$ 
  \rho \in \R
$,
be a family of functions
defined through
\begin{align}
\label{eq:supernorm}
&
  \| \varphi \|_{ t, q }^{ \delta, \rho }
\\ &
  :=
  \| 
    K_t( \varphi )
  \|_{ 
    G^0_{ q + 2 }( H_{ \delta }, V )
  }
  +
  \int_0^t
  \left( t - s \right)^{ 
    \min( \delta - \rho, 0 )
  }
  \left(
  \|
    ( K_s \varphi )'
  \|_{ 
    G_{ q + 1 }^0( H_{ \rho }, 
      L( H_{ \alpha }, V )
    )
  }
  +
  \|
    ( K_s \varphi )''
  \|_{ 
    G_{ q }^0( H_{ \rho }, 
      L^{(2)}( H_{ \beta }, V )
    )
  }
  \right)
  ds
\nonumber
\end{align}
for all 
$ t \in (0,\infty) $,
$
  \varphi \in 
  G^2_{ q }( H_{ \rho }, V )
$,
$ \delta \in (\rho - 1,\infty) 
$,
$ q \in [0,\infty) $
and all
$
  \rho \in 
  \R
$.
Please note that the integrand
in \eqref{eq:supernorm}
is indeed Borel measurable
in $ s \in [0,\infty) $
since $ H_{ \rho } $ is
separable
for every $ \rho \in \R $.
The next lemma collects
some properties of the
functions
$
  \left\| \cdot 
  \right\|_{ t, q }^{
    \delta, \rho
  }
  \colon
  G^2_{ q }( H_{ \rho }, V )
  \rightarrow
  [0,\infty)
$,
$ t \in (0,\infty) $,
$ q \in [0,\infty) $,
$ \delta \in (\rho - 1,\infty) 
$,
$
  \rho \in 
  \R
$.

\begin{lemma}[Properties
of 
$
  \left\| \cdot \right\|_{ t, q 
  }^{ \delta, \rho }
$,
$ t \in (0,\infty) $,
$ q \in [0,\infty) $,
$ \delta \in (\rho - 1,\infty) $,
$
  \rho \in 
  \R
$]
\label{lem:norm}
Assume that the setting
in Section~\ref{sec:setting}
is fulfilled and let
$
  \left\| \cdot \right\|_{ t, q 
  }^{ \delta, \rho }
$,
$ t \in (0,\infty) $,
$ q \in [0,\infty) $,
$ \delta \in (\rho - 1,\infty) $,
$
  \rho \in 
  \R
$,
be defined through
\eqref{eq:supernorm}. Then 
\begin{equation}
\label{eq:norm}
\begin{split}
  \left\| \varphi 
  \right\|_{ t, q }^{ \delta, \rho }
& \leq  
  \| \varphi \|_{ 
    G^2_{ q }( H_{ \rho }, V )
  }
  \,
  \Bigg(
  \max\!\big( 
    1, 
    \| 
      e^{ A t } 
    \|_{ 
      L( H_{ \rho }, H_{ \delta } ) 
    }^{ ( q + 2 ) } 
  \big)
\\ & \quad
  + 
  \max\!\big( 1, 
    \| e^{ A t } \|_{ L(H) }^{ (q + 1) } 
  \big)
    \int_0^t
      \left( t - s \right)^{
        \min( \delta - \rho, 0 )
      }
      \left[
      \|
        e^{ A s }
      \|_{ L(H_{ \alpha },H_{ \rho } ) }
      +
      \|
        e^{ A s }
      \|_{ L(H_{ \beta }, H_{ \rho }) }^2
      \right]
    ds
  \Bigg)
  < \infty
\end{split}
\end{equation}
for all 
$
  \varphi \in
  G^2_{ q }( H_{ \rho }, V )
$,
$ q \in [0,\infty) $,
$
  t \in (0,\infty)
$,
$
  \delta \in ( \rho - 1, \infty )
$,
$ 
  \rho \in 
  ( - \infty, 
    \min( 
      \alpha + 1 , 
      \beta + \frac{ 1 }{ 2 } 
    ) 
  ) 
$
and it holds
for every
$ t \in (0,\infty) $,
$ q \in [0,\infty) $,
$ 
  \rho \in 
  ( - \infty, 
    \min( 
      \alpha + 1 , 
      \beta + \frac{ 1 }{ 2 } 
    ) 
  ) 
$,
$ 
  \delta \in (\rho - 1,\infty) 
$
that the mapping
$
  \left\| \cdot \right\|_{ t, q }^{ 
    \delta, \rho
  }
  \colon
  G^2_{ q }( H_{ \rho }, V )
  \rightarrow
  [0,\infty)
$
is a norm on
$
  G^2_{ q }( H_{ \rho }, V )
$.
\end{lemma}

\begin{proof}[Proof
of Lemma~\ref{lem:norm}]
Combining 
\eqref{eq:normuse1}--\eqref{eq:normuse3},
\eqref{eq:Gnq1}
and \eqref{eq:Gnq2}
shows
\eqref{eq:norm}.
Next observe 
for every
$ t \in (0,\infty) $,
$ q \in [0,\infty) $,
$
  \rho \in 
  ( - \infty, 
    \min( 
      \alpha + 1 , 
      \beta + \frac{ 1 }{ 2 } 
    ) 
  ) 
$
and every
$
  \delta \in ( \rho - 1, \infty )
$
that
$
  \left\| \cdot \right\|_{ t, q }^{
    \delta, \rho 
  }
  \colon
  G^2_{ q }( H_{ \rho }, V )
  \rightarrow
  [0,\infty)
$
is a semi-norm on
$
  G^2_{ q }( H_{ \rho }, V )
$.
In addition, note  
for every 
$ \rho, \delta \in \R $,
$ q \in [0,\infty) $,
$ t \in (0,\infty) $
and every
$
  \varphi \in
  G^0_{ q }( H_{ \rho }, V )
$
that if
$
  \| 
    K_t( \varphi ) 
  \|_{
    G^0_q( H_{ \delta }, V )
  }
  = 0
$,
then
$
  \sup_{ 
    x \in e^{ A t }( H_{ \delta } ) 
  }
  \|
    \varphi(x)
  \|_V
$ = 0.
The fact that
for every 
$ \rho, \delta \in \R $
the set
$ 
  e^{ A t }( H_{ \delta } )
$
is dense in
$
  H_{ \rho }
$
therefore shows 
for every
$ t \in (0,\infty) $,
$ q \in [0,\infty) $,
$ 
  \rho \in 
  ( - \infty, 
    \min( 
      \alpha + 1 , 
      \beta + \frac{ 1 }{ 2 } 
    ) 
  ) 
$
and every
$
  \delta \in ( \rho - 1, \infty )
$
that
$
  \left\| \cdot \right\|_{ t, q }^{    
    \delta, \rho
  }
  \colon
  G^2_{ q }( H_{ \rho }, V )
  \rightarrow
  [0,\infty)
$
is indeed a norm on
$
  G^2_{ q }( H_{ \rho }, V )
$.
The proof of 
Lemma~\ref{lem:norm}
is thus completed.
\end{proof}

In the next step 
we denote 
for every
$ t \in (0,\infty ) $,
$ q \in [0,\infty) $,
$ 
  \rho \in 
  ( - \infty, 
    \min( 
      \alpha + 1 , 
      \beta + \frac{ 1 }{ 2 } 
    ) 
  ) 
$,
$ 
  \delta \in ( \rho - 1 , \infty )
$ 
by
$
  (
    \mathcal{G}^{ 2, \delta }_{ t, q }( 
    H_{ \rho }, V ),
    \left\| \cdot 
    \right\|_{ t, q }^{ 
      \delta, \rho
    }
  )
$
the completion of the normed
$ \mathbb{R} $-vector space
$
  (
    G^2_{ q }( H_{ \rho }, V ),
$
$
    \left\| \cdot 
    \right\|_{ t, q }^{ 
      \delta, \rho
    }
  )
$.
The pairs
$
  (
    \mathcal{G}^{ 2, \delta }_{ t, q
    }( H_{ \rho }, V ),
    \left\| \cdot \right\|_{ t, q }^{ 
      \delta, \rho
    }
  )
$
for 
$ t \in (0,\infty ) $,
$ q \in [0,\infty) $,
$ \delta \in ( \rho - 1, \infty ) $
and
$ 
  \rho \in 
  ( - \infty, 
    \min( 
      \alpha + 1 , 
      \beta + \frac{ 1 }{ 2 } 
    ) 
  ) 
$
are thus 
$ \mathbb{R} $-Banach
spaces.

\begin{theorem}[Weak regularity
for $ P_t $, $ t \in (0,\infty) $]
\label{thm:continuity}
Assume that the setting in
Section~\ref{sec:setting}
is fulfilled.
Then
$
  P_t
  \in
  L( 
    \mathcal{G}^{ 2, \delta }_{ t, q - 2
    }( H_{ \rho }, V ) ,
    G^0_{ q }( H_{ \delta }, V )
  )
$
and
\begin{align}
&
  \| 
    P_t
  \|_{ 
    L\left( 
      \mathcal{G}^{ 2, \delta 
      }_{ t, q - 2
      }( H_{ \rho }, V ) ,
      G^0_{ q }( H_{ \delta }, V )
    \right)
  }
\\ & \leq
\nonumber
  \max\!\Big(
    1,
    \| F \|_{ 
      G^0_1( H_{ \rho }, H_{ \alpha } )
    }
    ,
    \| B \|_{ 
      G^0_1( H_{ \rho }, 
      HS( U_0, H_{ \beta } ) )
    }^2
  \Big)
  \max\!
  \left( 
  1,
  \sup_{ s \in (0, t) }
  \left[
  s^{ \max( \rho - \delta, 0 ) }
  \|
    P_s
  \|_{
    L( 
      G^0_{ q }( H_{ \rho }, V ) ,
      G^0_{ q }( H_{ \delta }, V ) 
    )
  }
  \right]
  \right)
  < \infty
\end{align}
for all 
$ t \in (0,\infty) $,
$
  \delta \in [ \gamma, \infty )
  \cap
  ( \rho - 1 , \infty )
$,
$ q \in [2,\infty) $
and all
$
  \rho \in 
  [ 
    \gamma, 
    \min( 
      \alpha + 1, 
      \beta + \frac{ 1 }{ 2 } 
    )
  )
$.
\end{theorem}
\begin{proof}[Proof
of Theorem~\ref{thm:continuity}]
Equation~\eqref{eq:mildP2}
implies
\begin{equation}
\label{eq:use1}
\begin{split}
&
  \| 
    P_t( \varphi )
  \|_{ 
    G^0_{ q }( H_{ \delta }, V )
  }
\\ & \leq
  \| 
    K_t( \varphi )
  \|_{ 
    G^0_{ q }( H_{ \delta }, V )
  }
\\ & \quad
  +
  \left(
  \sup_{ s \in (0,t) }
  \left[
  s^{ \max( \rho - \delta, 0 ) }
  \|
    P_s
  \|_{
    L( 
      G^0_{ q }( H_{ \rho }, V ) ,
      G^0_{ q }( H_{ \delta }, V ) 
    )
  }
  \right]
  \right)
  \left(
  \int_0^t
  \left( t - s 
  \right)^{ - \max( \rho - \delta, 0 ) }
  \|
    L^{ (0) }_s( \varphi )
  \|_{ 
    G_{ q }^0( H_{ \rho }, V )
  }
  \,
  ds
  \right)
\\ & \leq
  \max\!\left(
  1 ,
  \sup_{ s \in (0,t) }
  \left[
  s^{ \max( \rho - \delta, 0 ) }
  \|
    P_s
  \|_{
    L( 
      G^0_{ q }( H_{ \rho }, V ) ,
      G^0_{ q }( H_{ \delta }, V ) 
    )
  }
  \right]
  \right)
\\ & \quad
  \cdot
  \left(
  \| 
    K_t( \varphi )
  \|_{ 
    G^0_{ q }( H_{ \delta }, V )
  }
  +
  \int_0^t
  \left( t - s 
  \right)^{ \min( \delta - \rho, 0 ) }
  \|
    L^{ (0) }( K_s( \varphi ) )
  \|_{ 
    G_{ q }^0( H_{ \rho }, V )
  }
  \,
  ds
  \right)
\end{split}
\end{equation}
for all 
$ t \in (0,\infty) $,
$
  \varphi \in 
  G^2_{ q }( H_{ \rho }, V )
$,
$
  \delta \in [\gamma,\infty)
  \cap
  ( \rho - 1, \infty )
$,
$ 
  q \in [2,\infty) 
$
and all
$
  \rho \in 
  [ 
    \gamma, 
    \min( 
      \alpha + 1, 
      \beta + \frac{ 1 }{ 2 } 
    )
  )
$.
Inequality~\eqref{eq:use1}
and Lemma~\ref{kolm:bound0}
then complete
the proof of 
Theorem~\ref{thm:continuity}.
\end{proof}

Below we illustrate 
Theorem~\ref{thm:continuity}
through some consequences.
To do so, we need the following
elementary lemma for probability
measures on separable 
Hilbert spaces.

\begin{lemma}[An embedding
for probability measures]
\label{lem:embedding}
Let 
$ 
  ( 
    H, 
    \left\| \cdot \right\|_H,
    \left< \cdot, \cdot \right>_H
  )
$
and
$ 
  ( 
    V, 
    \left\| \cdot \right\|_V,
    \left< \cdot, \cdot \right>_V
  )
$
be separable real
Hilbert spaces
with $ V \neq \{ 0 \} $
and let
$ 
  \mu_1, \mu_2 \colon
  \mathcal{B}(H)
  \to [0,1]
$
be two probability measures
with
$
  \int_{ H }
  \varphi( x ) \, \mu_1( dx )
  =
  \int_{ H }
  \varphi( x ) \, \mu_2( dx )
$
for all infinitely often
Fr\'{e}chet differentiable 
functions
$ 
  \varphi \colon H \to V 
$
with compact support.
Then $ \mu_1 = \mu_2 $.
\end{lemma}

\begin{proof}[Proof
of Lemma~\ref{lem:embedding}]
First of all, 
we denote throughout
this proof
for every $ x \in H $
and every $ r \in (0,\infty) $
by
$
  B_r( x )
:=
  \{
    y \in H
    \colon
    \left\| x - y \right\|_H \leq r
  \}
$
the ball in $ H $ 
on $ x $ with radius $ r $.
Next let $ v_0 \in V $
be a vector 
which satisfies 
$ \left\| v_0 \right\|_V = 1 $.
Such a vector does indeed
since we assumed that
$ V \neq \{ 0 \} $.
Furthermore, let
$ 
  \psi_k \colon \R \to [0,1]
$,
$ k \in \N $,
be a sequence of infinitely
often differentiable functions
with
$
  \psi_k(x) = 1
$
for all $ x \in [-1,1] $
and all $ k \in \N $
and with
$
  \psi_k(x) = 0
$
for all 
$ 
  x \in 
  (- \infty, - 1 - \frac{ 1 }{ k } ) 
  \cup 
  (1 + \frac{ 1 }{ k }, \infty) 
$
and all $ k \in \N $.
Moreover, let
$ N \in \N $,
let
$ x_1, \dots, x_N \in H $
and let
$
  r_1, \dots, r_N \in (0,\infty)
$.
In the next step we 
define a sequence
$ \varphi_k \colon H \to V $,
$ k \in \N $,
of functions
by
\begin{equation}
  \varphi_k(x)
:=
  v_0 
  \cdot
  \prod_{ n = 1 }^N
  \psi_k\!\left(
    \tfrac{ 
      \left\| x - x_n \right\|_H^2 
    }{ 
      \left( r_n \right)^2
    }
  \right)
\end{equation}
for all $ x \in H $
and all $ k \in \N $.
Note for every $ k \in \N $
that $ \varphi_k $ is infinitely
often Fr\'{e}chet differentiable 
with a compact support.
Therefore, we obtain
\begin{equation}
\label{eq:mu1mu2}
  \int_{ H } \varphi_k(x) \, \mu_1( dx )
  =
  \int_{ H } \varphi_k(x) \, \mu_2( dx )
\end{equation}
for all $ k \in \N $.
In the next step observe that 
$
  \varphi_k(x) = v_0
$
for all 
$ 
  x \in     
  \cap_{ n = 1 }^N
  B_{ r_n }( v_n )
$
and all
$ k \in \N $,
that
$
  \sup_{ k \in \N }
  \sup_{ x \in H }
  \left\| \varphi_k(x) \right\|_V
  \leq 1
$
and that
\begin{equation}
  \lim_{ k \to \infty }
  \varphi_k(x)
  = 
  \begin{cases}
      1 & 
      \colon 
      x \in
      \cap_{ n = 1 }^N
      B_{ r_n }( v_n )
    \\
      0 & 
      \colon 
      x \in
      H \backslash
      \left(
      \cap_{ n = 1 }^N
      B_{ r_n }( v_n )
      \right)
  \end{cases}
\end{equation}
for all $ x \in H $.
Combining this and
\eqref{eq:mu1mu2}
with Lebesgue's theorem
on dominated convergence
then proves that
$
  \mu_1\big(
    \cap_{ n = 1 }^N
    B_{ r_n }( x_n )
  \big)
  =
  \mu_2\big(
    \cap_{ n = 1 }^N
    B_{ r_n }( x_n )
  \big)
$.
Combining this, the fact that the set
\begin{equation}
  \cup_{ M \in \N }
  \left\{
    \cap_{ m = 1 }^M
    B_{ s_m }( y_m )
    \subset H
    \colon
    s_1, \dots, s_M \in (0,\infty),
    y_1, \dots, y_M \in H
  \right\}
\end{equation}
%of subsets of $ H $
is a $ \cap $-stable
generator of the Borel 
sigma-algebra
$ \mathcal{B}(H) $
and the uniqueness theorem
for measures (see, e.g., 
Lemma~1.42 in Klenke~\cite{k08b})
then completes the proof
of Lemma~\ref{lem:embedding}.
\end{proof}

Let us now illustrate 
Theorem~\ref{thm:continuity}
by a simple application.
First, we denote by
$
  G^2_{ 0 }( 
    H_{ \gamma }, 
    \R 
  )'
  :=
  L(
    G^2_{ 0 }( 
      H_{ \gamma }, 
      \R 
    ),
    \R
  )
$
and
$
  \mathcal{G}^{ 2, \gamma }_{ t, 0
  }( 
    H_{ \gamma }, 
    \R 
  )'
  :=
  L(
    \mathcal{G}^{ 2, \gamma 
    }_{ t, 0 }( 
      H_{ \gamma }, 
      \R 
    ),
    \mathbb{R}
  )
$
for
$ t \in (0,\infty) $
the topological
dual spaces of
$
  G^2_{ 0 }( 
    H_{ \gamma }, 
    \R
  )
$
and
$
  \mathcal{G}^{ 2, \gamma }_{ t, 0
  }( 
    H_{ \gamma }, 
    \R 
  )
$
for $ t \in (0,\infty) $
respectively.
Moreover, we denote by
$
  \mathcal{M}_2( H_{ \gamma } )
$
the set of all
probability measures
$
  \mu \colon 
  \mathcal{B}( H_{ \gamma } )
  \to [0,1]
$
which satisfy
$
  \int_{ H_{ \gamma } } 
  \left\| x \right\|_{ 
    H_{ \gamma } 
  }^2
  \mu( dx ) < \infty
$
and we consider
the mapping
$
  \mathcal{I} \colon
  \mathcal{M}_2( H_{ \gamma } )
  \to
  G^2_{ 0 }( 
    H_{ \gamma }, 
    \R
  )'
$
given by
$
  ( \mathcal{I} \mu)( \varphi )
  =
  \int_{ 
    H_{ \gamma }
  } 
  \varphi( x ) \, \mu( dx )
$
for all
$ 
  \varphi \in G^2_0( H_{ \gamma }, \R )
$
and all 
$
  \mu \in
  \mathcal{M}_2( H_{ \gamma } )
$.
Lemma~\ref{lem:embedding}
then proves that
$ \mathcal{I} $ is injective
and through $ \mathcal{I} $
we can thus identify the
probability measures
$ \mathcal{M}_2( H_{ \gamma } ) $
with finite second moment
as a subset of linear forms in
$
  G^2_{ 0 }( 
    H_{ \gamma }, 
    \R
  )'
$.
Next note 
that
Proposition~\ref{prop:prop}
proves 
that
the probability 
measure
$
  \mathbb{P}_{ X_t }
$
of the solution process
of the SPDE~\eqref{eq:SPDE}
at every time $ t \in (0,T] $
has a finite second moment and
is thus in 
$ \mathcal{M}_2( H_{ \gamma } ) $.
Hence,
the linear form
$
  \mathcal{I}( 
    \mathbb{P}_{ X_t } 
  )
  =
  \int_{ H_{ \gamma } } 
  ( \cdot ) \,
  d\mathbb{P}_{ X_t }
  \in
  G^2_{ 0 }( 
    H_{ \gamma }, 
    \R
  )'
$
corresponding to
the probability
measure
$
  \mathbb{P}_{ X_t }
$
of the solution of the SPDE~\eqref{eq:SPDE}
at time 
$   
  t \in (0,T]
$ 
is in
$
  G^2_{ 0 }( 
    H_{ \gamma }, 
    V 
  )'
$.
In addition, observe that
Theorem~\ref{thm:continuity},
in particular,
implies that
$
  \int_{ H_{ \gamma } } 
  ( \cdot ) \,
  d\mathbb{P}_{ X_t }
  \in
  \mathcal{G}^{ 2, \gamma }_{ t, 0 }( 
    H_{ \gamma }, 
    V 
  )'
$
for all $ t \in (0,T] $.
Moreover, 
note that
Lemma~\ref{lem:norm} implies
that
$
  \mathcal{G}^{ 2, \gamma 
  }_{ t, 0 }( 
    H_{ \gamma }, 
    V 
  )'
  \subset
  G^2_{ 0 }( 
    H_{ \gamma }, 
    V 
  )'
$
continuously
for all
$ t \in (0,\infty) $.
Theorem~\ref{thm:continuity}
thus proves
for every $ t \in (0,T] $
that
$
  \mathcal{I}( 
    \mathbb{P}_{ X_t } 
  )
  =
  \int_{ H_{ \gamma } } 
  ( \cdot ) \,
  d\mathbb{P}_{ X_t }
$
does not only lie in
$
  G^2_{ 0 }( 
    H_{ \gamma }, 
    V 
  )'
$
but also
in the smaller space
$
  \mathcal{G}^{ 2, \gamma 
  }_{ t, 0 }( 
    H_{ \gamma }, 
    V 
  )'
$
too.
In this sense 
Theorem~\ref{thm:continuity}
proves more regularity
of the probability measures
$
  \mathbb{P}_{ X_t } 
$,
$ t \in (0,T] 
$,
of the solution of the 
SPDE~\eqref{eq:SPDE}.
It thus establishes ``weak
regularity'' for the solution
of the SPDE~\eqref{eq:SPDE}.
In the remainder
of this subsection some further
consequences of
Theorem~\ref{thm:continuity}
are derived.

\begin{cor}
\label{cor:weakest}
Assume that the setting
in Section~\ref{sec:setting}
is fulfilled,
assume 
$ \alpha \leq \gamma $,
assume
$ \beta \leq \gamma $,
let
$
  \rho \in 
  [ 
    \gamma, 
    \min( 
      \alpha + 1, 
      \beta + \frac{ 1 }{ 2 } 
    )
  )
$
be a real number,
let
$
  (
    \tilde{H}, 
    \left< \cdot , \cdot \right>_{ \tilde{H} },  
    \left\| \cdot \right\|_{ \tilde{H} }
  )
$
be a separable $ \mathbb{R} $-Hilbert
space, let 
$ 
  R, \tilde{R} \in
  L( H_{ \rho }, \tilde{H} )
$, let
$
  \varphi \in
  C^2_{ Lip }( \tilde{H}, V )
$
and let
$
  \psi \colon H_{ \rho }
  \rightarrow
  \tilde{H}
$
be given by
$
  \psi(x)
  =
  \varphi( R x )
  -
  \varphi( \tilde{R} x )
$
for all $ x \in H_{ \rho } $.
Then
\begin{equation}
\label{eq:weak_crucial}
\begin{split}
&
  \|
    P_t( \psi )
  \|_{
    G^0_{ q }( H_{ \delta }, V )
  }
\\ & \leq
  \max\!\left(
    1,
    \| F \|_{ 
      G^0_1( H_{ \rho }, H_{ \alpha } )
    }
    ,
    \| B \|_{ 
      G^0_1( H_{ \rho }, 
      HS( U_0, H_{ \beta } ) )
    }^2
  \right)
  \max\!
  \left( 
  1,
  \sup_{ s \in (0, t) }
  \left[
  s^{ \max( \rho - \delta, 0 ) }
  \|
    P_s
  \|_{
    L( 
      G^0_{ q }( H_{ \rho }, V ) ,
      G^0_{ q }( H_{ \delta }, V ) 
    )
  }
  \right]
  \right)
\\ & \quad
  \cdot
  \| \varphi \|_{
    C^2_{ Lip }( \tilde{H}, V )
  }
  \,
  \frac{
    \max( t, 1 ) 
  }{
    t^{ \max( r + \rho - \delta, 0 ) }
  }
  \,
  \| 
    R - \tilde{R} 
  \|_{
    L( H_{ \rho + r }, \tilde{H} )
  }
  \left[
    \sup_{ 
      u \in [ \rho - \delta, 1] 
      \cup [0,1]
    }
    \sup_{ s \in (0,t] }
    \left(
    s^{ \max( u, 0 ) }
    \|
      e^{ A s }
    \|_{ 
      L( H, H_{ u } )
    }
    \right)
  \right]^3
\\ & \quad
  \cdot
  \big[
    1 +
    \| R \|_{ L( H_{ \rho }, \tilde{H} ) 
    } +
    \| \tilde{R} 
    \|_{ L( H_{ \rho }, \tilde{H} ) 
    }
  \big]^2
  \left(
    1
    +
    \int_0^1
    \left( 1 - s \right)^{
      \min( \delta - \rho, 0 )
    }
    s^{
      \left[
        \min\left(
          \alpha - \rho ,
          2 \beta - 2 \rho 
        \right) - r 
      \right]
    }
    \, ds
  \right)
  < \infty
\end{split}
\end{equation}
for all 
$ t \in (0,\infty) $,
$ q \in [3,\infty) $,
$ 
  \delta \in [\gamma, \infty) 
$
and all
$
  r \in  
  [ 
    0,
    \min( 1 + \alpha - \rho,
      1 + 2 \beta - 2 \rho
    )
  )
$.
In particular,
we have
\begin{equation}
\label{eq:weak_crucialB}
\begin{split}
&
  \sup_{ 
    \Phi \in 
    C^2_{ Lip }( H_{ \rho }, V ) 
    \backslash \{ 0 \}
  }
  \sup_{ 
    S \in L( H_{ \rho } ) 
  }
  \sup_{ 
    t \in (0,T] 
  }
  \left(
  \frac{
    t^{ \max( r + \rho - \delta, 0 ) }
    \,
    \|
      P_t( \Phi ) - 
      P_t\big( \Phi( S( \cdot ) ) \big)
    \|_{
      G^0_{ 3 }( H_{ \delta }, V )
    }
  }{
    \| 
      I - S 
    \|_{
      L( H_{ \rho + r }, H_{ \rho } )
    }
    \,
    \big(
      1 +
      \| S \|_{ L( H_{ \rho } ) 
      }^2 
    \big) 
    \,
    \| \varphi \|_{
      C^2_{ Lip }( H_{ \rho }, V )
    }
  }
  \right)
  < \infty
\end{split}
\end{equation}
for all 
$ 
  \delta \in [\gamma, \infty) 
$
and all
$
  r \in  
  [ 
    0,
    \min( 1 + \alpha - \rho,
      1 + 2 \beta - 2 \rho
    )
  )
$.
\end{cor}

\begin{proof}[Proof
of Corollary~\ref{cor:weakest}]
Throughout this the proof
the real numbers
$ \kappa_{ r, t } \in [1,\infty) $,
$ r \in [0,\infty) $,
$ t \in (0,\infty) $,
defined by
\begin{equation}
  \kappa_{ r, t }
:=
    \sup_{ u \in [-r,1] }
    \sup_{ s \in (0,t] }
    \left(
    s^{ \max( u, 0 ) }
    \|
      e^{ A s }
    \|_{ 
      L( H, H_{ u } )
    }
    \right)
  < \infty
\end{equation}
for all $ r \in [0,\infty) $ 
are used.
The quantities 
$ \kappa_{ r, t } $, 
$ r \in [0,\infty) $,
$ t \in (0,\infty) $,
are indeed
finite since $ e^{ A t } $,
$ t \in [0,\infty) $,
is an analytic semigroup.
The estimate
\begin{equation}
  \| 
    e^{ A t } 
  \|_{ L( H_{ a }, H_{ b } ) }
  =
  \| 
    e^{ A t } 
  \|_{ L( H, H_{ ( b - a ) } ) }
  \leq
  \kappa_{ 
    \max( a - b, 0 ), t
  }
  \,
  t^{ 
    \min\left( a - b, 0\right)
  }
\end{equation}
for all 
$ t \in (0,\infty) $
and all $ a, b \in \R $
with $ b - a \leq 1 $ then 
shows 
\begin{equation}
\label{eq:weakest1}
\begin{split}
&
  \|
    K_t( \psi )
  \|_{ 
    G_{q}^0( H_{ \delta }, V )
  }
\leq
  \| \varphi' \|_{
    L^{ \infty }( \tilde{H}, 
      L( \tilde{H}, V) 
    )
  }
  \,
  \| R - \tilde{R} \|_{
    L( H_{ r }, \tilde{H} )
  }
  \,
  \| e^{ A t } \|_{
    L( H_{ \delta }, H_r )
  } 
\\ & \leq
  \| \varphi \|_{
    C^2_{ Lip }( \tilde{H}, V )
  }
  \,
  \| R - \tilde{R} \|_{
    L( H_{ r }, \tilde{H} )
  }
  \,
  \kappa_{
    \max( \delta - r, 0 ), t
  }
  \,
  t^{ \min( \delta - r, 0 ) }
\end{split}
\end{equation}
and
\begin{equation}
\begin{split}
&
  \|
    ( K_t \psi )'
  \|_{ 
    G_{ q - 1 }^0( H_{ \rho }, 
      L( H_{ \alpha }, V )
    )
  }
\\ & \leq
  \sup_{ x \in H_{ \rho } }
  \frac{
    \|
      ( \varphi'( R e^{ A t } x ) - 
      \varphi'( \tilde{R} e^{ A t } x ) )
      R e^{ A t } 
    \|_{ 
      L( H_{ \alpha }, V )
    }
  }{
    \big(   
      1 + 
      \| x \|_{ H_{ \rho } } 
    \big)^{ (q - 1) }
  }
\\ & \quad +
  \sup_{ x \in H_{ \rho } }
  \frac{
    \|
      \varphi'( \tilde{R} e^{ A t } x )
        (R - \tilde{R}) 
        e^{ A t } 
    \|_{ 
      L( H_{ \alpha }, V )
    }
  }{
    \big(   
      1 + 
      \| x \|_{ H_{ \rho } }
    \big)^{ (q-1) }
  }
\\ & \leq
  \| \varphi'' \|_{
    L^{ \infty }( \tilde{H},
      L^{ (2) }( \tilde{H}, V )
    )
  }
  \,
  \| R \|_{
    L( H_{ \rho }, \tilde{H} )
  }
  \,
  \| R - \tilde{R} \|_{
    L( H_{ r }, \tilde{H} )
  }
  \,
  \underbrace{
  \| e^{ A t } \|_{
    L( H_{ \alpha }, H_{ \rho } )
  }
  \,
  \| e^{ A t } \|_{
    L( H_{ \rho }, H_r )
  }
  }_{
    \leq
    \left(
      \kappa_{ 0, t } 
    \right)^2
    t^{
      \left(
        \alpha - r
      \right)
    }
  }
\\ & \quad +
  \| \varphi' \|_{
    L^{ \infty }( \tilde{H},
      L( \tilde{H}, V )
    )
  }
  \,
  \| R - \tilde{R} \|_{
    L( H_{ r }, \tilde{H} )
  }
  \,
  \underbrace{
  \| e^{ A t } \|_{
    L( H_{ \alpha }, H_{ r } )
  }
  }_{
    \leq
    \kappa_{ 0, t } \,
    t^{
      \left(
        \alpha - r
      \right)
    }
  }
\end{split}
\end{equation}
and
\begin{equation}
\label{eq:weakest3}
\begin{split}
&
  \|
    ( K_t \psi )''
  \|_{ 
    G_{q-2}^0( H_{ \rho }, 
      L^{(2)}( H_{ \beta }, V )
    )
  }
\\ & \leq
  \sup_{ x \in H_{ \rho } }
  \sup_{ 
    \substack{
    \| v \|_{ H_{ \beta } } = 
    \| w \|_{ H_{ \beta } } = 1
    }
  }
  \frac{
    \big\|
      \big( \varphi''( R e^{ A t } x ) - 
      \varphi''( \tilde{R} e^{ A t } x ) \big)( 
        R e^{ A t } v,
        R e^{ A t } w 
      )
    \big\|_{ 
      V
    }
  }{
    \big(   
      1 + 
      \| x \|_{ H_{ \rho } } 
    \big)^{ (q-2) }
  }
\\ & \quad +
  \sup_{ x \in H_{ \rho } }
  \sup_{ 
    \substack{
    \| v \|_{ H_{ \beta } } = 
    \| w \|_{ H_{ \beta } } = 1
    }
  }
  \frac{
    \big\|
      \varphi''( \tilde{R} e^{ A t } x )\big( 
        (R - \tilde{R}) 
        e^{ A t } v,
        R e^{ A t } w 
      \big)
    \big\|_{ 
      V
    }
  }{
    \big(   
      1 + 
      \| x \|_{ H_{ \rho } } 
    \big)^{ (q - 2) }
  }
\\ & \quad +
  \sup_{ x \in H_{ \rho } }
  \sup_{ 
    \substack{
    \| v \|_{ H_{ \beta } } = 
    \| w \|_{ H_{ \beta } } = 1
    }
  }
  \frac{
    \big\|
      \varphi''( \tilde{R} e^{ A t } x )\big( 
        \tilde{R} e^{ A t } v ,
        (R - \tilde{R}) 
        e^{ A t } w
      \big)
    \big\|_{ 
      V
    }
  }{
    \big(   
      1 + 
      \| x \|_{ H_{ \rho } } 
    \big)^{ (q - 2) }
  }
\\ & \leq
  \left[
    \sup_{
      \substack{
        x, y \in \tilde{H}
      \\
        x \neq y
      }
    }
    \frac{ 
      \| \varphi''(x) - \varphi''(y) \|_{ 
        L^{ (2) }( \tilde{H}, V )
      }
    }{
      \| x - y \|_{ \tilde{H} }
    }
  \right]
  \| R \|_{
    L( H_{ \rho }, \tilde{H} )
  }^2
  \,
  \| R - \tilde{R} \|_{
    L( H_{ r }, \tilde{H} )
  }
  \,
  \underbrace{
  \| e^{ A t } \|_{
    L( H_{ \beta }, H_{ \rho } )
  }^2
  \,
  \| e^{ A t } \|_{
    L( H_{ \rho }, H_r )
  }
  }_{
    \leq
    \left(
      \kappa_{ 
        0, t
      } 
    \right)^3
    t^{
      \left(
        2 \beta - \rho - r 
      \right)
    }
  }
\\ & \quad +
  \| \varphi'' \|_{
    L^{ \infty }( \tilde{H},
      L^{ (2) }( \tilde{H}, V )
    )
  }
  \left[
    \| R \|_{
      L( H_{ \rho }, \tilde{H} )
    }
    +
    \| \tilde{R} \|_{
      L( H_{ \rho }, \tilde{H} )
    }
  \right]
  \| R - \tilde{R} \|_{
    L( H_{ r }, \tilde{H} )
  }
  \,
  \underbrace{
    \| e^{ A t } \|_{
      L( H_{ \beta }, H_{ \rho } )
    }
    \,
    \| e^{ A t } \|_{
      L( H_{ \beta }, H_{ r } )
    }
  }_{
    \leq
    \left(
      \kappa_{ 
        0, t
      } 
    \right)^2
    t^{
      \left(
        2 \beta - \rho - r 
      \right)
    }
  }
\end{split}
\end{equation}
for all $ q \in [3,\infty) $,
$ t \in (0,T] $,
$
  \delta \in [\gamma,\infty) 
$
and all
$
  r \in  
  [ 
    \rho,
    \min( 1 + \alpha ,
      1 + 2 \beta - \rho
    )
  )
$.
Combining
\eqref{eq:weakest1}--\eqref{eq:weakest3}
implies
\begin{equation}
\label{eq:specialnorm_est}
\begin{split}
&
  \left\|
    \psi
  \right\|_{
    t, q - 2
  }^{
    \delta, \rho
  }
  \leq
  \| \varphi \|_{
    C^2_{ Lip }( \tilde{H}, V )
  }
  \,
  \| R - \tilde{R} \|_{
    L( H_{ r }, \tilde{H} )
  }
  \left[
    \kappa_{ 
      \max( \delta - r, 0 ), t
    } 
  \right]^3
  \big[
    1 +
    \| R \|_{ L( H_{ \rho }, \tilde{H} ) 
    } +
    \| \tilde{R} 
    \|_{ L( H_{ \rho }, \tilde{H} ) 
    }
  \big]^2
\\ & \quad
  \cdot
  \left(
    t^{ \min( \delta - r, 0 ) }
    +
    \int_0^t
    \left( t - s \right)^{
      \min( \delta - \rho, 0 )
    }
    \max\!\left(
    s^{
      \left(
        \alpha - r 
      \right)
    }
    ,
    s^{
      \left(
        2 \beta - \rho - r 
      \right)
    }
    \right)
    ds
  \right)
\\ & \leq
  \| \varphi \|_{
    C^2_{ Lip }( \tilde{H}, V )
  }
  \,
  \| R - \tilde{R} \|_{
    L( H_{ r }, \tilde{H} )
  }
  \left[
    \kappa_{ 
      \max( \delta - r, 0 ), t
    } 
  \right]^3
  \big[
    1 +
    \| R \|_{ L( H_{ \rho }, \tilde{H} ) 
    } +
    \| \tilde{R} 
    \|_{ L( H_{ \rho }, \tilde{H} ) 
    }
  \big]^2
    \max\!\big( 
      1, t^{
        | 2 \beta - \rho - \alpha |
      }
    \big)  
\\ & \quad
  \cdot
  \left(
    t^{ \min( \delta - r, 0 ) }
    +
    t^{
      \left[
        \min( \delta - \rho, 0 )
        +
        \min\left(
          \alpha , 
          2 \beta - \rho  
        \right)
        + 1 - r
      \right]
    }
    \int_0^1
    \left( 1 - s \right)^{
      \min( \delta - \rho, 0 )
    }
    s^{
      \min\left(
        \alpha - r ,
        2 \beta - \rho - r 
      \right)
    }
    \, ds
  \right)
\\ & \leq
  \| \varphi \|_{
    C^2_{ Lip }( \tilde{H}, V )
  }
  \,
  \| R - \tilde{R} \|_{
    L( H_{ r }, \tilde{H} )
  }
  \left[
    \kappa_{ 
      \max( \delta - r, 0 )
    } 
  \right]^3
  \big[
    1 +
    \| R \|_{ L( H_{ \rho }, \tilde{H} ) 
    } +
    \| \tilde{R} 
    \|_{ L( H_{ \rho }, \tilde{H} ) 
    }
  \big]^2
  \max( t, 1 ) 
\\ & \quad
  \cdot
  t^{ \min( \delta - r, 0 ) }
  \left(
    1
    +
    \int_0^1
    \left( 1 - s \right)^{
      \min( \delta - \rho, 0 )
    }
    s^{
      \min\left(
        \alpha - r ,
        2 \beta - \rho - r 
      \right)
    }
    \, ds
  \right)
\end{split}
\end{equation}
for all 
$ t \in (0,T] $,
$ q \in [3,\infty) $,
$ 
  \delta \in [\gamma, \infty) 
$
and all
$
  r \in  
  [ 
    \rho,
    \min( 1 + \alpha ,
      1 + 2 \beta - \rho
    )
  )
$.
Next observe that
Theorem~\ref{thm:continuity}
implies that
\begin{equation}
\label{eq:theorem_reg_cons}
\begin{split}
  \|
    P_t( \psi )
  \|_{
    G^0_{ q }( H_{ \delta }, V )
  }
& \leq
  \max\!\left(
    1,
    \| F \|_{ 
      G^0_1( H_{ \rho }, H_{ \alpha } )
    }
    ,
    \| B \|_{ 
      G^0_1( H_{ \rho }, 
      HS( U_0, H_{ \beta } ) )
    }^2
  \right)
\\ & \quad
  \cdot
  \max\!
  \left( 
  1,
  \sup_{ s \in (0, t) }
  \left[
  s^{ \max( \rho - \delta, 0 ) }
  \|
    P_s
  \|_{
    L( 
      G^0_{ q }( H_{ \rho }, V ) ,
      G^0_{ q }( H_{ \delta }, V ) 
    )
  }
  \right]
  \right)
  \cdot
  \left\|
    \psi
  \right\|_{
    t, q - 2
  }^{
    \delta, \rho
  }
\end{split}
\end{equation}
for all 
$ t \in (0,T] $,
$ q \in [3,\infty) $
and all
$ 
  \delta \in [\gamma, \infty) 
$.
Combining
\eqref{eq:specialnorm_est}
and \eqref{eq:theorem_reg_cons}
then shows \eqref{eq:weak_crucial}.
Inequality~\eqref{eq:weak_crucial}
implies \eqref{eq:weak_crucialB}.
This completes the proof
of Corollary~\ref{cor:weakest}.
\end{proof}

In the remainder of this subsection,
Corollary~\ref{cor:weakest}
is illustrated
by three simple consequences
(Corollary~\ref{cor:spatial},
Corollary~\ref{cor:temporal}
and 
Corollary~\ref{cor:galerkin}).
Corollary~\ref{cor:spatial}
follows immediately
from inequality~\eqref{eq:weak_crucialB}
in Corollary~\ref{cor:weakest}
and its proof is therefore
omitted.

\begin{cor}[Spatial
weak semigroup regularity]
\label{cor:spatial}
Assume that the setting
in Section~\ref{sec:setting}
is fulfilled and
assume
$ \alpha \leq \gamma $
and
$ \beta \leq \gamma $.
Then
\begin{equation}
\label{eq:weakspatialsemigroup1}
  \sup_{
    \substack{
      \varphi \in 
      C^2_{ Lip }( H_{ \rho }, V )
      \backslash \{ 0 \}
    }
  }
  \sup_{ 
    t \in (0,T]
  }
  \sup_{
    h \in (0,T]
  }
  \left(
  \frac{
    t^{ 
      \max( \rho - \delta + r, 0 ) 
    } \,
    \|
      P_{ t }( K_h( \varphi ) ) 
      - 
      P_{ t }( \varphi )
    \|_{ 
      G^0_3( H_{ \delta }, V )
    }
  }{
    h^{ r } \,
    \| \varphi \|_{ 
      C^2_{ Lip }( H_{ \rho }, V )
    }
  } 
  \right)
  < \infty
\end{equation}
for all
$ 
  \delta \in [\gamma, \infty)
$,
$ 
  r \in 
  [0,1+ \alpha - \rho) \cap
  [0,1+2 \beta - 2 \rho) 
$
and all
$
  \rho \in 
  [
    \gamma,
    \alpha + 1
  )
  \cap
  [
    \gamma,
    \beta + \frac{ 1 }{ 2 }
  )
$.
In particular, if 
the real number 
$ p \in [2,\infty) $
in Assumption~\ref{initial}
satisfies
$ p \geq 3 $,
then
\begin{equation}
  \sup_{
    \substack{
      \varphi \in 
      C^2_{ Lip }( H_{ \gamma }, V )
      \backslash \{ 0 \}
    }
  }
  \sup_{
    t \in [0,T]
  }
  \sup_{ 
    h \in (0,T]
  }
  \left(
  \frac{
    t^{ r } \,
    \big\|
    \mathbb{E}\big[ 
      \varphi( e^{ A h } X_{ t } )
    \big]
    -
    \mathbb{E}\big[ 
      \varphi( X_{ t } )
    \big]
    \big\|_V
  }{
    h^r \,
    \| \varphi \|_{ 
      C^2_{ Lip }( H_{ \gamma }, V )
    }
  } 
  \right)
  < \infty
\end{equation}
for all 
$ 
  r \in 
  [0,1+ \alpha - \gamma) \cap
  [0,1+2 \beta - 2 \gamma) 
$.
\end{cor}

\begin{cor}[Temporal
weak regularity]
\label{cor:temporal}
Assume that the setting
in Section~\ref{sec:setting}
is fulfilled and
assume
$ \alpha \leq \gamma $
and
$ \beta \leq \gamma $.
Then
\begin{equation}
  \sup_{ 
    \substack{
      t_1, t_2 \in (0,T]
    \\
      t_1 < t_2
    }
  }
  \left(
  \frac{
    | t_1 |^{
      \max( \rho - \delta + r, 0 )
    }
    \,
    \|
      P_{ t_2 }
      - 
      P_{ t_1 }
    \|_{ 
      L( C^2_{ Lip }( H_{ \rho }, V ),
      G^0_3( H_{ \delta }, V ) )
    }
  }{
    | t_2 - t_1 |^r
  } 
  \right)
  < \infty
\end{equation}
for all
$ 
  \delta \in [\gamma, \infty)
$,
$ 
  r \in 
  [0,1+ \alpha - \rho) \cap
  [0,1+2 \beta - 2 \rho) 
$
and all
$
  \rho \in 
  [
    \gamma,
    \alpha + 1
  )
  \cap
  [
    \gamma,
    \beta + \frac{ 1 }{ 2 }
  )
$.
In particular, if 
the real number 
$ p \in [2,\infty) $
in Assumption~\ref{initial}
satisfies
$ p \geq 3 $,
then
\begin{equation}
  \sup_{
    \substack{
      \varphi \in 
      C^2_{ Lip }( H_{ \gamma }, V )
      \backslash \{ 0 \}
    }
  }
  \sup_{ 
    \substack{
      t_1, t_2 \in [0,T]
    \\
      t_1 \neq t_2
    }
  }
  \left(
  \frac{
    |t_1|^{ r } \,
    \big\|
    \mathbb{E}\big[ 
      \varphi( X_{ t_2 } )
    \big]
    -
    \mathbb{E}\big[ 
      \varphi( X_{ t_1 } )
    \big]
    \big\|_V
  }{
    | t_2 - t_1 |^r \,
    \| \varphi \|_{ 
      C^2_{ Lip }( H_{ \gamma }, V )
    }
  } 
  \right)
  < \infty
\end{equation}
for all 
$ 
  r \in 
  [0,1+ \alpha - \gamma) \cap
  [0,1+2 \beta - 2 \gamma) 
$.
\end{cor}

\begin{proof}[Proof
of Corollary~\ref{cor:temporal}]
First, define real
numbers
$ 
  c_{ \rho, \delta, r } \in [0,\infty) 
$, 
$ 
  r \in 
  [0,1+ \alpha - \rho) \cap
  [0,1+2 \beta - 2 \rho) 
$,
$
  \delta \in [\gamma, \infty)
$,
$
  \rho \in 
  [
    \gamma,
    \alpha + 1
  )
  \cap
  [
    \gamma,
    \beta + \frac{ 1 }{ 2 }
  )
$,
through
\begin{equation}
  c_{ \rho, \delta, r }
:=
  \sup_{
    \substack{
      \varphi \in 
      C^2_{ Lip }( H_{ \rho }, V )
      \backslash \{ 0 \}
    }
  }
  \sup_{ 
    t \in (0,T]
  }
  \sup_{
    h \in (0,T]
  }
  \left(
  \frac{
    t^{ 
      \max( \rho - \delta + r, 0 ) 
    } \,
    \|
      P_{ t }( K_h( \varphi ) ) 
      - 
      P_{ t }( \varphi )
    \|_{ 
      G^0_3( H_{ \delta }, V )
    }
  }{
    h^{ r } \,
    \| \varphi \|_{ 
      C^2_{ Lip }( H_{ \rho }, V )
    }
  } 
  \right)
  < \infty
\end{equation}
for all
$ 
  r \in 
  [0,1+ \alpha - \rho) \cap
  [0,1+2 \beta - 2 \rho) 
$,
$ 
  \delta \in [\gamma, \infty)
$
and all
$
  \rho \in 
  [
    \gamma,
    \alpha + 1
  )
  \cap
  [
    \gamma,
    \beta + \frac{ 1 }{ 2 }
  )
$.
Corollary~\ref{cor:spatial}
shows that these real numbers
are indeed finite.
In the next step we combine
\eqref{eq:mildP2}
and the definition
of  
$ 
  c_{ \rho, \delta, r } \in [0,\infty) 
$
to obtain that
\begin{equation}
\label{eq:esttemporal}
\begin{split}
&
  \frac{
    | t_1 
    |^{
      \max\left(
        \rho - \delta + r, 0
      \right)
    }
  \left\|
  ( 
    P_{ t_2 } \varphi 
  )( x )
  -
  ( 
    P_{ t_1 } \varphi 
  )( x )
  \right\|_{ V }
  }{
    | t_2 - t_1 |^r 
  }
\\ & \leq
  \frac{
    | t_1 
    |^{
      \max\left(
        \rho - \delta + r, 0
      \right)
    }
    \left\|
    (
      P_{ t_1 } K_{ ( t_2 - t_1 ) } 
      \varphi 
    )(x)
    -
    ( 
      P_{ t_1 } \varphi 
    )( x )
    \right\|_V
  }{
    | t_2 - t_1 |^r 
  }
+
  \int_{ t_1 }^{ t_2 }
  \left[
  \frac{
    s^{
      \max\left(
        \rho - \delta + r, 0
      \right)
    }
    \,
    \|
      ( 
        P_s \,
        L^{ (0) }_{ ( t_2 - s ) } \varphi
      ) ( x )
    \|_V
  }{
    | t_2 - t_1 |^r 
  }
  \right]
  ds
\\ & \leq
    c_{ \rho, \delta, r }
    \left[
      1 + \| x \|_{ H_{ \delta } }
    \right]^{ 3 }
    \|
      \varphi
    \|_{ 
      C^2_{ Lip }( H_{ \rho }, V )
    }
+ 
  \left( 1 + T \right)
  \int_{ t_1 }^{ t_2 }
  \left[
  \frac{
    s^{
      \max\left(
        \rho - \delta, 0
      \right)
    }
    \,
    \|
      ( 
        P_s \,
        L^{ (0) }_{ ( t_2 - s ) } \varphi
      ) ( x )
    \|_V
  }{
    | t_2 - t_1 |^{ 
      \min( 
        1 + \alpha - \rho, 
        1 + 2 \beta - 2 \rho 
      )
    } 
  }
  \right]
  ds
\\ & \leq
    c_{ \rho, \delta, r }
    \left[
      1 + \| x \|_{ H_{ \delta } }
    \right]^{ 3 }
    \|
      \varphi
    \|_{ 
      C^2_{ Lip }( H_{ \rho }, V )
    }
\\ & \quad + 
  \left( 1 + T \right)
  \left[
    \sup_{
      t \in (0,T]
    }
    \sup_{
      s \in (0, t)
    }
    \left(
    s^{
      \max\left(
        \rho - \delta, 0
      \right)
    }
    \,
    | t - s |^{ 
      \max( 
        \rho - \alpha, 
        2 \rho - 2 \beta 
      )
    } 
    \,
    \|
      ( 
        P_s \,
        L^{ (0) }_{ t - s } \varphi
      ) ( x )
    \|_V
    \right)
  \right]
\end{split}
\end{equation}
for all $ t_1, t_2 \in (0,T] $
with $ t_1 < t_2 $,
$ 
  x \in H_{ \delta } 
$,
$ 
  r \in 
  [0,1 + \alpha - \rho) \cap
  [0,1 + 2 \beta - 2 \rho) 
$,
$ 
  \delta \in [\gamma, \infty)
$,
$
  \varphi \in 
  C^2_{ Lip }( H_{ \rho }, V )
$
and all
$
  \rho \in 
  [
    \gamma,
    \alpha + 1
  )
  \cap
  [
    \gamma,
    \beta + \frac{ 1 }{ 2 }
  )
$.
Furthermore, observe that
Lemma~\ref{kolm:bound0} 
shows that
\begin{equation}
\label{eq:esttemporal2}
\begin{split}
&
  \sup_{ s \in (0,t) }
  \left[
  s^{ 
    \max( \rho - \delta , 0 ) 
  }
  \left( t - s \right)^{ 
    \max( \rho - \alpha, 2 \rho - 2 \beta )
  }
  \,
  \|
    P_{ s } (
    L^{ (0) }_{ t - s }( \varphi ) )
  \|_{
    G^0_3( H_{ \delta }, V )
  }
  \right]
\\ & \leq
  \sup_{ s \in (0, t) }
  \left[
  s^{ 
    \max( \rho - \delta , 0 ) 
  }
  \|
    P_{ s }
  \|_{
    L( 
      G^0_3( H_{ \rho }, V ) ,
      G^0_3( H_{ \delta }, V ) 
    )
  }
  \left( t - s \right)^{ 
    \max( \rho - \alpha, 2 \rho - 2 \beta )
  }
  \|
    L^{ (0) }_{ t - s }( \varphi ) )
  \|_{
    G^0_3( H_{ \delta }, V ) 
  }
  \right]
\\ & \leq
  \left[
    \sup_{ s \in (0, t) }
    s^{ 
      \max( \rho - \delta , 0 ) 
    }
    \|
      P_{ s }
    \|_{
      L( 
        G^0_3( H_{ \rho }, V ) ,
        G^0_3( H_{ \delta }, V ) 
      )
    }
  \right]
  \max\!\big(
    \| F \|_{ 
      G^0_1( H_{ \rho }, H_{ \alpha } )
    }
    ,
    \| B \|_{ 
      G^0_1( H_{ \rho }, 
      HS( U_0, H_{ \beta } ) )
    }^2
  \big)
\\ & \quad
  \cdot
  \left[
  \sup_{ s \in (0,t) }
  s^{ 
    \max( \rho - \alpha, 2 \rho - 2 \beta )
  }
  \Big(
  \|
    ( K_s \varphi )'
  \|_{ 
    G_{ 2 }^0( H_{ \rho },
      L( H_{ \alpha }, V )  
    )
  }
  +
  \|
    ( K_s \varphi )''
  \|_{ 
    G_{ 1 }^0( H_{ \rho },
      L^{(2)}( H_{ \beta }, V )  
    )
  } 
  \Big)
  \right]
\\ & \leq
  \left[
    \sup_{ s \in (0, T] }
    s^{ 
      \max( \rho - \delta , 0 ) 
    }
    \|
      P_{ s }
    \|_{
      L( 
        G^0_3( H_{ \rho }, V ) ,
        G^0_3( H_{ \delta }, V ) 
      )
    }
  \right]
  \max\!\big(
    \| F \|_{ 
      G^0_1( H_{ \rho }, H_{ \alpha } )
    }
    ,
    \| B \|_{ 
      G^0_1( H_{ \rho }, 
      HS( U_0, H_{ \beta } ) )
    }^2
  \big)
\\ & \quad
  \cdot
  \max\!\left( 
    1, 
    \sup_{ s \in [0,T] }
    \| e^{ A s } \|_{ L(H) }^2 
  \right)
\\ & \quad
  \cdot
  \left[
    \sup_{ s \in (0,T] }
    s^{ 
      \max( \rho - \alpha, 2 \rho - 2 \beta )
    }
    \left(
    \|
      e^{ A s }
    \|_{ L(H_{ \alpha }, H_{ \rho }) }
    \,
    \|
      \varphi'
    \|_{
      G^0_2( 
        H_{ \rho }, L( H_{ \rho }, V) 
      )
    } 
    +
    \|
      e^{ A s }
    \|_{ L( H_{ \beta }, H_{ \rho } ) 
    }^2
    \,
    \|
      \varphi''
    \|_{
      G^0_1( 
        H_{ \rho }, 
        L^{(2)}( H_{ \rho }, V) 
      )
    } 
    \right)
  \right]
\\ & \leq
  \left[
    \sup_{ s \in (0, T] }
    s^{ 
      \max( \rho - \delta , 0 ) 
    }
    \,
    \|
      P_{ s }
    \|_{
      L( 
        G^0_3( H_{ \rho }, V ) ,
        G^0_3( H_{ \delta }, V ) 
      )
    }
  \right]
  \max\!\big(
    \| F \|_{ 
      G^0_1( H_{ \rho }, H_{ \alpha } )
    }
    ,
    \| B \|_{ 
      G^0_1( H_{ \rho }, 
      HS( U_0, H_{ \beta } ) )
    }^2
  \big)
\\ & \quad
  \cdot
  \max\!\left( 
    1, 
    \sup_{ s \in [0,T] }
    \| e^{ A s } \|_{ L(H) }^2 
  \right)
  \left(
    \|
      \varphi'
    \|_{
      G^0_2( 
        H_{ \rho }, L( H_{ \rho }, V) 
      )
    } 
    +
    \|
      \varphi''
    \|_{
      G^0_1( 
        H_{ \rho }, 
        L^{(2)}( H_{ \rho }, V) 
      )
    } 
  \right)
\\ & \quad 
  \cdot
  \left[
    \sup_{ s \in (0,T] }
    s^{ 
      \max( \rho - \alpha, 2 \rho - 2 \beta )
    }
  \max\!
  \big(
    \|
      e^{ A s }
    \|_{ 
      L( H_{ \alpha }, H_{ \rho } ) 
    }
    ,
    \|
      e^{ A s }
    \|_{ 
      L( H_{ \beta }, H_{ \rho } ) 
    }^2
  \big)
  \right]
\end{split}
\end{equation}
for all 
$
  t \in (0,T]
$,
$ 
  \delta \in [\gamma, \infty)
$,
$ 
  \varphi \in 
  G^2_{ 1 }( H_{ \rho }, V )
$
and all
$
  \rho \in 
  [
    \gamma,
    \alpha + 1
  )
  \cap
  [
    \gamma,
    \beta + \frac{ 1 }{ 2 }
  )
$.
Combining \eqref{eq:esttemporal}
and \eqref{eq:esttemporal2}
completes the proof of 
Corollary~\ref{cor:temporal}.
\end{proof}

\begin{cor}[Galerkin
approximations]
\label{cor:galerkin}
Assume that the setting
in Section~\ref{sec:setting}
is fulfilled,
assume
$ \alpha \leq \gamma $
and
$ \beta \leq \gamma $,
let
$
  \rho \in 
  [
    \gamma,
    \alpha + 1
  )
  \cap
  [
    \gamma,
    \beta + \frac{ 1 }{ 2 }
  )
$
and
$ 
  \delta \in [\gamma, \infty)
$
be real numbers
and let
$ P_N \in L(H_{ \rho }) $,
$ N \in \N $,
be a sequence of bounded
linear operators
with
$
  \sup_{ N \in \N } 
  \| P_N \|_{ L( H_{ \rho } ) }
  < \infty
$.
Then
\begin{equation}
  \sup_{
    \substack{
      \varphi \in 
      C^2_{ Lip }( H_{ \rho }, V )
      \backslash \{ 0 \}
    }
  }
  \sup_{ 
    N \in \mathbb{N}
  }
  \sup_{
    t \in (0,T]
  }
  \left(
  \frac{
    t^{ 
      \max( \rho - \delta + r, 0 ) 
    } \,
    \|
      P_{ t }( \varphi ) -
      P_{ t }( \varphi( P_N( \cdot ) ) ) 
    \|_{ 
      G^0_3( H_{ \delta }, V )
    }
  }{
    \| I - P_N \|_{
      L( H_{ \rho + r }, H_{ \rho } )
    }
    \,
    \| \varphi \|_{ 
      C^2_{ Lip }( H_{ \rho }, V )
    }
  } 
  \right)
  < \infty
\end{equation}
for all
$ 
  r \in 
  [0,1+ \alpha - \rho) \cap
  [0,1+2 \beta - 2 \rho) 
$.
In particular, if
$ 
  ( \lambda_n )_{ n \in \mathbb{N} } 
  \subset (0,\infty) 
$
is a non-decreasing 
sequence
of real numbers, 
if
$
  ( e_n )_{ n \in \mathbb{N} }
  \subset H
$
is an orthonormal basis
of $ H $ with
$
  D(A) = \{ 
    v \in H \colon
    \sum_{ n = 1 }^{ \infty }
    | \lambda_n |^2
    | \left< e_n, v \right>_H |^2
    < \infty
  \}
$
and 
$ 
  A v = 
  \sum_{ n = 1 }^{ \infty }
  - \lambda_n \left< e_n, v \right>_H
  e_n
$
for all $ v \in D(A) $,
if $ \rho = \gamma = 0 $
and $ p \geq 3 $
and if 
$
  P_N( v )
  =
  \sum_{ n = 1 }^N
  \left< e_n, v \right>_H
  e_n
$
for all $ v \in H $,
$ N \in \mathbb{N} $,
then
\begin{equation}
  \sup_{
    \substack{
      \varphi \in 
      C^2_{ Lip }( H, V )
      \backslash \{ 0 \}
    }
  }
  \sup_{ 
    N \in \mathbb{N}
  }
  \sup_{
    t \in (0,T]
  }
  \left(
  \frac{
    t^{ 
      \max( r - \delta, 0 ) 
    } \,
    \|
      P_{ t }( \varphi ) -
      P_{ t }( \varphi( P_N( \cdot ) ) ) 
    \|_{ 
      G^0_3( H_{ \delta }, V )
    }
  }{
    ( \lambda_N )^{ - r } 
    \,
    \| \varphi \|_{ 
      C^2_{ Lip }( H, V )
    }
  } 
  \right)
  < \infty
\end{equation}
and
\begin{equation}
  \sup_{
    \substack{
      \varphi \in 
      C^2_{ Lip }( H, V )
      \backslash \{ 0 \}
    }
  }
  \sup_{ 
    N \in \mathbb{N}
  }
  \sup_{
    t \in (0,T]
  }
  \left(
  \frac{
    t^r \,
    \big\|
    \mathbb{E}\big[ 
      \varphi( X_t )
    \big]
    -
    \mathbb{E}\big[ 
      \varphi( P_N( X_t ) )
    \big]
    \big\|_V
  }{
    ( \lambda_N )^{-r} \,
    \| \varphi \|_{ 
      C^2_{ Lip }( H, V )
    }
  } 
  \right)
  < \infty
\end{equation}
for all
$ 
  r \in 
  [0,1+ \alpha) \cap
  [0,1+2 \beta) 
$.
\end{cor}

Corollary~\ref{cor:galerkin}
follows directly from 
Corollary~\ref{cor:weakest}
and its proof 
is therefore
omitted.
Corollary~\ref{cor:galerkin}
is a certain spatial weak numerical
approximation result for SPDEs.
Further weak numerical
approximation results
for SPDEs
can be found in 
\cite{h03b,dd06,dp09,gkl09,h10c,ls10,d11,kll11,KovacsLarssonLindgren2012,Brehier2012,Kruse2012}.

\subsubsection{Stochastic Taylor
expansions for solutions of SPDEs}
\label{sec:tay}

A further application of the mild
It\^{o} formula~\eqref{eq:itoformel}
is the derivation of stochastic
Taylor expansions for solutions
of stochastic partial differential 
equations.
In Kloeden \citationand\ 
Platen~\cite{kp92}
stochastic Taylor expansions
are derived for solutions of
finite dimensional
stochastic ordinary differential
equations by an iterated application
of the standard It\^{o} formula.
Clearly, this strategy can not be 
accomplished in the infinite dimensional
SPDE setting since the standard It\^{o}
formula can, in general, 
not be applied to
the solution process of an SPDE. 
However, by using the 
mild It\^{o} 
formula~\eqref{eq:itoformel}
instead of the standard It\^{o}
formula, this approach can
be generalized to solutions
of SPDEs in a 
straightforward way. 
The main difference 
to the finite dimensional setting in 
Kloeden \citationand\ 
Platen~\cite{kp92}
is that the linear
operators 
$ L^{ (0) }_t $, $ t \in (0,T] $,
and
$ L^{ (1) }_t $, $ t \in (0,T] $,
in \eqref{eq:defL0}
and \eqref{eq:defL1}
here
depend explicitly 
on the time variable
$ t \in (0,T] $ too 
(compare \eqref{eq:defL0}
and \eqref{eq:defL1} here
with (1.13) and (1.14)
in Chapter~5 in \cite{kp92};
see also 
Theorem~\ref{thm:strongtay}
and Theorem~\ref{thm:weaktay}
below for more details).
Similar and related 
stochastic Taylor expansions for SPDEs
can be found 
in Buckdahn \citationand\ 
Ma~\cite{bm02}, 
Bayer \citationand\ 
Teichmann~\cite{bt08},
Conus~\cite{c08}, 
Jentzen \citationand\  
Kloeden~\cite{jk09c},
Buckdahn, 
Bulla \citationand\ 
Ma~\cite{bbm10}
and
Jentzen~\cite{j10}.

For formulating the
stochastic
Taylor expansions below
some notations 
are introduced 
(see also Chapter~5
in \cite{kp92}).
By
$ 
  \mathcal{M} 
  := \{ \emptyset \}
  \cup 
  \left(
  \cup_{ n = 1 }^{ \infty }
  \{ 0, 1 \}^n
  \right)
$ 
the set of
multi-indices is denoted.
Moreover, define
two functions
$ 
  \left| \cdot \right| 
  \colon \mathcal{M}
  \rightarrow \{ 0, 1, 2, \dots \} 
$
and
$ 
  -( \cdot ) \colon 
  \mathcal{M} \backslash
  \{ \emptyset \} 
  \rightarrow \mathcal{M} 
$
by 
$ 
  | \emptyset | := 0 
$,
by
$ 
  | (\alpha_1, \alpha_2,\dots, \alpha_n) |
  := n 
$
and by
$ 
  - (\alpha_1, \alpha_2, \dots, \alpha_n) 
  := 
  (\alpha_2, \alpha_3, \dots, \alpha_n) 
$
for all
$ 
  \alpha_1, \alpha_2, \dots, \alpha_n \in 
  \{ 0, 1 \} 
$
and all
$ 
  n \in \mathbb{N} 
$.
Thus note that
$ 
  | \alpha | \geq 1
$
and 
$ 
  \alpha_1, \dots ,
  \alpha_{ | \alpha | } 
  \in \{ 0, 1 \} 
$
for all 
$ 
  \alpha \in 
  \mathcal{A} 
  \backslash
  \{ \emptyset \} 
$.
Furthermore, a finite 
nonempty 
subset $ \mathcal{A}
\subset \mathcal{M} $
of $ \mathcal{M} $
is called {\it hierarchical set}
if $ - \alpha \in \mathcal{A} $
for all $ \alpha \in \mathcal{A}
\backslash \{ \emptyset \} $.
Next define a function 
$ 
  \mathbb{B} \colon 
  \mathcal{P}( \mathcal{M} )
  \rightarrow
  \mathcal{P}( \mathcal{M} )
$
by
$ 
  \mathbb{B}(\mathcal{A})
  := 
  \{ 
    \alpha \in \mathcal{M} 
    \backslash \mathcal{A} 
    \colon 
    -\alpha \in \mathcal{A} 
  \} 
$ for all
$ \mathcal{A} \subset \mathcal{M} $.
Finally, 
let 
$ 
  W^0 \colon [0,T] 
  \rightarrow \mathbb{R} 
$
be a function 
and let
$ 
  ( W^1_t )_{ t \in [0,T] } 
$
be a cylindrical $ Q $-Wiener process
defined by
$
  W^0(t) := t
$
and
$
  W^1_t
  := 
  W_t
$
for all $ t \in [0,T] $.
Using this notation
the mild It\^{o} 
formula~\eqref{eq:operatorIto}
can be written as
\begin{equation}
\label{eq:mildito3}
  \varphi( X_t )
  =
  \varphi( e^{ A (t - t_0) } X_{ t_0 } )
+
  \sum_{ i = 0 }^{ 1 }
  \int_{ t_0 }^t
  \!
  \big( 
    L_{ ( t - s ) }^{(i)} \varphi
  \big) ( X_s )
  \, dW^i_s
\end{equation}
$ \mathbb{P} $-a.s.\ for 
all $ t_0, t \in [0,T] $
with $ t_0 \leq t $
and all 
$ 
  \varphi \in 
  C^2( H_{ \gamma }, V ) 
$.
Moreover, for two normed
$ \mathbb{R} $-vector spaces
and 
$ n \in \{ 0, 1, 2, \dots \} $
we define 
$
  C_b^n( V_1, V_2 )
:=
  \left\{
    \varphi \in C^n( V_1, V_2 )
    \colon
      \| \varphi 
      \|_{ 
        L^{ \infty }(
          V_1, V_2 
        ) 
      }
      +
      \sum_{ k = 1 }^n
      \| \varphi^{ (k) } 
      \|_{ 
        L^{ \infty }(
          V_1, L^{ (k) }(V_1, V_2) 
        ) 
      }
      < \infty
  \right\}
$,
$
  C_b( V_1, V_2 )
  :=
  C_b^0( V_1, V_2 )
$
and
$
  C^{ \infty }_b( V_1, V_2 )
  :=
  \cap_{ k \in \mathbb{N} }
  C^k_b( V_1, V_2 )
$.
We are now ready to present the
stochastic Taylor expansions
based on the mild It\^{o} 
formula~\eqref{eq:mildito3}.

\begin{theorem}[Strong 
stochastic Taylor
expansions]
\label{thm:strongtay}
Assume that the setting
in Section~\ref{sec:setting}
is fulfilled, 
assume 
$ 
  F \in 
  C^{ \infty }_b( 
    H_{ \gamma },
    H_{ \alpha } 
  ) 
$,
assume 
$ 
  B \in 
  C^{ \infty }_b( 
    H_{ \gamma },
    HS( U_0, H_{ \beta } ) 
  ) 
$
and let
$ 
  \varphi 
  \in C^{ \infty }_b( H_{ \gamma }, V )
$.
Then
\begin{align}
\label{eq:strongtay}
&
  \varphi( X_{ t } ) 
= 
  \varphi\big( e^{ A( t - t_0) } X_{ t_0 } \big) 
\\&
+
  \sum_{ 
    \substack{
      \alpha 
      \in \mathcal{A}
      \\
      \alpha \neq \emptyset
    } 
  }
  \int_{ t_0 }^{ t }
  \int_{ t_0 }^{ s_{ |\alpha| } }
  \dots
  \int_{ t_0 }^{ s_2 }
  \Big( 
    L_{ 
      (
        s_2 - s_1 
      ) 
    }^{ 
      ( \alpha_1 )
    }
    \dots
    L_{ 
      (
        s_{ |\alpha| } - s_{ |\alpha| - 1}
      ) 
    }^{ 
      ( \alpha_{ |\alpha| - 1 } )
    } \,
    L_{ 
      (
        t - s_{ |\alpha| }
      ) 
    }^{ 
      ( \alpha_{ |\alpha| } ) 
    } \,
    \varphi 
  \Big)\big( 
    e^{ 
      A ( s_1 - t_0 ) 
    } 
    X_{ t_0 } 
  \big) \,
  dW^{ \alpha_1 }_{ s_1 } \,
  dW^{ \alpha_2 }_{ s_2 }
  \dots
  dW^{ \alpha_{ |\alpha| } }_{
    s_{ |\alpha| }
  }
\nonumber
\\&+
  \sum_{ 
    \alpha 
    \in \mathbb{B}(\mathcal{A})
  }
  \int_{ t_0 }^{ t }
  \int_{ t_0 }^{ s_{ |\alpha| } }
  \dots
  \int_{ t_0 }^{ s_2 }
  \Big( 
    L_{ 
      (
        s_2 - s_1 
      ) 
    }^{ 
      ( \alpha_1 )
    }
    \dots
    L_{ 
      (
        s_{ |\alpha| } - s_{ |\alpha| - 1}
      ) 
    }^{ 
      ( \alpha_{ |\alpha| - 1 } )
    } \,
    L_{ 
      (
        t - s_{ |\alpha| }
      ) 
    }^{ 
      ( \alpha_{ |\alpha| } ) 
    } \,
    \varphi 
  \Big)\big( 
    X_{ s_1 } 
  \big) \,
  dW^{ \alpha_1 }_{ s_1 } \,
  dW^{ \alpha_2 }_{ s_2 }
  \dots
  dW^{ \alpha_{ |\alpha| } }_{
    s_{ |\alpha| }
  }
\nonumber
\end{align}
for all $ t_0, t \in [0,T] $
with $ t_0 \leq t $ and
all hierarchical sets
$ \mathcal{A} \subset \mathcal{M} $.
\end{theorem}

\begin{proof}[Proof
of Theorem~\ref{thm:strongtay}]
Theorem~\ref{thm:strongtay}
immediately follows from an
iterated application of the
mild It\^{o} formula~\eqref{eq:mildito3}.
\end{proof}

The term 
$ 
  \varphi\big( 
    e^{ A (t-t_0) } X_{ t_0 } 
  \big) 
  +
  \sum_{
  \alpha \in \mathcal{A},
  \alpha \neq \emptyset  
  }
  \dots
$,
$ t \in [t_0,T] $,
on the left hand side of
\eqref{eq:strongtay}
is referred as {\it strong stochastic 
Taylor approximation} 
(or truncated strong stochastic 
Taylor expansion) of
$ \varphi( X_t ) $,
$ t \in [t_0,T] $,
corresponding to the
hierarchical set 
$ \mathcal{A}
\subset \mathcal{M} $
for $ t_0 \in [0,T] $.
The expression 
$ 
  \sum_{
  \alpha \in \mathbb{B}(\mathcal{A})
  }
  \dots
$,
$ t \in [t_0,T] $,
on the left hand side of
\eqref{eq:strongtay}
is called {\it remainder term}
of the strong stochastic 
Taylor expansions of
$ \varphi( X_t ) $,
$ t \in [t_0,T] $,
corresponding to the
hierarchical set 
$ \mathcal{A}
\subset \mathcal{M} $
for $ t_0 \in [0,T] $.
Next observe that,
in the case 
$ H = \mathbb{R}^d $
with $ d \in \mathbb{N} $
and $ A = 0 $,
Theorem~\ref{thm:strongtay}
essentially reduces to
Theorem~5.5.1
in 
Kloeden \citationand\ 
Platen~\cite{kp92}.
Let us also add the following remark
on possible generalizations
of Theorem~\ref{thm:strongtay}.

\begin{remark}
The assumption 
in Theorem~\ref{thm:strongtay}
that $ F $, $ B $
and $ \varphi $ are infinitely
often Fr\'{e}chet differentiable
can be relaxed.
To be more precise, 
to obtain \eqref{eq:strongtay}
for a given hierarchical
set 
$ 
  \mathcal{A} \subset \mathcal{M} 
$,
it is sufficient to assume 
that
$ 
  F \in C_b( H_{ \gamma }, H_{ \alpha } )
$ 
is
$ 
  \max_{ 
    \alpha \in \mathbb{B}( \mathcal{A} ),
    \alpha_1 = 0
  } 
  \min\big\{
    2 k - 2 - 
    \sum_{ i = 1 }^{ k - 1 }
    \alpha_i
    \colon
    k \in \{ 1, \dots, |\alpha|\},
    \alpha_{ k + 1 } = 
    \ldots =
    \alpha_{ | \alpha | } =
    1
  \big\}
$--times, 
that
$ 
  B \in 
  C_b( 
    H_{ \gamma },
    HS( U_0, H_{ \beta } ) 
  )
$ 
is
$ 
  \max_{ 
    \alpha \in \mathbb{B}( \mathcal{A} )
  }
  (
  2 | \alpha | - 2
  - 
  \sum_{ i = 1 }^{ | \alpha | - 1 }
  \alpha_i
  )
$--times
and that
$ 
  \varphi \in
  C_b( 
    H_{ \gamma },
    V
  )
$
is
$
    \max_{ 
      \alpha \in \mathbb{B}( \mathcal{A} )
    }
    (
    2 | \alpha | 
    - 
    \sum_{ i = 1 }^{ | \alpha | }
    \alpha_i
    )
$--times
continuously
Fr\'{e}chet 
differentiable
with globally
bounded 
Fr\'{e}chet derivatives.
Moreover,
the boundedness
assumptions on $ F, B $
and $ \varphi $ and its derivatives
can be reduced
if $ p \in [2,\infty) $ 
in Assumption~\ref{initial}
is assumed
to be sufficiently large.
\end{remark}

In the next step 
Theorem~\ref{thm:strongtay}
is illustrated with two possible
examples.
First, in the case of the
hierarchical set 
$ \mathcal{A} = \{ \emptyset \} $,
equation~\eqref{eq:strongtay}
reduces to
\eqref{eq:mildito3}, i.e., we have
\begin{equation}
  \varphi( X_t )
  =
  \underbrace{
  \varphi( e^{ A (t - t_0) } X_{ t_0 } )
  }_{
    \substack{
    \text{strong stochastic Taylor}
    \\
    \text{approximation corresponding}
    \\
    \text{to the hierarchical set }
    \mathcal{A} = \{ \emptyset \}
    }
  }
+
  \underbrace{
  \sum_{ i = 0 }^{ 1 }
  \int_{ t_0 }^t
  \!
  \big( 
    L_{ ( t - s ) }^{(i)} \varphi
  \big) ( X_s )
  \, dW^i_s
  }_{
    \substack{
    \text{remainder term corresponding}
    \\
    \text{to the hierarchical set }
    \mathcal{A} = \{ \emptyset \}
    }
  }
\end{equation}
$ \mathbb{P} $-a.s.\ for 
all $ t_0, t \in [0,T] $
with $ t_0 \leq t $
and all 
$ 
  \varphi \in 
  C^{ \infty }_b( H_{ \gamma }, V ) 
$. 
Second, in the case of the
hierarchical set
$ 
  \mathcal{A} 
  = 
  \{ 
    \emptyset,
    (1) 
  \} 
$,
equation~\eqref{eq:strongtay}
simplifies to
\begin{equation}
\begin{split}
  \varphi( X_t )
& =
  \underbrace{
  \varphi( e^{ A (t - t_0) } X_{ t_0 } )
  +
  \int_{ t_0 }^t
  \varphi'( 
    e^{ A (t - t_0) } X_{ t_0 } 
  )
  \, e^{ A (t - s) }  
  B(
    e^{ A (s - t_0) } X_{ t_0 } 
  )
  \, dW_s
  }_{
    \text{strong stochastic Taylor
    approximation corresponding
    to }
    \mathcal{A} = \{ \emptyset, (1) \}
  }
\\ & \quad
  \underbrace{
  \begin{split}
  &
  +
  \int_{ t_0 }^t
  \big( 
    L_{ ( t - s ) }^{(0)} \varphi
  \big) ( X_s )
  \,
  ds
  +
  \int_{ t_0 }^t
  \int_{ t_0 }^s
  \big( 
    L_{ ( s - u ) }^{(0)}
    L_{ ( t - s ) }^{(1)} \varphi
  \big) ( X_u )
  \,
  du
  \,
  dW_s
  \\&
  +
  \int_{ t_0 }^t
  \int_{ t_0 }^s
  \big( 
    L_{ ( s - u ) }^{(1)}
    L_{ ( t - s ) }^{(1)} \varphi
  \big) ( X_u )
  \,
  dW_u
  \,
  dW_s
  \end{split}
  }_{
    \text{remainder term corresponding
    to }
    \mathcal{A} = \{ \emptyset, (1) \}
  }
\end{split}
\end{equation}
$ \mathbb{P} $-a.s.\ for 
all $ t_0, t \in [0,T] $
with $ t_0 \leq t $
and all 
$ 
  \varphi \in 
  C^{ \infty }_b( H_{ \gamma }, V ) 
$. 
After having presented strong 
stochastic Taylor expansions
%based on the mild It\^{o} 
%formula~\eqref{eq:mildito3}
in Theorem~\ref{thm:strongtay},
we now formula the corresponding
weak stochastic Taylor expansions
based on the mild It\^{o} 
formula~\eqref{eq:mildito3}.

\begin{theorem}[Weak 
stochastic Taylor
expansions]
\label{thm:weaktay}
Assume that the setting
in Section~\ref{sec:setting}
is fulfilled, 
let 
$ n \in \mathbb{N} $,
assume 
$ 
  F \in 
  C_b^{ (2n - 2) }( H_{ \gamma },
  H_{ \alpha } ) 
$,
assume 
$ 
  B \in 
  C_b^{ (2n - 2) }( H_{ \gamma },
  HS( U_0, H_{ \beta } ) ) 
$
and let
$ 
  \varphi \in 
  C_b^{ 2n }( H_{ \gamma }, V )
$.
Then
\begin{align}
\label{eq:weaktay}
&
\nonumber
  \mathbb{E}\big[ \varphi( X_{ t } ) \big]
= 
  \mathbb{E}\big[ 
    \varphi\big( e^{ A (t - t_0) } X_{ t_0 } 
    \big) 
  \big]
\\ &
\nonumber
  +
  \sum_{ 
    k = 1
  }^{ n - 1 }
  \int_{ t_0 }^{ t }
  \int_{ t_0 }^{ s_{ k } }
  \dots
  \int_{ t_0 }^{ s_2 }
  \mathbb{E}\!\left[
  \big( 
    L_{ 
      (
        s_2 - s_1 
      ) 
    }^{ 
      ( 0 )
    }
    \dots
    L_{ 
      (
        s_{ k } - s_{ k - 1}
      ) 
    }^{ 
      ( 0 )
    } \,
    L_{ 
      (
        t - s_{ k }
      ) 
    }^{ 
      ( 0 ) 
    } \,
    \varphi 
  \big)\!\left( 
    e^{ 
      A ( s_1 - t_0 ) 
    } 
    X_{ t_0 } 
  \right)
  \right] 
  ds_1 \,
  ds_2
  \dots
  ds_{ k }
\\ & +
  \int_{ t_0 }^{ t }
  \int_{ t_0 }^{ s_{ n } }
  \dots
  \int_{ t_0 }^{ s_2 }
  \mathbb{E}\!\left[
  \big( 
    L_{ 
      (
        s_2 - s_1 
      ) 
    }^{ 
      ( 0 )
    }
    \dots
    L_{ 
      (
        s_{ n } - s_{ n - 1}
      ) 
    }^{ 
      ( 0 )
    } \,
    L_{ 
      (
        t - s_{ n }
      ) 
    }^{ 
      ( 0 ) 
    } \,
    \varphi 
  \big)\!\left( 
    X_{ s_1 } 
  \right)
  \right] 
  ds_1 \,
  ds_2
  \dots
  ds_{ n }
\end{align}
for all $ t_0, t \in [0,T] $
with $ t_0 \leq t $.
\end{theorem}

\begin{proof}[Proof of Theorem~\ref{thm:weaktay}]
Equation~\eqref{eq:weaktay}
immediately follows by
taking expectations on both
sides of 
equation~\eqref{eq:strongtay}
with the hierarchical set
$ 
  \mathcal{A} = 
  \{ 
    \alpha \in \mathcal{M}
    \colon
    | \alpha | \leq n, 
    \sum_{ i = 1 }^{ |\alpha| }    
    \alpha_i = 0
  \}
$.
\end{proof}

Using definition~\eqref{eq:defPt},
the weak stochastic Taylor expansions 
in Theorem~\ref{thm:weaktay}
can also be written in the following form.

\begin{cor}
\label{cor:weaktay}
Assume that the setting
in Section~\ref{sec:setting}
is fulfilled, 
let 
$ n \in \mathbb{N} $,
assume 
$ 
  F \in 
  C_b^{ (2n - 2) }( H_{ \gamma },
  H_{ \alpha } ) 
$,
assume 
$ 
  B \in 
  C_b^{ (2n - 2) }( H_{ \gamma },
  HS( U_0, H_{ \beta } ) ) 
$
and let
$ 
  \varphi \in 
  C_b^{ 2n }( H_{ \gamma }, V )
$.
Then
\begin{equation}
\label{eq:weaktay2}
\begin{split}
  \big( P_t \varphi \big)( x )
&= 
  \big( K_t \varphi 
  \big)(x)
\\ & \quad
  +
  \sum_{ 
    k = 1
  }^{ n - 1 }
  \int_{ 0 }^{ t }
  \int_{ 0 }^{ s_{ k } }
  \dots
  \int_{ 0 }^{ s_2 }
  \big( 
    K_{ s_1 } \,
    L_{ 
      (
        s_2 - s_1 
      ) 
    }^{ 
      ( 0 )
    }
    \dots
    L_{ 
      (
        s_{ k } - s_{ k - 1}
      ) 
    }^{ 
      ( 0 )
    } \,
    L_{ 
      (
        t - s_{ k }
      ) 
    }^{ 
      ( 0 ) 
    } \,
    \varphi 
  \big)\!\left( 
    x 
  \right)
  ds_1 \,
  ds_2
  \dots
  ds_{ k }
\\ & \quad +
  \int_{ 0 }^{ t }
  \int_{ 0 }^{ s_{ n } }
  \dots
  \int_{ 0 }^{ s_2 }
  \big( 
    P_{ s_1 } \,
    L_{ 
      (
        s_2 - s_1 
      ) 
    }^{ 
      ( 0 )
    }
    \dots
    L_{ 
      (
        s_{ n } - s_{ n - 1}
      ) 
    }^{ 
      ( 0 )
    } \,
    L_{ 
      (
        t - s_{ n }
      ) 
    }^{ 
      ( 0 ) 
    } \,
    \varphi 
  \big)( 
    x 
  ) \,
  ds_1 \,
  ds_2
  \dots
  ds_{ n }
\end{split}
\end{equation}
for all 
$ x \in H_{ \gamma } $
and all
$ t \in [0,\infty) $.
\end{cor}

\subsubsection{Further mild
It\^{o} formulas for 
solutions of SPDEs}

This subsection presents
two slightly different
variants 
(Corollary~\ref{cor:anotherito1}
and Proposition~\ref{prop:secondito})
of the mild It\^{o} formula
in Corollary~\ref{cor:ito}.
Both variants assume that
the test function $ \varphi $ in
Corollary~\ref{cor:ito}
fulfills additional regularity.
The first variant 
(see Corollary~\ref{cor:anotherito1}
below)
is a direct consequence of
Corollary~\ref{cor:ito}.

\begin{cor}[Another - 
somehow mild - 
It{\^o} type formula
for solutions of SPDEs]
\label{cor:anotherito1}
Assume that
the setting in Section~\ref{sec:setting}
is fulfilled.
Then
\begin{equation}
\label{eq:welldefinedOTHER1}
  \mathbb{P}\!\left[
    \int_{ t_0 }^t
    \big\|
      \varphi'\big( 
        (
          I +
          e^{ A (t - t_0) } -
          e^{ A (t - s) } 
        )
        X_{ t_0 }  
    \big)  
  \big\|_{
    L( H_r, V )
  }
    \,
    \big\|
      A \, e^{ A( t - s ) } \,
      X_{ t_0 } 
    \big\|_{
      H_r
    }
    \,
    ds
    < \infty 
  \right] = 1 
\end{equation}
and
\begin{equation}
\label{eq:OTHERITO1}
\begin{split}
  \varphi( X_t )
 &=
  \varphi( X_{ t_0 } ) 
  +
  \int_{ t_0 }^t
  \varphi'\big( 
    (
      I +
      e^{ A (t - t_0) } -
      e^{ A (t - s) } 
    )
    X_{ t_0 }  
  \big) 
  A \, e^{ A( t - s ) } \,
  X_{ t_0 } \, ds
\\ & \quad
  +
  \int_{ t_0 }^t
  \varphi'( e^{ A( t - s ) } X_s ) \,
  e^{ A( t - s ) }
  F( X_s ) \, ds
%\\ & \quad
  +
  \int_{ t_0 }^t
  \varphi'( e^{ A( t - s ) } X_s ) \,
  e^{ A( t - s ) }
  B( X_s ) \, dW_s
\\&\quad+
  \frac{1}{2}
  \sum_{ j \in \mathcal{J} }
  \int_{ t_0 }^t
  \varphi''( e^{ A( t - s ) } X_s )
  \big(
    e^{ A( t - s ) }
    B( X_s ) g_j,
    e^{ A( t - s ) }
    B( X_s ) g_j
  \big) \, ds
\end{split}
\end{equation}
$ \mathbb{P} $-a.s.\ for 
all
$ t_0, t \in [ 0, T ] $
with $ t_0 \leq t $,
$ 
  \varphi \in 
  C^2( H_{ r }, V ) 
$
and all
$
  r \in (-\infty, \gamma ) 
$.
\end{cor}

\begin{proof}[Proof
of Corollary~\ref{cor:anotherito1}]
Let $ r \in (-\infty,\gamma) $
be a fixed real number
and define
a family 
$
  \bar{X}^{ t_0, t }
  \colon
  [t_0,t] \times \Omega
  \rightarrow H_{ r }
$,
$
  (t_0, t) \in \angle
$,
of adapted stochastic 
processes with continuous 
sample paths
by
\begin{equation}
\begin{split}
  \bar{X}^{ t_0, t }_u
& :=
  X_{ t_0 } +
  \int_{ t_0 }^u 
    A \, e^{ A ( t - s ) } X_{ t_0 } \, 
  ds
=
  X_{ t_0 } +
  e^{ A (t - u) }
  \left(
    e^{ A (u - t_0) } - I
  \right)
  X_{ t_0 }
\\ & =
  \left(
    I +
    e^{ A (t - t_0) } -
    e^{ A (t - u) } 
  \right)
  X_{ t_0 }
\end{split}
\end{equation}
for all $ u \in [t_0,t] $
and all $ t_0, t \in [0,T] $
with $ t_0 \leq t $.
The fundamental theorem of
calculus then implies
\begin{equation}
\label{eq:hauptsatz}
\begin{split}
  \varphi( e^{ A( t - t_0 ) } X_{ t_0 } )
& =
  \varphi\!\left( 
    X_{ t_0 } +
    \int_{ t_0 }^t 
      A \, e^{ A ( t - s ) } X_{ t_0 } \, 
    ds
  \right)
=
  \varphi( \bar{X}^{t_0,t}_t )
\\ & =
  \varphi( X_{ t_0 } )
  +
  \int_{ t_0 }^{ t }
  \varphi'( \bar{X}^{ t_0, t }_s ) \, A \,
  e^{ A (t - s) } \, X_{ t_0 }
  \, ds
\end{split}
\end{equation}
for 
all $ t_0, t \in [ 0, T ] $
with $ t_0 \leq t $
and all
$ 
  \varphi \in 
  C^2( H_{ r }, V ) 
$.
Combining \eqref{eq:hauptsatz}
and Corollary~\ref{cor:anotherito1}
then  completes the proof
of Corollary~\ref{cor:ito}.
\end{proof}

Observe
that
equations~\eqref{eq:well1d}--\eqref{eq:well3d}
and
equation~\eqref{eq:welldefinedOTHER1}
ensure that all deterministic
and stochastic integrals in
\eqref{eq:OTHERITO1}
are well defined.

\begin{prop}[A further - 
somehow mild - It{\^o} type 
formula for solutions of SPDEs]
\label{prop:secondito}
Assume that
the setting in Section~\ref{sec:setting}
is fulfilled. Then
\begin{equation}
\label{eq:thirdito_well01}
  \mathbb{P}\!\left[
  \int_{ t_0 }^t
  \big\|
  \varphi'\big( 
    X_{ t_0 }
    +
    e^{ A (t - s) }
    ( 
      X_s - X_{ t_0 }
    )  
  \big) 
  \big\|_{ L( H_r, V ) }
  \,
  \big\|
    A \, e^{ A (t - s) } X_{ t_0 }
  \big\|_{ H_r }
  \,
  ds < \infty
  \right] = 1
  ,
\end{equation}
\begin{equation}
  \mathbb{P}\!\left[
  \int_{ t_0 }^t
  \big\|
  \varphi'\big( 
    X_{ t_0 }
    +
    e^{ A (t - s) }
    ( 
      X_s - X_{ t_0 }
    )  
  \big) 
  \big\|_{ L( H_r, V ) }
  \,
  \big\|
    e^{ A( t - s ) }
    F( X_s ) 
  \big\|_{ H_r }
  \,
  ds < \infty
  \right] = 1
  ,
\end{equation}
\begin{equation}
  \mathbb{P}\!\left[
    \int_{ t_0 }^t
    \big\|
      \varphi'\big( 
        X_{ t_0 }
        +
        e^{ A (t - s) }
        ( 
          X_s - X_{ t_0 }
        )  
      \big) \,
      e^{ A(t-s) } B( X_s )
    \big\|_{ HS(U_0, V ) }^2 
    \,
    ds
    < \infty 
  \right] = 1,
\end{equation}
\begin{equation}
\label{eq:thirdito_well04}
  \mathbb{P}\!\left[
    \int_{ t_0 }^t
    \big\|
      \varphi''\big( 
        X_{ t_0 }
        +
        e^{ A (t - s) }
        ( 
          X_s - X_{ t_0 }
        )  
      \big) 
    \big\|_{ L^{(2)}( H_{ r }, V ) }
    \,
    \big\|
      e^{ A (t - s) } B( X_s )
    \big\|_{ HS(U_0, H_{ r }) }^2 
    \,
    ds
    < \infty 
  \right] = 1
\end{equation}
and
\begin{equation}
\label{eq:anotherito}
\begin{split}
  \varphi( X_t )
& =
  \varphi( X_{ t_0 } )
  +
  \int_{ t_0 }^t
  \varphi'\big( 
    X_{ t_0 }
    +
    e^{ A (t - s) }
    ( 
      X_s - X_{ t_0 }
    )  
  \big) 
  \big[
    A \, e^{ A (t - s) } X_{ t_0 }
    +
    e^{ A( t - s ) }
    F( X_s ) 
  \big] \, ds
\\ & \quad +
  \int_{ t_0 }^t
  \varphi'\big( 
    X_{ t_0 }
    +
    e^{ A (t - s) }
    ( 
      X_s - X_{ t_0 }
    )  
  \big) \,
  e^{ A( t - s ) }
  B( X_s ) \, dW_s
\\ & \quad +
  \frac{1}{2}
  \sum_{ j \in \mathcal{J} }
  \int_{ t_0 }^t
  \varphi''\big( 
    X_{ t_0 }
    +
    e^{ A (t - s) }
    ( 
      X_s - X_{ t_0 }
    )  
  \big) 
  \big(
    e^{ A( t - s ) }
    B( X_s ) g_j,
    e^{ A( t - s ) }
    B( X_s ) g_j
  \big) \, ds
\end{split}
\end{equation}
$ \mathbb{P} $-a.s.\ for 
all
$ t_0, t \in [ 0, T ] $
with $ t_0 \leq t $,
$ 
  \varphi \in 
  C^2( H_{ r }, V ) 
$
and all
$
  r \in (-\infty, \gamma ) 
$.
\end{prop}

\begin{proof}[Proof
of Proposition~\ref{prop:secondito}]
First, observe that
the well known identity
$
  e^{ A t } v =
  v +
  \int_0^t A \, e^{ A s } v \, ds
$
for all $ v \in H_{ \gamma } $ 
and all $ t \in [0,\infty) $
shows
\begin{equation}
  X_t
=
  X_{ t_0 }
+
  \int_{ t_0 }^t 
  \left[
    A \,
    e^{ A (t - s) }
    X_{ t_0 }
    +
    e^{ A (t - s) } F( X_s ) 
  \right] ds
+
  \int_{ t_0 }^t
  e^{ A (t - s) } B( X_s ) \, dW_s
\end{equation}
$ \mathbb{P} $-a.s.\ for all
$ t_0, t \in [0,T] $
with $ t_0 \leq t $.
In the next step 
let $ r \in (-\infty, \gamma) $
be a fixed real number and
let
$ 
  \bar{X}^{ t_0, t } 
  \colon [t_0,t] \times
  \Omega \rightarrow H_{ r }
$,
$ 
  (t_0, t) \in \angle
$,
be a family
of adapted
stochastic processes with
continuous sample paths
given by
\begin{equation}
\begin{split}
  \bar{X}^{t_0, t}_u
& =
  X_{ t_0 }
+
  \int_{ t_0 }^u 
    \left[
      A \,
      e^{ A (t - s) }
      X_{ t_0 }
      +
      e^{ A (t - s) } F( X_s ) 
    \right]
  ds
+
  \int_{ t_0 }^u
  e^{ A (t - s) } B( X_s ) \, dW_s
\\ & =
  X_{ t_0 }
+
  e^{ A (t - t_0) }
  X_{ t_0 }
-
  e^{ A (t - u) }
  X_{ t_0 }
  +
  \int_{ t_0 }^u 
    e^{ A (t - s) } F( X_s ) \,
  ds
+
  \int_{ t_0 }^u
  e^{ A (t - s) } B( X_s ) \, dW_s
\end{split}
\end{equation}
$ \mathbb{P} $-a.s.\ for
all $ u \in [t_0,t] $
and all
$ t_0, t \in [0,T] $
with $ t_0 \leq t $
(see also \eqref{eq:semiX} above).
The standard
It{\^o} formula in infinite 
dimensions (see
Theorem~2.4 in Brze\'{z}niak,
Van Neerven, Veraar 
\citationand\ 
Weis~\cite{bvvw08}) 
then gives
\begin{equation}
\label{eq:itoformel2}
\begin{split}
  \varphi( \bar{X}_u^{ t_0, t } )
&=
  \varphi( \bar{X}_{ t_0 }^{ t_0, t } )
  +
  \int_{t_0}^u
  \varphi'( \bar{X}_{ s }^{ t_0, t } ) 
  \left[
    A \, e^{ A (t - s) } \, X_{ t_0 }
    +
    e^{ A (t - s) } F( X_s ) 
  \right] ds
\\&\quad
  +
  \int_{t_0}^u
  \varphi'( \bar{X}_{ s }^{ t_0, t } ) \,
  e^{ A (t - s) } B( X_s ) \, dW_s
\\ & \quad +
  \frac{1}{2}
  \sum_{ j \in \mathcal{J} }
  \int_{t_0}^u
  \varphi''( \bar{X}_{ s }^{ t_0, t } )
  \left(
    e^{ A (t - s) } B( X_s ) g_j,
    e^{ A (t - s) } B( X_s ) g_j
  \right) ds
\end{split}
\end{equation}
$ \mathbb{P} $-a.s.\ for
all $ u \in [t_0,t] $,
$ t_0, t \in [0,T] $
with $ t_0 \leq t $
and all
$ \varphi \in C^2( H_{ r }, V ) $.
This, in particular, shows
\begin{equation}
\label{eq:OTHERITOitoformel2}
\begin{split}
  \varphi( \bar{X}^{ t_0, t }_t )
=
  \varphi( X_t )
&=
  \varphi( X_{ t_0 } )
  +
  \int_{t_0}^t
  \varphi'( \bar{X}_{ s }^{ t_0, t } ) 
  \left[
    A \, e^{ A (t - s) } \, X_{ t_0 }
    +
    e^{ A (t - s) } F( X_s ) 
  \right] ds
\\&\quad
  +
  \int_{t_0}^t
  \varphi'( \bar{X}_{ s }^{ t_0, t } ) \,
  e^{ A (t - s) } B( X_s ) \, dW_s
\\ & \quad +
  \frac{1}{2}
  \sum_{ j \in \mathcal{J} }
  \int_{t_0}^t
  \varphi''( \bar{X}_{ s }^{ t_0, t } )
  \left(
    e^{ A (t - s) } B( X_s ) g_j,
    e^{ A (t - s) } B( X_s ) g_j
  \right) ds
\end{split}
\end{equation}
$ \mathbb{P} $-a.s.\ for
all $ t_0, t \in [0,T] $
with $ t_0 \leq t $
and all
$ \varphi \in C^2( H_{ r }, V ) $
(see also \eqref{eq:itoformel2inProof}
above).
Putting the identity
\begin{equation}
\begin{split}
  \bar{X}^{ t_0, t }_s
& =
  X_{ t_0 }
+
  \int_{ t_0 }^s 
    \left[
      A \,
      e^{ A (t - u) }
      X_{ t_0 }
      +
      e^{ A (t - u) } F( X_u ) 
    \right]
  ds
+
  \int_{ t_0 }^s
  e^{ A (t - u) } B( X_u ) \, dW_u
\\ & =
  X_{ t_0 }
+
  e^{ A (t - s) }
  \int_{ t_0 }^s 
    \left[
      A \,
      e^{ A (s - u) }
      X_{ t_0 }
      +
      e^{ A (s - u) } F( X_u ) 
    \right]
  ds
\\ & \quad
+
  e^{ A (t - s) }
  \int_{ t_0 }^s
  e^{ A (s - u) } B( X_u ) \, dW_u
=
  X_{ t_0 }
  +
  e^{ A (t - s) }
  \left( 
    X_s - X_{ t_0 }
  \right)
\end{split}
\end{equation}
$ \mathbb{P} $-a.s.\ for 
all $ s \in [t_0,T] $
and all
$ t_0, t \in [0,T] $
with $ t_0 \leq t $
(see also \eqref{eq:fact} above)
into \eqref{eq:OTHERITOitoformel2}
finally shows \eqref{eq:anotherito}.
The proof of
Proposition~\ref{prop:secondito}
is thus completed.
\end{proof}

Note
that
equations~\eqref{eq:thirdito_well01}--\eqref{eq:thirdito_well04}
imply that all deterministic
and stochastic integrals in
\eqref{eq:anotherito}
are well defined.
Finally, observe that
the It\^{o} type formulas 
in Corollary~\ref{cor:anotherito1}
and Proposition~\ref{prop:secondito}
can be generalized to the more general
case of mild It\^{o} processes
(or mild semimartingales, cf.\ Remark~\ref{rem:moregeneral})
if additional assumptions on the
semigroup are fulfilled.

\subsection{Numerical
approximations processes
for SPDEs}
\label{sec:numerics}

This subsection demonstrates
how different types of numerical
approximation processes for SPDEs
can be formulated 
as mild It\^{o} processes.
To this end the following notation is used. 
Let
$
  \lfloor \cdot \rfloor_N 
  \colon
  [0,T] 
  \rightarrow 
  [0,T]
$,
$ N \in \mathbb{N} $,
be a sequence of mappings
given by
\begin{equation}
  \lfloor t \rfloor_N 
:=
  \max\!\Big(
    s \in 
    \left\{ 
      0, \tfrac{T}{N}, 
      \tfrac{2 T}{N}, \dots, 
      \tfrac{ (N-1) T }{ N }, T
    \right\} 
    \colon
    s \leq t
  \Big)
\end{equation}
for all $ t \in [0,T] $
and all $ N \in \mathbb{N} $.
Both Euler (see 
Subsection~\ref{sec:euler})
and
Milstein (see
Subsection~\ref{sec:mil})
type approximations for SPDEs
are formulated as mild It\^{o}
processes. We begin
with Euler type approximations
for SPDEs in 
Subsection~\ref{sec:euler}.

\subsubsection{Euler type 
approximations for SPDEs}
\label{sec:euler}

It is illustrated here how 
exponential Euler
approximations, accelerated
exponential Euler approximations,
linear implicit Euler approximations
and linear implicit Crank-Nicolson
approximations can be formulated as
mild It\^{o} processes.

\paragraph{Exponential Euler
approximations for SPDEs}
%\label{sec:splitup}

Let 
$
  Y^N
  \colon [0,T]
  \times
  \Omega
  \rightarrow H_{ \gamma }
$,
$
  N \in \mathbb{N} 
$,
be a sequence of predictable
stochastic
processes given by
\begin{equation}
\label{eq:split0}
\begin{split}
  Y^N_t
& =
  e^{ A t } \, \xi 
  +
  \int_0^t
  e^{
    A 
    \left( t - 
      \lfloor s \rfloor_N
    \right)
  }
  \,
  F\big( 
    Y^N_{
      \lfloor s \rfloor_N 
    } 
  \big)
  \,
  ds
  +
  \int_0^t
  e^{
    A 
    \left( t - 
      \lfloor s \rfloor_N
    \right)
  }
  \,
  B\big( 
    Y^N_{
      \lfloor s \rfloor_N 
    } 
  \big)
  \,
  dW_s
\\ & =
  e^{ A t } \, \xi 
  +
  \int_0^t
  e^{ A (t - s) }
  \;
  e^{
    A 
    \left( s - 
      \lfloor s \rfloor_N
    \right)
  }
  \,
  F\big( 
    Y^N_{
      \lfloor s \rfloor_N 
    } 
  \big)
  \,
  ds
  +
  \int_0^t
  e^{ A (t - s) }
  \;
  e^{
    A 
    \left( s - 
      \lfloor s \rfloor_N
    \right)
  }
  \,
  B\big( 
    Y^N_{
      \lfloor s \rfloor_N 
    } 
  \big)
  \,
  dW_s
\end{split}
\end{equation}
$ \mathbb{P} $-a.s.\ for 
all $ t \in [0,T] $
and all $ N \in \mathbb{N} $.
Observe that
for each $ N \in \mathbb{N} $
the stochastic
process
$
  Y^N \colon
  [0,T] \times \Omega
  \rightarrow H_{ \gamma }
$
is a mild It\^{o} process
with semigroup 
$ 
  e^{ A (t_2 - t_1) } \in 
  L( H_{ \min( \alpha, \beta, \gamma ) },
    H_{ \gamma }
  ) 
$, $ (t_1,t_2) \in \angle $,
with mild drift
\begin{equation}
\label{eq:split_milddrift}
  e^{
    A 
    \left( t - 
      \lfloor t \rfloor_N
    \right)
  }
  \,
  F\big( 
    Y^N_{
      \lfloor t \rfloor_N 
    } 
  \big),
\qquad
  t \in [0,T],
\end{equation}
and with mild diffusion
\begin{equation}
\label{eq:split_milddiffusion}
  e^{
    A 
    \left( t - 
      \lfloor t \rfloor_N
    \right)
  }
  \,
  B\big( 
    Y^N_{
      \lfloor t \rfloor_N 
    } 
  \big),
\qquad
  t \in [0,T] .
\end{equation}
Proposition~\ref{propsimple}
hence shows
\begin{equation}
\label{eq:propcons}
  Y^N_{ t }
 =
  e^{
    A \left( t - t_0 \right)
  }
  \,
  Y^N_{ t_0 } 
  +
  \int_{ t_0 }^{
    t
  }
  e^{
    A \left( t - \lfloor s \rfloor_N \right)
  }
  \,
      F\big( 
        Y^N_{ \lfloor s \rfloor_N } 
      \big) \,
    ds
  +
  \int_{ t_0 }^{
    t
  }
  e^{
    A \left( t - \lfloor s \rfloor_N \right)
  }
  \,
      B\big( 
        Y^N_{ \lfloor s \rfloor_N } 
      \big) \,
    dW_s
\end{equation}
$ \mathbb{P} $-a.s.\ for 
all $ t_0, t \in [0,T] $
with $ t_0 \leq t $ and
from \eqref{eq:propcons}
we conclude
\begin{equation}
\label{eq:split}
\begin{split}
  Y^N_{ 
    \frac{ (n+1) T }{ N } 
  }
& =
  e^{
    A \frac{ T }{ N }
  }
  \,
  Y^N_{ \frac{ n T }{ N } } 
  +
    \int_{ \frac{ n T }{ N } }^{
      \frac{ (n+1) T }{ N }    
    }
  e^{
    A \frac{ T }{ N }
  }
  \,
      F\big( 
        Y^N_{ \frac{ n T }{ N } } 
      \big) \,
    ds
  +
    \int_{ \frac{ n T }{ N } }^{
      \frac{ (n+1) T }{ N }    
    }
  e^{
    A \frac{ T }{ N }
  }
  \,
      B\big( 
        Y^N_{ \frac{ n T }{ N } } 
      \big) \,
    dW_s
\\ & =
  e^{
    A \frac{ T }{ N }
  }
  \left(
    Y^N_{ \frac{ n T }{ N } } 
    +
    \frac{ T }{ N } \cdot
    F\big( 
      Y^N_{ \frac{ n T }{ N } } 
    \big)
    +
    \int_{ \frac{ n T }{ N } }^{
      \frac{ (n+1) T }{ N }    
    }
      B\big( 
        Y^N_{ \frac{ n T }{ N } } 
      \big) \,
    dW_s
  \right)
\end{split}
\end{equation}
$ \mathbb{P} $-a.s.\ for 
all 
$ n \in \left\{ 0, 1, \dots, N-1
\right\} $
and all
$ N \in \mathbb{N} $.
The mild It\^{o} processes
$ Y^N $, $ N \in \mathbb{N} $,
are thus nothing else but appropriate
time continuous interpolations
of exponential Euler approximations
(a.k.a.~splitting-up 
approximations or
exponential integrator 
approximations; 
see, e.g., 
\cite{gk03a,lr04,cv10,lt11}
and the references therein).
Note that the mild 
drift~\eqref{eq:split_milddrift}
and the mild
diffusion~\eqref{eq:split_milddiffusion}
of the exponential Euler 
approximations~\eqref{eq:split}
contain the 
{\it correction
term}
$
  e^{
    A 
    \left( t - 
      \lfloor t \rfloor_N
    \right)
  }
$, $ t \in [0,T] $,
when compared to 
the mild drift~\eqref{eq:milddrift}
and the mild diffusion~\eqref{eq:milddiffusion}
of the exact solution of the
SPDE~\eqref{eq:SPDE}.

\paragraph{Accelerated 
exponential Euler
approximations for SPDEs}
This paragraph demonstrates
that accelerated
exponential Euler 
approximations 
can be written as mild It\^{o} 
processes.
Let 
$
  \tilde{Y}^N
  \colon [0,T]
  \times
  \Omega
  \rightarrow H_{ \gamma }
$,
$
  N \in \mathbb{N} 
$,
be a sequence of predictable
stochastic
processes given by
\begin{equation}
\label{eq:lin0}
\begin{split}
  \tilde{Y}^N_t
& =
  e^{ A t } \,
  \xi 
  +
  \int_0^t
  e^{
    A 
    ( t - 
      s
    )
  }
  \,
  F\big( 
    \tilde{Y}^N_{
      \lfloor s \rfloor_N 
    } 
  \big)
  \,
  ds
  +
  \int_0^t
  e^{
    A 
    ( t - 
      s
    )
  }
  \,
  B\big( 
    \tilde{Y}^N_{
      \lfloor s \rfloor_N 
    } 
  \big)
  \,
  dW_s
\end{split}
\end{equation}
$ \mathbb{P} $-a.s.\ for 
all $ t \in [0,T] $
and all $ N \in \mathbb{N} $.
Note that
for each $ N \in \mathbb{N} $
the stochastic
process
$
  \tilde{Y}^N \colon
  [0,T] \times \Omega
  \rightarrow H_{ \gamma }
$
is a mild It\^{o} process
with semigroup 
$ 
  e^{ A (t_2 - t_1) } \in 
  L( 
    H_{ \min( \alpha, \beta, \gamma ) } ,
    H_{ \gamma }
  ) 
$, $ (t_1,t_2) \in \angle $,
with mild drift
\begin{equation}
\label{eq:lin_milddrift}
  F\big( 
    \tilde{Y}^N_{
      \lfloor t \rfloor_N 
    } 
  \big),
\qquad
  t \in [0,T],
\end{equation}
and with mild diffusion
\begin{equation}
\label{eq:lin_milddiffusion}
  B\big( 
    \tilde{Y}^N_{
      \lfloor t \rfloor_N 
    } 
  \big),
\qquad
  t \in [0,T] .
\end{equation}
Proposition~\ref{propsimple}
therefore gives
\begin{equation}
\label{eq:propcons2}
  \tilde{Y}^N_{ t }
 =
  e^{
    A \left( t - t_0 \right)
  }
  \,
  \tilde{Y}^N_{ t_0 } 
  +
  \int_{ t_0 }^{
    t
  }
  e^{
    A \left( t - s \right)
  }
  \,
      F\big( 
        \tilde{Y}^N_{ \lfloor s \rfloor_N } 
      \big) \,
    ds
  +
  \int_{ t_0 }^{
    t
  }
  e^{
    A \left( t - s \right)
  }
  \,
      B\big( 
        \tilde{Y}^N_{ \lfloor s \rfloor_N } 
      \big) \,
    dW_s
\end{equation}
$ \mathbb{P} $-a.s.\ for 
all $ t_0, t \in [0,T] $
with $ t_0 \leq t $ and
this implies
\begin{equation}
\label{eq:lin}
\begin{split}
  \tilde{Y}^N_{ 
    \frac{ (n+1) T }{ N } 
  }
  =
  e^{
    A \frac{ T }{ N }
  }
  \,
  \tilde{Y}^N_{ \frac{ n T }{ N } } 
    +
    \left(
    \int_0^{ \frac{T}{N} } e^{ A s } \, ds
    \right) 
    F\big( 
      \tilde{Y}^N_{ \frac{ n T }{ N } } 
    \big)
  +
  \int_{ 
    \frac{ n T }{ N }
  }^{
      \frac{ ( n + 1 ) T }{ N }
  }
    e^{ 
      A \left(  
        \frac{ (n+1) T }{ N } - s
      \right)
    } \,
    B\big( 
      \tilde{Y}^N_{ \frac{ n T }{ N } } 
    \big) \,
  dW_s
\end{split}
\end{equation}
$ \mathbb{P} $-a.s.\ for 
all 
$ n \in \left\{ 0, 1, \dots, N-1
\right\} $
and all
$ N \in \mathbb{N} $.
In particular,
in the case of additive
noise, i.e., $ B(v) = B(0) $
for all $ v \in H_{ \gamma } $,
equation~\eqref{eq:lin}
reduces to
\begin{equation}
\label{eq:linB}
\begin{split}
&
  \tilde{Y}^N_{ 
    \frac{ (n+1) T }{ N } 
  }
=
  e^{
    A \frac{ T }{ N }
  }
  \,
  \tilde{Y}^N_{ \frac{ n T }{ N } } 
    +
    \left(
    \int_0^{ \frac{T}{N} } e^{ A s } \, ds
    \right) 
    F\big( 
      \tilde{Y}^N_{ \frac{ n T }{ N } } 
    \big)
  +
  \int_{ 
    \frac{ n T }{ N }
  }^{
      \frac{ ( n + 1 ) T }{ N }
  }
    e^{ 
      A \left(  
        \frac{ (n+1) T }{ N } - s
      \right)
    } \,
    B(0) 
    \,
  dW_s
\end{split}
\end{equation}
$ \mathbb{P} $-a.s.\ for 
all 
$ n \in \left\{ 0, 1, \dots, N-1
\right\} $
and all
$ N \in \mathbb{N} $.
The mild It\^{o} processes
$ \tilde{Y}^N $, $ N \in \mathbb{N} $,
are thus nothing else but appropriate
time continuous interpolations
of the numerical approximations
in \cite{jk09b}
in the case of additive noise
(see~(3.3) in \cite{jk09b})
and in \cite{j10}
in the general case 
(see~(21) and (50) 
in \cite{j10}).

Note that the mild 
drift~\eqref{eq:lin_milddrift}
and the 
diffusion~\eqref{eq:lin_milddiffusion}
of the approximation
processes~\eqref{eq:lin0}
do not contain the correction
term
$
  e^{
    A 
    \left( t - 
      \lfloor t \rfloor_N
    \right)
  }
$, $ t \in [0,T] $,
in mild
drift~\eqref{eq:split_milddrift}
and the mild
diffusion~\eqref{eq:split_milddiffusion}
of the exponential Euler
approximations~\eqref{eq:split0}.
The approximation 
processes~\eqref{eq:lin0}
thus seem to be more close
to the exact solution~\eqref{eq:SPDE}
than the exponential Euler
approximations~\eqref{eq:split0}
(compare the mild drifts~\eqref{eq:lin_milddrift},
\eqref{eq:split_milddrift},
\eqref{eq:milddrift}
and the mild 
diffusions~\eqref{eq:lin_milddiffusion},
\eqref{eq:split_milddiffusion},
\eqref{eq:milddiffusion}).
Indeed, under suitable assumptions, 
it has been shown 
(see \cite{jk09b,jkw11} for details)
that
$ \tilde{Y}^N $, $ N \in \mathbb{N} $,
converges
to $ X $ significantly
faster than 
$ Y^N $, 
$ N \in \mathbb{N} $.
This motivates to call approximations
of the form~\eqref{eq:lin}
and \eqref{eq:linB} as
{\it accelerated exponential
Euler approximations}.
The crucial point in
the accelerated exponential Euler
approximations is that
they contain 
the stochastic integrals
\begin{equation}
\label{eq:integrals}
  \int_{ 
    \frac{ n T }{ N }
  }^{
      \frac{ ( n + 1 ) T }{ N }
  }
    e^{ 
      A \left(  
        \frac{ (n+1) T }{ N } - s
      \right)
    } \,
    B\big( 
      \tilde{Y}^N_{ \frac{ n T }{ N } } 
    \big) \,
  dW_s
\end{equation}
for  
$ n \in \left\{ 0, 1, \dots, N-1
\right\} $
and
$ N \in \mathbb{N} $
in the scheme
instead of simply increments
of driving noise process.
This enables them to converge,
under suitable assumptions,
significantly faster to $ X $ than
schemes using only 
increments of the
driving noise process 
such as \eqref{eq:split}.
In addition, 
in the case of additive noise,
the
stochastic integrals~\eqref{eq:integrals} 
in \eqref{eq:linB}
depends linearly on the cylindrical
Wiener process $ W_t $, $ t \in [0,T] $
and are easy to simulate.
Therefore,
the accelerated exponential
Euler approximations~\eqref{eq:linB}
can in the case of
additive noise be simulated 
quite efficiently
(see Section~3 in \cite{jk09b}
and, particularly, 
see Figure~2 in \cite{jkw11}
for details).
Further investigations 
and related results on approximation
methods that make use
of stochastic integrals of the 
form \eqref{eq:integrals}
can, e.g., be found in
\cite{jk09b,jk09d,lt10,lt10b,d10,jkw11,j11,mtac11,xd11}.

\paragraph{Linear
implicit Euler
approximations for SPDEs}
%\label{sec:lineuler}

Next it is shown that
linear implicit Euler approximations
can be formulated as mild It\^{o} 
processes.
For this we assume $ \eta = 0 $ 
in the following
in order to avoid trivial complications.
Moreover, let 
$
  \bar{S}^N \colon 
  \angle
  \rightarrow
  L( H_{ \gamma } )
$,
$ N \in \mathbb{N} $,
be a sequence of mappings
given by
\begin{equation}
\label{eq:lineulersg}
  \bar{S}^N_{ t_1, t_2 }
:=
  \Big(
    I - A 
    \left( t_1 - \lfloor t_1 \rfloor_N \right)
  \Big)
  \Big(
    I - A
    \left( t_2 - \lfloor t_2 \rfloor_N \right)
  \Big)^{ \! -1 }
  \Big(
    I - \tfrac{ T }{ N } A
  \Big)^{ 
    \left(
      \lfloor t_1  
      \rfloor_N
      - 
      \lfloor t_2 \rfloor_N 
    \right)
    \frac{ N }{ T }
  }
\end{equation}
for all 
$ t_1, t_2 \in [0,T] $
with $ t_1 < t_2 $
and
all $ N \in \mathbb{N} $.
Moreover,
let 
$
  \bar{Y}^N
  \colon [0,T]
  \times
  \Omega
  \rightarrow H_{ \gamma }
$,
$
  N \in \mathbb{N} 
$,
be a sequence of predictable
stochastic
processes given by
$
  \bar{Y}^N_0 = \xi
$
and
\begin{equation}
\label{eq:lineuler0}
\begin{split}
  \bar{Y}^N_t
& =
  \bar{S}^N_{ 0, t } \,
  \xi 
  +
  \int_0^t
  \bar{S}_{ \lfloor s \rfloor_N, t }^N
  \,
  F\big( 
    \bar{Y}^N_{
      \lfloor s \rfloor_N 
    } 
  \big)
  \,
  ds
  +
  \int_0^t
  \bar{S}_{ \lfloor s \rfloor_N, t }^N 
  \,
  B\big( 
    \bar{Y}^N_{
      \lfloor s \rfloor_N 
    } 
  \big)
  \,
  dW_s
\\ & =
  \bar{S}^N_{ 0, t } \,
  \xi 
  +
  \int_0^t
  \bar{S}_{ s, t }^N
  \,
  \big(
    I - 
    A 
    \left( s - 
      \lfloor s \rfloor_N
    \right)
  \big)^{ -1 }
  \,
  F\big( 
    \bar{Y}^N_{
      \lfloor s \rfloor_N 
    } 
  \big)
  \,
  ds
\\ & \quad
  +
  \int_0^t
  \bar{S}_{ s, t }^N 
  \,
  \big(
    I - 
    A 
    \left( s - 
      \lfloor s \rfloor_N
    \right)
  \big)^{ -1 }
  \,
  B\big( 
    \bar{Y}^N_{
      \lfloor s \rfloor_N 
    } 
  \big)
  \,
  dW_s
\end{split}
\end{equation}
$ \mathbb{P} $-a.s.\ for 
all $ t \in (0,T] $
and all $ N \in \mathbb{N} $.
Observe that
for each $ N \in \mathbb{N} $
the stochastic
process
$
  \bar{Y}^N \colon
  [0,T] \times \Omega
  \rightarrow H_{ \gamma }
$
is a mild It\^{o} process
with semigroup 
$ \bar{S}^N $,
with mild drift
\begin{equation}
  \mathbbm{1}_{ (0,\infty) }(
    t - \lfloor t \rfloor_N
  ) \;
  \big(
    I - 
    A 
    \left( t - 
      \lfloor t \rfloor_N
    \right)
  \big)^{ -1 }
  \,
  F\big( 
    \bar{Y}^N_{
      \lfloor t \rfloor_N 
    } 
  \big),
\qquad
  t \in [0,T],
\end{equation}
and with mild diffusion
\begin{equation}
  \mathbbm{1}_{ (0,\infty) }(
    t - \lfloor t \rfloor_N
  ) \;
  \big(
    I - 
    A 
    \left( t - 
      \lfloor t \rfloor_N
    \right)
  \big)^{ -1 }
  \,
  B\big( 
    \bar{Y}^N_{
      \lfloor t \rfloor_N 
    } 
  \big),
\qquad
  t \in [0,T] .
\end{equation}
Proposition~\ref{propsimple}
therefore implies
\begin{equation}
\label{eq:lineul}
\begin{split}
  \bar{Y}^N_{ 
    \frac{ (n+1) T }{ N } 
  }
& =
  \Big(
    I - 
    \tfrac{ T }{ N } A
  \Big)^{ \! -1 }
  \left(
    \bar{Y}^N_{ \frac{ n T }{ N } } 
    +
    \frac{ T }{ N } \cdot
    F\big( \bar{Y}^N_{ \frac{ n T }{ N } } \big)
    +
    \int_{ \frac{ n T }{ N } }^{
      \frac{ ( n + 1 ) T }{ N }
    }
      B\big( \bar{Y}^N_{ \frac{ n T }{ N } } \big) \,
    dW_s
  \right)
\end{split}
\end{equation}
$ \mathbb{P} $-a.s.\ for 
all 
$ n \in \left\{ 0, 1, \dots, N-1
\right\} $
and all
$ N \in \mathbb{N} $.
This shows that
the stochastic
processes
$
  \bar{Y}^N
$,
$ N \in \mathbb{N} $,
are nothing else but
appropriate time
continuous interpolations
of linear
implicit Euler approximations
(see, e.g., \cite{g99,h02,h03a,w05b,mrw07,mr07b,k11,cv12}
and the references therein)
for the SPDE~\eqref{eq:SPDE}.
Note that the semigroups~\eqref{eq:lineulersg}
of the linear implicit Euler 
approximations~\eqref{eq:lineuler0}
depend explicitly on both variables
$ t_1 $ and $ t_2 $
with $ 0 \leq t_1 \leq t_2 \leq T $
instead of on the difference
$ t_2 - t_1 $ only
although the semigroup 
$ e^{ A t } $, $ t \in [0,T] $,
of the underlying SPDE~\eqref{eq:SPDE}
depends on one variable only.

\paragraph{Linear
implicit Crank-Nicolson
approximations for SPDEs}
%\label{sec:crank}

Finally, in this paragraph it is 
demonstrated 
that linear implicit Crank-Nicolson
approximations
can be formulated as mild It\^{o} 
processes too.
As in the case of the linear
implicit Euler approximations
we assume
$ \eta = 0 $
in the following in order to avoid 
trivial complications.
Let 
$
  \hat{S}^N \colon 
  \angle
  \rightarrow
  L( H_{ \gamma } )
$,
$ N \in \mathbb{N} $,
be a sequence of mappings
given by
\begin{equation}
\label{eq:lincranksg}
  \hat{S}^N_{ t_1, t_2 }
:=
  \Big(
    I - 
    A
    \tfrac{ 
      \left( t_1 - \lfloor t_1 \rfloor_N \right) 
    }{ 2 }    
  \Big)
  \Big(
    I - 
    A
    \tfrac{ 
      \left( t_2 - \lfloor t_2 \rfloor_N \right) 
    }{ 2 }    
  \Big)^{ \! -1 }
  \Big(
    I - \tfrac{ T }{ 2 N } A
  \Big)^{ 
    \left(
      \lfloor t_1 \rfloor_N 
      - 
      \lfloor t_2 \rfloor_N 
    \right)
    \frac{ N }{ T }
  }
\end{equation}
for all 
$ t_1, t_2 \in [0,T] $
with $ t_1 < t_2 $
and
all $ N \in \mathbb{N} $.
Moreover,
let 
$
  \hat{Y}^N
  \colon [0,T]
  \times
  \Omega
  \rightarrow H_{ \gamma }
$,
$
  N \in \mathbb{N} 
$,
be a sequence of predictable
stochastic
processes given by
$
  \hat{Y}^N_0 = \xi
$
and
\begin{equation}
\label{eq:crank0}
\begin{split}
  \hat{Y}^N_t
& =
 \hat{S}^N_{ 0, t } \,
  \xi 
  +
  \int_0^t
  \hat{S}_{ \lfloor s \rfloor_N, t }^N \,
  \Big(
    \tfrac{ 1 }{ 2 } A \,
      \hat{Y}^N_{
        \lfloor s \rfloor_N 
      } 
    +
    F\big( 
    \hat{Y}^N_{
      \lfloor s \rfloor_N 
    } 
  \big)
  \Big) \,
  ds
  +
  \int_0^t
  \hat{S}_{ \lfloor s \rfloor_N, t }^N 
  \,
  B\big( 
    \hat{Y}^N_{
      \lfloor s \rfloor_N 
    } 
  \big)
  \,
  dW_s
\\ & =
  \hat{S}^N_{ 0, t } \,
  \xi 
  +
  \int_0^t
  \hat{S}_{ s, t }^N
  \left(
    I - 
    A 
    \tfrac{
      \left( s - 
        \lfloor s \rfloor_N
      \right) 
    }{ 2 }
  \right)^{ \! - 1 }
  \Big(
    \tfrac{ 1 }{ 2 } A \,
      \hat{Y}^N_{
        \lfloor s \rfloor_N 
      } 
    +
    F\big( 
    \hat{Y}^N_{
      \lfloor s \rfloor_N 
    } 
  \big)
  \Big) 
  \,
  ds
\\ & \quad
  +
  \int_0^t
  \hat{S}_{ s, t }^N 
  \left(
    I - 
    A 
    \tfrac{
      \left( s - 
        \lfloor s \rfloor_N
      \right) 
    }{ 2 }
  \right)^{ \! - 1 }
  B\big( 
    \hat{Y}^N_{
      \lfloor s \rfloor_N 
    } 
  \big)
  \,
  dW_s
\end{split}
\end{equation}
$ \mathbb{P} $-a.s.\ for 
all $ t \in (0,T] $
and all $ N \in \mathbb{N} $.
Observe 
for every $ N \in \mathbb{N} $
that
the stochastic
process
$
  \hat{Y}^N \colon
  [0,T] \times \Omega
  \rightarrow H_{ \gamma }
$
is a mild It\^{o} process
with semigroup 
$ \hat{S}^N $,
with mild drift
\begin{equation}
  \mathbbm{1}_{ (0,\infty) }(
    t - \lfloor t \rfloor_N
  ) \,
  \left(
    I - 
    A 
    \tfrac{
      \left( t - 
        \lfloor t \rfloor_N
      \right) 
    }{ 2 }
  \right)^{ \! - 1 }
  \Big(
    \tfrac{ 1 }{ 2 } A \,
      \hat{Y}^N_{
        \lfloor t \rfloor_N 
      } 
    +
    F\big( 
    \hat{Y}^N_{
      \lfloor t \rfloor_N 
    } 
  \big)
  \Big) ,
  \;
  t \in [0,T],
\end{equation}
and with mild diffusion
\begin{equation}
  \mathbbm{1}_{ (0,\infty) }(
    t - \lfloor t \rfloor_N
  ) \,
  \left(
    I - 
    A 
    \tfrac{
      \left( t - 
        \lfloor t \rfloor_N
      \right) 
    }{ 2 }
  \right)^{ \! - 1 }
  B\big( 
    \hat{Y}^N_{
      \lfloor t \rfloor_N 
    } 
  \big),
  \;
  t \in [0,T] .
\end{equation}
Proposition~\ref{propsimple}
hence gives
\begin{equation}
\label{eq:lineul}
\begin{split}
  \hat{Y}^N_{ 
    \frac{ (n+1) T }{ N } 
  }
  =
  \Big(
    I - 
    \tfrac{ T }{ 2 N } A
  \Big)^{ \! -1 }
  \left(
    \Big(
      I + 
      \tfrac{ T }{ 2 N } A
    \Big) \,
    \hat{Y}^N_{ \frac{ n T }{ N } } 
    +
    \frac{ T }{ N } \cdot
    F\big( 
      \hat{Y}^N_{ \frac{ n T }{ N } } 
    \big)
    +
    \int_{ \frac{ n T }{ N } }^{
      \frac{ ( n + 1 ) T }{ N }
    }
      B\big( 
        \hat{Y}^N_{ \frac{ n T }{ N } } 
      \big) \,
    dW_s
  \right)
\end{split}
\end{equation}
$ \mathbb{P} $-a.s.\ for 
all 
$ n \in \left\{ 0, 1, \dots, N-1
\right\} $
and all
$ N \in \mathbb{N} $.
This shows that
the stochastic
processes
$
  \hat{Y}^N
$,
$ N \in \mathbb{N} $,
are nothing else but
appropriate time
continuous interpolations
of linear
implicit Crank-Nicolson 
approximations 
(see, e.g., \cite{s99,w05b,brsv08}
and the references therein)
for the 
SPDE~\eqref{eq:SPDE}.

\subsubsection{Milstein type 
approximations for SPDEs}
\label{sec:mil}

The stochastic Taylor expansions
in Subsection~\ref{sec:tay}
can be used to derive higher order
numerical approximation
methods for SPDEs. In the sequel
this is illustrated for Milstein type
approximations for SPDEs
(see \cite{gk96,ks01,ms06,lcp10,jr12,bl11a,bl11b,WangGan2012,b12}).
For these derivations we 
assume that the diffusion term
$ 
  B \colon H_{ \gamma }
  \rightarrow
  HS( U_0, H_{ \beta } )
$
is twice continuously
Fr\'{e}chet differentiable
with globally bounded derivatives
and that $ \beta = \gamma $.
First note, 
in view of the
mild It\^{o} 
formula~\eqref{eq:operatorIto},
that
equation~\eqref{eq:SPDE}
can also be written as
\begin{equation}
\label{eq:SPDE2}
  X_t
=
  e^{ A (t - t_0) } X_{ t_0 }
+
  \int_{ t_0 }^t
  \big(
    L^{ (0) }_{ (t - s) } \text{id}
  \big)( X_s ) \,
  ds
+
  \int_{ t_0 }^t
  \big(
    L^{ (1) }_{ (t - s) } \text{id}
  \big)( X_s ) \,
  dW_s
\end{equation}
$ \mathbb{P} $-a.s.\ for all
$ t_0, t \in [0,T] $
with $ t_0 \leq t $
where 
$ 
  \text{id} 
  = \text{id}_{ H_{ \gamma } }
  \colon H_{ \gamma }
  \rightarrow 
  H_{ \gamma }
$ is the identity on $ H_{ \gamma } $.
Next the mild It\^{o} formula~\eqref{eq:operatorIto}
is applied to the
test function
$
  \big(
    L^{ (1) }_{ (t - s) } \text{id}
  \big)( x ) 
  \in HS( U_0, H_{ \beta } )
$,
$ 
  x \in H_{ \gamma } 
$,
to obtain
\begin{equation}
\label{eq:Bito}
\begin{split}
&
  \big(
    L^{ (1) }_{ (t - s) } \text{id}
  \big)( X_s ) 
\\&
=
  \big(
    L^{ (1) }_{ (t - s) } \text{id}
  \big)\big( 
    e^{ A (s - t_0) } X_{ t_0 } 
  \big)
  + 
  \int_{ t_0 }^s
  \big( 
    L^{ (0) }_{ (s - u) } 
    L^{ (1) }_{ (t - s) }
    \text{id} 
  \big)( X_u )
  \, du
  + 
  \int_{ t_0 }^s
  \big( 
    L^{ (1) }_{ (s - u) } 
    L^{ (1) }_{ (t - s) }
    \text{id} 
  \big)( X_u )
  \, dW_u
\end{split}
\end{equation}
$ \mathbb{P} $-a.s.\ for all
$ t_0, s, t \in [0,T] $ with 
$ t_0 \leq s \leq t $.
Putting \eqref{eq:Bito}
into \eqref{eq:SPDE2}
then gives
\begin{align}
\label{eq:milsteintay}
  X_t
&=
  e^{ A (t - t_0) } X_{ t_0 }
+
  \int_{ t_0 }^t
  \big(
    L^{ (0) }_{ (t - s) } \text{id}
  \big)( X_s ) \,
  ds
+
  \int_{ t_0 }^t
  \big(
    L^{ (1) }_{ (t - s) } \text{id}
  \big)\big( e^{ A (s - t_0) } X_{ t_0 } 
  \big) \,
  dW_s
\nonumber
\\ & 
  +
  \int_{ t_0 }^t
  \int_{ t_0 }^s
  \big( 
    L^{ (0) }_{ (s - u) } 
    L^{ (1) }_{ (t - s) }
    \text{id} 
  \big)( X_u )
  \, du 
  \, dW_s
  + 
  \int_{ t_0 }^t
  \int_{ t_0 }^s
  \big( 
    L^{ (1) }_{ (s - u) } 
    L^{ (1) }_{ (t - s) }
    \text{id} 
  \big)( X_u )
  \, dW_u
  \, dW_s
\\ & =
\nonumber
  e^{ A (t - t_0) } X_{ t_0 }
+
  \int_{ t_0 }^t
  e^{ A (t - s) } F( X_s ) \,
  ds
+
  \int_{ t_0 }^t
  e^{ A (t - s) } 
  B\big( e^{ A (s - t_0) } X_{ t_0 } 
  \big) \,
  dW_s
\\ & 
  +
  \int_{ t_0 }^t
  \int_{ t_0 }^s
  \big( 
    L^{ (0) }_{ (s - u) } 
    L^{ (1) }_{ (t - s) }
    \text{id} 
  \big)( X_u )
  \, du 
  \, dW_s
  + 
  \int_{ t_0 }^t
  \int_{ t_0 }^s
  e^{ A (t - s) }
  B'\big(
    e^{ A (s - u) }
    X_u
  \big) \,
  e^{ A (s - u) }
  B( X_u )
  \, dW_u
  \, dW_s
\nonumber
\end{align}
$ \mathbb{P} $-a.s.\ for
all $ t_0, t \in [0,T] $
with $ t_0 \leq t $.
The identity~\eqref{eq:milsteintay}
corresponds to the strong stochastic
Taylor expansion in
Theorem~\ref{thm:strongtay}
which is
described by the hierarchical
set
$ 
  \mathcal{A} = \{ \emptyset, (1) \}
$;
see Subsection~\ref{sec:tay}
for details.
Using the approximations
$ e^{ A h } \approx I $
and
$ X_{ t_0 + h } \approx X_{ t_0 } $
for small $ h \in [0,T] $
and omitting the 
integral
$
  \int_{ t_0 }^t
  \int_{ t_0 }^s
  \big( 
    L^{ (0) }_{ (s - u) } 
    L^{ (1) }_{ (t - s) }
    \text{id} 
  \big)( X_u )
  \, du 
  \, dW_s
$
in \eqref{eq:milsteintay}
then results in the approximation
\begin{align}
\label{eq:milsteinapprox}
\nonumber
  X_t 
& \approx
  e^{ A (t - t_0) } X_{ t_0 }
+
  \int_{ t_0 }^t
  e^{ A (t - s) } F( X_{ t_0 } ) \,
  ds
+
  \int_{ t_0 }^t
  e^{ A (t - s) } 
  B( X_{ t_0 } 
  ) \,
  dW_s
\\ & \quad
  + 
  \int_{ t_0 }^t
  \int_{ t_0 }^s
  e^{ A (t - s) }
  B'(
    X_{ t_0 }
  ) \,
  e^{ A (s - u) }
  B( X_{ t_0 } )
  \, dW_u
  \, dW_s
\\ & \approx
  e^{ A (t - t_0) } 
  \left( 
    X_{ t_0 }
    +
     F( X_{ t_0 } ) 
     \cdot \left( t - t_0 \right)
    +
    \int_{ t_0 }^t
    B( X_{ t_0 } 
    ) \,
    dW_s
  + 
  \int_{ t_0 }^t
  \int_{ t_0 }^s
  B'(
    X_{ t_0 }
  ) \,
  B( X_{ t_0 } )
  \, dW_u
  \, dW_s
  \right)
\nonumber
\end{align}
for $ t_0, t \in [0,T] $
with small $ t - t_0 \geq 0 $.
The 
approximation~\eqref{eq:milsteinapprox}
can then be used to define exponential
Milstein approximations for SPDEs
(see equations~\eqref{eq:expmil}
and \eqref{eq:expmil2} below
for details).

In the following it is
demonstrated how different
types of Milstein
approximations for SPDEs
can be formulated as
mild It\^{o} processes.
To this end
we assume in the remainder
of this subsection
that 
$ B \colon H_{ \gamma }
\rightarrow HS( U_0, H_{ \beta } ) $
in Assumption~\ref{diffusion}
is once continuously Fr\'{e}chet
differentiable
and that
$ \beta = \gamma $.

\paragraph{Exponential
Milstein approximations 
for SPDEs}

This paragraph formulates
exponential Milstein
approximations as mild It\^{o} 
processes.
To be more precise, let 
$
  Z^N
  \colon [0,T]
  \times
  \Omega
  \rightarrow H_{ \gamma }
$,
$
  N \in \mathbb{N} 
$,
be a sequence of predictable
stochastic
processes given by
\begin{equation}
\label{eq:expmil}
\begin{split}
  Z^N_t
& =
  e^{ A t } \, \xi 
  +
  \int_0^t
  e^{
    A 
    \left( t - 
      \lfloor s \rfloor_N
    \right)
  }
  \,
  F\big( 
    Z^N_{
      \lfloor s \rfloor_N 
    } 
  \big)
  \,
  ds
  +
  \int_0^t
  e^{
    A 
    \left( t - 
      \lfloor s \rfloor_N
    \right)
  }
  \,
  B\big( 
    Z^N_{
      \lfloor s \rfloor_N 
    } 
  \big)
  \,
  dW_s
\\ & \quad +
  \int_0^t
  e^{
    A 
    \left( t - 
      \lfloor s \rfloor_N
    \right)
  }
  \,
  B'\big( 
    Z^N_{
      \lfloor s \rfloor_N 
    } 
  \big)
  \bigg(
    \int_{ \lfloor s \rfloor_N
    }^{ s }
    B\big( 
      Z^N_{
        \lfloor s \rfloor_N 
      } 
    \big) \,
    dW_u
  \bigg)
  \,
  dW_s
\end{split}
\end{equation}
$ \mathbb{P} $-a.s.\ for 
all $ t \in [0,T] $
and all $ N \in \mathbb{N} $.
Note 
for every $ N \in \mathbb{N} $
that the stochastic
process
$
  Z^N \colon
  [0,T] \times \Omega
  \rightarrow H_{ \gamma }
$
is a mild It\^{o} processes
with semigroup 
$ 
  e^{ A (t_2 - t_1) } \in 
  L( H_{ \min( \alpha, \beta, \gamma ) } ,
    H_{ \gamma }
  ) 
$, $ (t_1,t_2) \in \angle $,
with mild drift
\begin{equation}
  e^{
    A 
    \left( t - 
      \lfloor t \rfloor_N
    \right)
  }
  \,
  F\big( 
    Z^N_{
      \lfloor t \rfloor_N 
    } 
  \big),
  \;
  t \in [0,T],
\end{equation}
and with mild diffusion
\begin{equation}
\label{eq:expmil_milddiffusion}
  e^{
    A 
    \left( t - 
      \lfloor t \rfloor_N
    \right)
  }
  \left(
  B\big( 
    Z^N_{
      \lfloor t \rfloor_N 
    } 
  \big)
  +
  B'\big( 
    Z^N_{
      \lfloor t \rfloor_N 
    } 
  \big)
  \bigg(
    \int_{ \lfloor t \rfloor_N
    }^{ t }
    B\big( 
      Z^N_{
        \lfloor t \rfloor_N 
      } 
    \big) \,
    dW_s
  \bigg)
  \right), \;
  t \in [0,T] .
\end{equation}
Proposition~\ref{propsimple}
hence yields
\begin{equation}
\label{eq:expmil2}
\begin{split}
  Z^N_{ 
    \frac{ (n+1) T }{ N } 
  }
& =
  e^{
    A \frac{ T }{ N }
  }
  \,
  \Bigg(
    Z^N_{ \frac{ n T }{ N } } 
    +
    \frac{ T }{ N } \cdot
    F\big( 
      Z^N_{ \frac{ n T }{ N } } 
    \big)
    +
    \int_{ \frac{ n T }{ N } }^{
      \frac{ ( n + 1 ) T }{ N }
    }
      B\big( 
        Z^N_{ \frac{ n T }{ N } } 
      \big) \,
    dW_s
\\ & \quad +
    \int_{ \frac{ n T }{ N } }^{
      \frac{ ( n + 1 ) T }{ N }
    }
      B'\big( 
        Z^N_{ \frac{ n T }{ N } } 
      \big) \bigg( 
        \int_{ \frac{ n T }{ N } }^s 
        B\big( 
          Z^N_{ \frac{ n T }{ N } } 
        \big) \,
        dW_u
      \bigg) \,
    dW_s
  \Bigg)
\end{split}
\end{equation}
$ \mathbb{P} $-a.s.\ for 
all 
$ n \in \left\{ 0, 1, \dots, N-1
\right\} $
and all
$ N \in \mathbb{N} $.
The mild It\^{o} processes
$ Z^N $, $ N \in \mathbb{N} $,
are thus nothing else but appropriate
time continuous interpolations
of exponential Milstein
approximations
(see \cite{ms06,jr12,bl11b}).

\paragraph{Linear implicit
Euler-Milstein approximations 
for SPDEs}

In this paragraph 
linear implicit Euler-Milstein
approximations 
are formulated as mild It\^{o} 
processes.
Let 
$
  \bar{Z}^N
  \colon [0,T]
  \times
  \Omega
  \rightarrow H_{ \gamma }
$,
$
  N \in \mathbb{N} 
$,
be a sequence of predictable
stochastic
processes given by
$
  \bar{Z}^N_0 = \xi
$
and
\begin{equation}
\begin{split}
  \bar{Z}^N_t
& =
  \bar{S}^N_{0, t} \, \xi 
  +
  \int_0^t
  \bar{S}^N_{
    \lfloor s \rfloor_N, t
  }
  \,
  F\big( 
    \bar{Z}^N_{
      \lfloor s \rfloor_N 
    } 
  \big)
  \,
  ds
  +
  \int_0^t
  \bar{S}^N_{
    \lfloor s \rfloor_N, t
  }
  \,
  B\big( 
    \bar{Z}^N_{
      \lfloor s \rfloor_N 
    } 
  \big)
  \,
  dW_s
\\ & \quad +
  \int_0^t
  \bar{S}^N_{
    \lfloor s \rfloor_N, t
  }
  \,
  B'\big( 
    \bar{Z}^N_{
      \lfloor s \rfloor_N 
    } 
  \big)
  \bigg(
    \int_{ \lfloor s \rfloor_N
    }^{ s }
    B\big( 
      \bar{Z}^N_{
        \lfloor s \rfloor_N 
      } 
    \big) \,
    dW_u
  \bigg)
  \,
  dW_s
\end{split}
\end{equation}
$ \mathbb{P} $-a.s.\ for 
all $ t \in (0,T] $
and all $ N \in \mathbb{N} $
(see \eqref{eq:lineulersg} for the
definition of the mappings
$ \bar{S}^N \colon \angle
\rightarrow L( H_{ \gamma } ) $,
$ N \in \mathbb{N} $).
Observe 
for every $ N \in \mathbb{N} $
that the stochastic
process
$
  \bar{Z}^N \colon
  [0,T] \times \Omega
  \rightarrow H_{ \gamma }
$
is a mild It\^{o} process
with semigroup 
$ \bar{S}^N $,
with mild drift
\begin{equation}
  \mathbbm{1}_{ (0,\infty) }(
    t - \lfloor t \rfloor_N
  ) \;
  \bar{S}^N_{ \lfloor t \rfloor_N, t }
  \,
  F\big( 
    \bar{Z}^N_{
      \lfloor t \rfloor_N 
    } 
  \big),
\;\;
  t \in [0,T],
\end{equation}
and with mild diffusion
\begin{equation}
  \mathbbm{1}_{ (0,\infty) }(
    t - \lfloor t \rfloor_N
  ) \;
  \bar{S}^N_{ \lfloor t \rfloor_N, t }
  \left(
  B\big( 
    \bar{Z}^N_{
      \lfloor t \rfloor_N 
    } 
  \big)
  +
  B'\big( 
    \bar{Z}^N_{
      \lfloor t \rfloor_N 
    } 
  \big)
  \bigg(
    \int_{ \lfloor t \rfloor_N
    }^{ t }
    B\big( 
      \bar{Z}^N_{
        \lfloor t \rfloor_N 
      } 
    \big) \,
    dW_s
  \bigg)
  \right), 
\;\;
  t \in [0,T] .
\end{equation}
Proposition~\ref{propsimple}
hence gives
\begin{equation}
\begin{split}
  \bar{Z}^N_{ 
    \frac{ (n+1) T }{ N } 
  }
& =
  \Big(
    I - 
    \tfrac{ T }{ N } A
  \Big)^{ \! -1 }
  \,
  \Bigg(
    \bar{Z}^N_{ \frac{ n T }{ N } } 
    +
    \frac{ T }{ N } \cdot
    F\big( 
      \bar{Z}^N_{ \frac{ n T }{ N } } 
    \big)
    +
    \int_{ \frac{ n T }{ N } }^{
      \frac{ ( n + 1 ) T }{ N }
    }
      B\big( 
        \bar{Z}^N_{ \frac{ n T }{ N } } 
      \big) \,
    dW_s
\\ & \quad +
    \int_{ \frac{ n T }{ N } }^{
      \frac{ ( n + 1 ) T }{ N }
    }
      B'\big( 
        \bar{Z}^N_{ \frac{ n T }{ N } } 
      \big) \bigg( 
        \int_{ \frac{ n T }{ N } }^s 
        B\big( 
          \bar{Z}^N_{ \frac{ n T }{ N } } 
        \big) \,
        dW_u
      \bigg) \,
    dW_s
  \Bigg)
\end{split}
\end{equation}
$ \mathbb{P} $-a.s.\ for 
all 
$ n \in \left\{ 0, 1, \dots, N-1
\right\} $
and all
$ N \in \mathbb{N} $.
The mild It\^{o} processes
$ \bar{Z}^N $, $ N \in \mathbb{N} $,
are thus nothing else but appropriate
time continuous interpolations
of linear implicit Euler-Milstein
approximations
(see \cite{ks01,ms06,bl11b,b12}).

\paragraph{Linear implicit
Crank-Nicolson-Milstein 
Milstein approximations 
for SPDEs}

Finally, let
$
  \hat{Z}^N
  \colon [0,T]
  \times
  \Omega
  \rightarrow H_{ \gamma }
$,
$
  N \in \mathbb{N} 
$,
be a sequence of predictable
stochastic
processes given by
$
  \hat{Z}^N_0 = \xi
$
and
\begin{equation}
\begin{split}
  \hat{Z}^N_t
& =
  \hat{S}^N_{0, t} \, \xi 
  +
  \int_0^t
  \hat{S}^N_{
    \lfloor s \rfloor_N, t
  }
  \,
  \Big(
    \tfrac{ 1 }{ 2 } A \,
      \hat{Z}^N_{
        \lfloor s \rfloor_N 
      } 
    +
    F\big( 
    \hat{Z}^N_{
      \lfloor s \rfloor_N 
    } 
  \big)
  \Big) \,
  ds
  +
  \int_0^t
  \hat{S}^N_{
    \lfloor s \rfloor_N, t
  }
  \,
  B\big( 
    \hat{Z}^N_{
      \lfloor s \rfloor_N 
    } 
  \big)
  \,
  dW_s
\\ & \quad +
  \int_0^t
  \hat{S}^N_{
    \lfloor s \rfloor_N, t
  }
  \,
  B'\big( 
    \hat{Z}^N_{
      \lfloor s \rfloor_N 
    } 
  \big)
  \bigg(
    \int_{ \lfloor s \rfloor_N
    }^{ s }
    B\big( 
      \hat{Z}^N_{
        \lfloor s \rfloor_N 
      } 
    \big) \,
    dW_u
  \bigg)
  \,
  dW_s
\end{split}
\end{equation}
$ \mathbb{P} $-a.s.\ for 
all $ t \in (0,T] $
and all $ N \in \mathbb{N} $
(see \eqref{eq:lincranksg} 
for the definition
the mappings
$ \hat{S}^N \colon
\angle \rightarrow L( H_{ \gamma } ) $,
$ N \in \mathbb{N} $)
and note 
for each $ N \in \mathbb{N} $
that the stochastic
process
$
  \hat{Z}^N \colon
  [0,T] \times \Omega
  \rightarrow H_{ \gamma }
$
is a mild It\^{o} processes
with semigroup 
$ \hat{S}^N $,
with mild drift
\begin{equation}
  \mathbbm{1}_{ (0,\infty) }(
    t - \lfloor t \rfloor_N
  ) \;
  \hat{S}^N_{ \lfloor t \rfloor_N, t }
  \,
  \Big(
    \tfrac{ 1 }{ 2 } A \,
      \hat{Y}^N_{
        \lfloor t \rfloor_N 
      } 
    +
    F\big( 
    \hat{Y}^N_{
      \lfloor t \rfloor_N 
    } 
  \big)
  \Big) ,
\quad
  t \in [0,T],
\end{equation}
and with mild diffusion
\begin{equation}
  \mathbbm{1}_{ (0,\infty) }(
    t - \lfloor t \rfloor_N
  ) \;
  \hat{S}^N_{ \lfloor t \rfloor_N, t }
  \left(
  B\big( 
    \hat{Z}^N_{
      \lfloor t \rfloor_N 
    } 
  \big)
  +
  B'\big( 
    \hat{Z}^N_{
      \lfloor t \rfloor_N 
    } 
  \big)
  \bigg(
    \int_{ \lfloor t \rfloor_N
    }^{ t }
    B\big( 
      \hat{Z}^N_{
        \lfloor t \rfloor_N 
      } 
    \big) \,
    dW_s
  \bigg)
  \right), 
\quad
  t \in [0,T] .
\end{equation}
Proposition~\ref{propsimple}
therefore shows
\begin{equation}
\begin{split}
  \hat{Z}^N_{ 
    \frac{ (n+1) T }{ N } 
  }
& =
  \Big(
    I - 
    \tfrac{ T }{ 2 N } A
  \Big)^{ \! -1 }
  \,
  \Bigg(
    \Big(
      I + 
      \tfrac{ T }{ 2 N } A
    \Big) \,
    \hat{Z}^N_{ \frac{ n T }{ N } } 
    +
    \frac{ T }{ N } \cdot
    F\big( 
      \hat{Z}^N_{ \frac{ n T }{ N } } 
    \big)
\\ & \quad +
    \int_{ \frac{ n T }{ N } }^{
      \frac{ ( n + 1 ) T }{ N }
    }
      B\big( 
        \hat{Z}^N_{ \frac{ n T }{ N } } 
      \big) \,
    dW_s
  +
    \int_{ \frac{ n T }{ N } }^{
      \frac{ ( n + 1 ) T }{ N }
    }
      B'\big( 
        \hat{Z}^N_{ \frac{ n T }{ N } } 
      \big) \bigg( 
        \int_{ \frac{ n T }{ N } }^s 
        B\big( 
          \hat{Z}^N_{ \frac{ n T }{ N } } 
        \big) \,
        dW_u
      \bigg) \,
    dW_s
  \Bigg)
\end{split}
\end{equation}
$ \mathbb{P} $-a.s.\ for 
all 
$ n \in \left\{ 0, 1, \dots, N-1
\right\} $
and all
$ N \in \mathbb{N} $.
The mild It\^{o} processes
$ \hat{Z}^N $, $ N \in \mathbb{N} $,
are thus nothing else but appropriate
time continuous interpolations
of linear implicit 
Crank-Nicolson-Milstein
approximations for SPDEs
(see \cite{ms06,bl11b}).

\bibliographystyle{acm}
\bibliography{../Bib/bibfile}
\end{document}